\documentclass[a4paper,fleqn]{cas-sc}

\usepackage[numbers]{natbib}
\usepackage[utf8]{inputenc} % For accents é, ö, etc.
\usepackage{accents}				% For custom arrows on math
\usepackage{todonotes}
\usepackage{subcaption}
\usepackage{stmaryrd} % Brackets for v_jump
\usepackage{amsthm} % For theorems
\usepackage[makeroom]{cancel} % To cancel terms in eqs
\usepackage{bookmark} % To add bookmarks in pdf
\usepackage{empheq}
\usepackage{ulem}
\newcommand\numberthis{\addtocounter{equation}{1}\tag{\theequation}}

\newcommand{\norm}[1]{\left\lVert#1\right\rVert}
\newcommand{\mat}[1]{\underline{\mathbf{#1}}}
\newcommand\acclrvec[1]{\accentset{\,\leftrightarrow}{#1}}	% define
\newcommand{\blocktensor}[1]{\acclrvec{{\mathbf #1}}}	
\newcommand\blocktensorG[1]{\acclrvec{\boldsymbol #1}}	
\newcommand{\Nabla} {\vec{\nabla}}
\newcommand{\numfluxb}[1]{\hat{\mathbf{#1}} }
\newcommand{\numflux}[1]{\hat{{#1}} }
\newcommand\threeMatrix[1]{\underline{ #1}}				
\def\d{\mathrm{d}}
\newcommand{\bigpartialderiv}[2]{ \frac{\partial {#1}}{\partial {#2} } }

\newcommand\stateG[1]{\boldsymbol #1}			

\newcommand{\noncon}{\stateG{\Upsilon}}	

\newcommand{\entVar}{{\mathbf{v}}}

\newcommand\state[1]{\mathbf{#1}}
\newcommand{\supEuler}{{\mathrm{Euler}}}
\newcommand{\supMHD}{{\mathrm{MHD}}}
\newcommand{\supGLM}{{\mathrm{GLM}}}
\newcommand{\DG}{{\mathrm{DG}}}

\newcommand{\Jan}{\stateG{\Phi}}
\newcommand{\JanVec}{\blocktensorG{\Phi}}
\newcommand{\phiMHD}{\stateG{\phi}^\supMHD}

\newcommand{\phiGLM}{\blocktensorG{\phi}^\supGLM}
\newcommand{\phiGLMs}{\stateG{\phi}^\supGLM}

\newcommand{\numnonconsD}[1]{ #1^{\Diamond} }

\newcommand{\avg}[1]{\left\{\hspace*{-3pt}\left\{#1\right\}\hspace*{-3pt}\right\}}
\newcommand{\jump}[1]{\ensuremath{\left\llbracket #1 \right\rrbracket}}
\newcommand{\numnonconsDxi}[1]{ #1^{1\Diamond} }
\newcommand{\numnonconsDeta}[1]{#1^{2\Diamond} }
\newcommand{\numnonconsDzeta}[1]{ #1^{3\Diamond} }
\newcommand{\numnonconsSxi}[1]{ #1^{1\star} }

\newcommand\eqLGL{\mathrel{\stackrel{\makebox[0pt]{\mbox{\normalfont\tiny LGL}}}{=}}}

\newtheorem{lemma}{Lemma}

\newtheorem{remark}{Remark}

\begin{document}

\let\WriteBookmarks\relax
\def\floatpagepagefraction{1}
\def\textpagefraction{.001}

\shorttitle{Entropy-Stable Gauss Collocation Methods for Ideal MHD}
\shortauthors{Rueda-Ramírez et~al.}
%\begin{frontmatter}

\title [mode = title]{Entropy-Stable Gauss Collocation Methods for Ideal Magneto-Hydrodynamics}

\author[1]{Andrés M. Rueda-Ramírez}[%type=editor,
                        %auid=000,bioid=1,
                        %prefix=Sir,
                        %role=Researcher,
                        orcid=0000-0001-6557-9162]
\cormark[1]
\ead{aruedara@uni-koeln.de}

\credit{Conceptualization, Methodology, Software, Validation, Formal analysis, Data Curation, Writing - Original Draft, Visualization}

\author[2]{Florian J. Hindenlang}
[orcid=0000-0002-0439-249X]
\credit{Methodology, Software, Formal analysis, Writing - Original Draft}

\author[3]{Jesse Chan}
\credit{Conceptualization, Methodology, Formal analysis, Writing - Original Draft}

\author[1,4]{Gregor J. Gassner}
[orcid=0000-0002-1752-1158]
%\ead{ggassner@uni-koeln.de}
\credit{Conceptualization, Methodology, Validation, Formal analysis, Data Curation, Writing - Original Draft}

\address[1]{Department of Mathematics and Computer Science, University of Cologne, Weyertal 86-90, 50931 Cologne, Germany}

\address[2]{Max Planck Institute for Plasma Physics, Boltzmannstraße 2, 85748 Garching, Germany}

\address[3]{Department of Computational and Applied Mathematics, Rice University, 6100 Main St, Houston, TX, 77005}

\address[4]{Center for Data and Simulation Science, University of Cologne, 50931 Cologne, Germany}

\cortext[cor1]{Corresponding author}

\begin{abstract}
In this paper, we present an entropy-stable Gauss collocation discontinuous Galerkin (DG) method on 3D curvilinear meshes for the GLM-MHD equations: the single-fluid magneto-hydrodynamics (MHD) equations with a generalized Lagrange multiplier (GLM) divergence cleaning mechanism.
For the continuous entropy analysis to hold and to ensure Galilean invariance in the divergence cleaning technique, the GLM-MHD system requires the use of non-conservative terms.

Traditionally, entropy-stable DG discretizations have used a collocated nodal variant of the DG method, also known as the discontinuous Galerkin spectral element method (DGSEM) on Legendre-Gauss-Lobatto (LGL) points.
Recently, \citet["Efficient Entropy Stable Gauss Collocation Methods". SIAM (2019)]{chan2019efficient} presented an entropy-stable DGSEM scheme that uses Legendre-Gauss points (instead of LGL points) for conservation laws.
Our main contribution is to extend the discretization technique of \citeauthor{chan2019efficient} to the non-conservative GLM-MHD system.

We provide a numerical verification of the entropy behavior and convergence properties of our novel scheme on 3D curvilinear meshes.
Moreover, we test the robustness and accuracy of our scheme with a magneto-hydrodynamic Kelvin-Helmholtz instability problem.
The numerical experiments suggest that the entropy-stable DGSEM on Gauss points for the GLM-MHD system is more accurate than the LGL counterpart.
\end{abstract}

%\includegraphics{figs/grabs.pdf}

\begin{highlights}
\item We propose a novel entropy-stable discontinuous Galerkin (DG) method on Gauss nodes for the GLM-MHD equations.
\item The entropy-stable Gauss DG discretization shows better accuracy than the Gauss-Lobatto (LGL) discretization
\item We reformulate the existing LGL discretization of the GLM-MHD system to fit a general framework that includes the Gauss and LGL methods.
\end{highlights}

\begin{keywords}
Compressible Magnetohydrodynamics,
Entropy Stability,
Discontinuous Galerkin Spectral Element Methods,
Gauss nodes
% \sep 
\end{keywords}

\maketitle

\section{Introduction}

The ideal magnetohydrodynamics (MHD) equations are a set of partial differential equations (PDEs) that describe the behavior of electrically conducting compressible fluids (plasmas), which find applications in numerous fields, e.g., space physics, plasma physics, astrophysics, and geophysics, among others.
The ideal MHD system is a combination of the compressible Euler equations of fluid dynamics with Maxwell's equations of electromagnetism.
There are two relevant physical constraints to the ideal MHD equations that are not explicitly built into the partial differential equations:
\begin{enumerate}
    \item The divergence-free condition on the magnetic field, $\Nabla \cdot \vec{B} = 0$, which rules out the existence of magnetic monopoles.
    \item The second law of thermodynamics, which states that the thermodynamic entropy of a closed system can only increase or remain constant in time.
\end{enumerate}

Since these constraints are not explicitly built into the MHD equations, numerical discretizations of the MHD system must take additional considerations to guarantee their fulfillment.

There is an extensive collection of techniques in the literature to deal with the divergence-free constraint in numerical discretizations of the ideal MHD system.
In this work, we employ the divergence cleaning method proposed by \citet{Munz2000} and \citet{Dedner2002}, which expands the MHD system with a generalized Lagrange multiplier (GLM), which is advected and damped to minimize the divergence error.
The expanded system of partial differential equations is known as the GLM-MHD system.
Since the GLM technique minimizes (but does not remove completely) the divergence error, the GLM-MHD system includes non-conservative terms, which arise from Maxwell's equations when $\Nabla \cdot \vec{B} \ne 0$.
These non-conservative terms symmetrize the system of PDEs and are necessary to fulfill the second law of thermodynamics with non-vanishing divergence errors \cite{Derigs2018,Chandrashekar2016}.

Discontinuous Galerkin (DG) methods offer a straightforward way to discretize PDEs with arbitrarily high-order accuracy.
We are interested in high-order DG methods because they are robust for advection-dominated problems, readily parallelizable because of their compact stencils \cite{Hindenlang2012}, flexible for complex 3D unstructured grids \cite{Wang2013High,Cockburn2000}, and well suited to perform $h/p$ adaptation  \cite{Riviere2008,Kopriva2002,RuedaRamirez2019a,
RuedaRamirez2019}.

Following the work of \citet{Gassner2013}, \citet{Fisher2013} and \citet{Carpenter2014}, split-form DG methods that fulfill the second law of thermodynamics have been proposed, e.g., for the shallow water equations \cite{wintermeyer2017entropy}, the incompressible Navier-Stokes equations \cite{manzanero2020entropy}, the compressible Navier-Stokes equations \cite{Gassner2018}, multi-phase fluid equations \cite{Renac2019,manzanero2020entropy2}, the GLM-MHD system \cite{Bohm2018}, among others.
These entropy-stable methods rely on a nodal (collocated) variant of the DG method that fulfills the summation by parts (SBP) property, also known as the discontinuous Galerkin spectral element method (DGSEM) on Legendre-Gauss-Lobatto (LGL) points and a so-called flux-differencing representation of the fluxes and non-conservative terms. 
With a careful selection of the numerical fluxes, the LGL-DGSEM scheme becomes provably entropy stable, i.e. consistent with the second law of thermodynamics, which provides additional nonlinear stability.

Recently, entropy-stable and kinetic-energy-preserving collocation DG schemes have also been constructed on Gauss (instead of LGL) quadrature points \cite{chan2019efficient,chan2019discretely,ortleb2016kinetic,ortleb2017kinetic}.
The difference is that Gauss-based operators are not classic SBP operators, but satisfy a generalized SBP property \cite{del2019extension,hicken2016multidimensional}.
The main motivation to do so is that the Gauss quadrature provides better accuracy and robustness than the LGL quadrature.
Of particular interest is the work of \citet{chan2019efficient}, as it introduces an efficient entropy-stable DGSEM scheme that does not need an "all-to-all" coupling of the degrees of freedom for the evaluation of the numerical fluxes.

In this paper, we extend the entropy-stable Gauss collocation method of \citet{chan2019efficient} to the GLM-MHD non-conservative system.
The main contributions of this paper are: (i) we rewrite the entropy-stable Gauss-DGSEM method of \citet{chan2019efficient} in an element-local form, which facilitates its implementation, and recreate the entropy conservation/stability proof deriving the exact entropy production terms; (ii) we develop a novel provably entropy conservative/stable Gauss-DGSEM discretization of the GLM-MHD equations that retains high-order accuracy and obtains more accurate results than its LGL counterpart; (iii) we reformulate the entropy-stable LGL-DGSEM discretization of the GLM-MHD equations proposed by \citet{Bohm2018} using a single volume numerical non-conservative term;  
and (iv) we provide a numerical verification of the methods and apply them in an under-resolved MHD turbulence simulation.

The paper is organized as follows. 
In Section \ref{sec:NotationPhysics}, we briefly describe the notation and introduce the GLM-MHD system.
In Section \ref{sec:ESDGSEMs}, we provide a brief literature review on existing entropy-stable discontinuous Galerkin discretizations on LGL points for conservation laws and for the non-conservative GLM-MHD system.
Next, in Section \ref{sec:Gauss_DGSEM}, we discuss the novel entropy-stable Gauss-DGSEM discretizations. 
First, rewrite the entropy-stable Gauss-DGSEM of \citet{chan2019efficient} in an element-local form.
Next, we introduce the entropy-stable Gauss-DGSEM of the non-conservative GLM-MHD system, and show its extension to three-dimensional unstructured and non-conforming curvilinear meshes.
Finally, the numerical verification and validation of the methods is presented in Section \ref{sec:Results} and the conclusions are presented in Section \ref{sec:Conclusions}.

\section{Notation and Governing Equations} \label{sec:NotationPhysics}

\subsection{Notation} \label{sec:notation}
We adopt the notation of \cite{Bohm2018,Gassner2018,RuedaRamirez2020,Rueda-Ramirez2020} to work with vectors of different nature. 
Spatial vectors are noted with an arrow (e.g. $\vec{x}=(x,y,z) \in \mathbb{R}^3$), state vectors are noted in bold (e.g. $\state{u}=(\rho, \rho \vec{v}, \rho E, \vec{B}, \psi)^T$), and block vectors, which contain a state vector in every spatial direction, are noted as
\begin{equation}
\blocktensor{f} =
\begin{bmatrix}
\mathbf{f}_1 \\ 
\mathbf{f}_2 \\
\mathbf{f}_3
\end{bmatrix}.
\end{equation}

The gradient of a state vector is a block vector,
\begin{equation}
\Nabla \mathbf{q} =
\begin{bmatrix}
\partial_x \mathbf{q} \\
\partial_y \mathbf{q} \\
\partial_z \mathbf{q} 
\end{bmatrix},
\end{equation}
and the gradient of a spatial vector is defined as the transpose of the outer product with the nabla operator,
\begin{equation}
\Nabla \vec{v} := 
\left( \Nabla \otimes \vec{v} \right)^T = 
\left( \Nabla \vec{v}^T \right)^T = 
\mat{L} =
\begin{bmatrix}
\bigpartialderiv{v_1}{x} & \bigpartialderiv{v_1}{y} & \bigpartialderiv{v_1}{z} \\
\bigpartialderiv{v_2}{x} & \bigpartialderiv{v_2}{y} & \bigpartialderiv{v_2}{z} \\
\bigpartialderiv{v_3}{x} & \bigpartialderiv{v_3}{y} & \bigpartialderiv{v_3}{z}
\end{bmatrix},
\end{equation}
where we remark that we note general matrices with an underline.

We define the notation for the jump operator, arithmetic and logarithmic means between a left and right state, $a_L$ and $a_R$, as
\begin{equation}
\jump{a}_{(L,R)} := a_R-a_L, 
~~~~~~~~~ 
\avg{a}_{(L,R)} := \frac{1}{2}(a_L+a_R), 
~~~~~~~~ 
a^{\ln}_{(L,R)} := \jump{a}_{(L,R)}/\jump{\ln(a)}_{(L,R)}.
\label{means}
\end{equation}
A numerically stable procedure to evaluate the logarithmic mean is given in \cite{Ismail2009}.

\subsection{The Ideal GLM-MHD Equations} \label{sec:GLM-MHD}

\subsubsection{The System of Equations}

In this work, we use the variant of the ideal GLM-MHD equations that is consistent with the continuous entropy analysis of \citet{Derigs2018}.
The system of equations reads 
\begin{equation} \label{eq:GLM-MHD}
\partial_t \mathbf{u} 
+ \Nabla \cdot \blocktensor{f}^a (\mathbf{u}) 
+ \noncon(\mathbf{u}, \Nabla \mathbf{u})
= \state{0},
\end{equation}
with the state vector $\state{u} = (\rho, \rho \vec{v}, \rho E, \vec{B}, \psi)^T$, the advective flux $\blocktensor{f}^a$, 
and the non-conservative term $\noncon$.
Here, $\rho$ is the density, $\vec{v} = (v_1, v_2, v_3)^T$ is the velocity, $E$ is the specific total energy, $\vec{B} = (B_1, B_2, B_3)^T$ is the magnetic field, and $\psi$ is the so-called \textit{divergence-correcting field}, a generalized Lagrange multiplier (GLM) that is added to the original MHD system to minimize the magnetic field divergence. 
While these equations do not enforce the divergence-free condition exactly, $\Nabla \cdot \vec{B} = 0$, they evolve towards a divergence-free state \cite{Munz2000,Dedner2002,Derigs2018}.

The advective flux contains Euler, ideal MHD and GLM contributions,
\begin{equation}
\blocktensor{f}^a(\mathbf{u}) = \blocktensor{f}^{a,\supEuler} +\blocktensor{f}^{a,\supMHD}+\blocktensor{f}^{a,\supGLM}=
%% Euler part
%% ----------
\begin{pmatrix} 
\rho \vec{v} \\[0.15cm]
\rho (\vec{v}\, \vec{v}^{\,T}) + p\threeMatrix{I} \\[0.15cm]
\vec{v}\left(\frac{1}{2}\rho \left\|\vec{v}\right\|^2 + \frac{\gamma p}{\gamma -1}\right)  \\[0.15cm]
\threeMatrix{0}\\ \vec{0}\\[0.15cm]
\end{pmatrix} +
%% MHD part
%% --------
\begin{pmatrix} 
\vec{0} \\[0.15cm]
\frac{1}{2 \mu_0} \|\vec{B}\|^2 \threeMatrix{I} - \frac{1}{\mu_0} \vec{B} \vec{B}^T \\[0.15cm]
\frac{1}{\mu_0} \left( \vec{v}\,\|\vec{B}\|^2 - \vec{B}\left(\vec{v}\cdot\vec{B}\right) \right) \\[0.15cm]
\vec{v}\,\vec{B}^T - \vec{B}\,\vec{v}^{\,T} \\ \vec{0}\\[0.15cm]
\end{pmatrix} +
\begin{pmatrix} 
\vec{0} \\[0.15cm]
\threeMatrix{0} \\[0.15cm]
\frac{c_h}{\mu_0} \psi \vec{B} \\[0.15cm]
c_h \psi \threeMatrix{I} \\ c_h \vec{B}\\[0.15cm]
\end{pmatrix},
\label{eq:advective_fluxes}
\end{equation}
where $p$ is the gas pressure, $\threeMatrix{I}$ is the $3\times 3$ identity matrix, $\mu_0$ is the permeability of the medium, and $c_h$ is the \textit{hyperbolic divergence cleaning speed}.

We close the system with the (GLM) calorically perfect gas assumption \cite{Derigs2018},
\begin{equation}
p = (\gamma-1)\left(\rho  E - \frac{1}{2}\rho\left\|\vec{v}\right\|^2 - \frac{1}{2 \mu_0}\|\vec{B}\|^2 - \frac{1}{2 \mu_0}\psi^2\right),
\label{eqofstate}
\end{equation}
where $\gamma$ denotes the heat capacity ratio.

The non-conservative term has two main components, $\noncon = \noncon^\supMHD + \noncon^\supGLM$, with
\begin{align}
\noncon^\supMHD &= (\Nabla \cdot \vec{B}) \phiMHD =  \left(\Nabla \cdot \vec{B}\right) 
\left( 0 \,,\, \mu_0^{-1} \vec{B} \,,\, \mu_0^{-1} \vec{v}\cdot\vec{B} \,,\,  \vec{v} \,    ,\, 0 \right)^T \,, \label{Powell}\\ 
\noncon^\supGLM &= \phiGLM \cdot \Nabla \psi \quad =   \phiGLMs_1 \,\frac{\partial \psi}{\partial x} + \phiGLMs_2 \frac{\partial \psi}{\partial y} + \phiGLMs_3 \frac{\partial \psi}{\partial z} \,,\label{NC_GLM}
\end{align}
where $\phiGLM$ is a block vector with
\begin{equation}\label{Galilean}
\phiGLMs_\ell = \mu_0^{-1} \left(0 \,,\, 0\,,\,0\,,\,0\,,\,  v_\ell \psi \,,\, 0\,,\,0\,,\,0\,,\, v_\ell \right)^T, \quad \ell = 1,2,3.
\end{equation} 
The first non-conservative term, $\noncon^\supMHD$, is the well-known Powell term \cite{Powell2001}, and the second non-conservative term, $\noncon^\supGLM$, results from Galilean invariance of the full GLM-MHD system \cite{Derigs2017}. 

We note that for a magnetic field with vanishing divergence, $\Nabla \cdot \vec{B} = 0$, equation \eqref{eq:GLM-MHD} reduces to the conservative ideal MHD equations, which describe the conservation of mass, momentum, energy, and magnetic flux.

\subsubsection{Thermodynamic Properties of the System}

Assuming positive density and pressure, $\rho,p>0$, we obtain a strictly convex mathematical entropy function for the ideal GLM-MHD equations,
\begin{equation}
S(\state{u}) = - \frac{\rho s}{\gamma-1},
\label{entropy}
\end{equation}
where $s = \ln\left(p \rho^{-\gamma}\right)$ is the thermodynamic entropy.
From the entropy function, we define the entropy variables,
\begin{equation}
\entVar = \frac{\partial S}{\partial \state{u}} = \left(\frac{\gamma-s}{\gamma-1} - \beta \norm{\vec{v}}^2,~2\beta v_1,~2\beta v_2,~2\beta v_3,~-2\beta,~2\beta B_1,~2\beta B_2,~2\beta B_3,~2\beta \psi\right)^T,
\label{entrvars}
\end{equation}
with $\beta = \frac{\rho}{2p}$, a quantity that is proportional to the inverse temperature.

\citet{Derigs2017} showed that if we contract \eqref{eq:GLM-MHD} with the entropy variables, we obtain an entropy conservation law,
\begin{equation}\label{eq:EntropyConservationLaw}
 \frac{\partial S}{\partial t} + \Nabla \cdot \vec{f}^{\,S} = 0,
\end{equation}
if the solution is smooth, and an entropy inequality,
\begin{equation}\label{eq:EntropyInequality}
 \frac{\partial S}{\partial t} + \Nabla \cdot \vec{f}^{\,S} \le 0,
\end{equation}
if the solution contains discontinuities.
Here, $\vec{f}^{\,S} = \vec{v} S$ is the so-called entropy flux.

Finally, we define the entropy flux potential to be
\begin{equation}\label{eq:entPotential}
\vec{\Psi} := \entVar^T \blocktensor{f}^a -\vec{f}^S +\theta \vec{B},
\end{equation}
where $\theta$ is the contraction of $\phiMHD$ from the Powell term \eqref{Powell} into entropy space,
\begin{equation}\label{eq:contractSource}
\theta = \entVar^T \phiMHD = 2\beta(\vec{v}\cdot\vec{B}). 
\end{equation}

\subsubsection{One-Dimensional MHD System}

To simplify the analysis of the GLM-MHD system, we write a one-dimensional version of \eqref{eq:GLM-MHD},
\begin{equation} \label{eq:GLM-MHD_1D}
\bigpartialderiv {\mathbf{u}}{t}
+ \bigpartialderiv {\state{f}^a} {x}
+ \noncon_1
= \state{0},
\end{equation}
where the state variable is $\state{u} = (\rho, \rho \vec{v}, \rho E, \vec{B}, \psi)^T$, as before, and the advective flux in $x$ is
\begin{equation}
\state{f}^a(\mathbf{u}) :=
\state{f}_1^{a,\supEuler} +\state{f}_1^{a,\supMHD}+\state{f}_1^{a,\supGLM}=
%% New Euler part
%% --------------
\begin{pmatrix} 
\rho v_1 \\[0.15cm]
\rho v_1^2 + p \\[0.15cm]
\rho v_1 v_2 \\[0.15cm]
\rho v_1 v_3 \\[0.15cm]
v_1 \left( \frac{1}{2} \rho \|\vec{v}\|^2 + \frac{\gamma p}{\gamma -1}\right) \\[0.15cm]
0 \\[0.15cm]
0 \\[0.15cm]
0 \\[0.15cm]
0 \\[0.15cm]
\end{pmatrix} +
%% MHD part
%% --------
\begin{pmatrix} 
0 \\[0.15cm]
\frac{1}{\mu_0} \left( \frac{1}{2} \|\vec{B}\|^2 - B_1 B_1 \right) \\[0.15cm]
- B_1 B_2 / \mu_0 \\[0.15cm]
- B_1 B_3 / \mu_0 \\[0.15cm]
\frac{1}{\mu_0} \left( v_1 \|\vec{B}\|^2 - B_1 \left(\vec{v}\cdot\vec{B}\right) \right) \\[0.15cm]
0 \\[0.15cm]
v_1 B_2 - v_2 B_1 \\[0.15cm]
v_1 B_3 - v_3 B_1 \\[0.15cm]
0 \\[0.15cm]
\end{pmatrix} +
%% GLM part
%% --------
\begin{pmatrix} 
0 \\[0.15cm]
0 \\[0.15cm]
0 \\[0.15cm]
0 \\[0.15cm]
\frac{c_h}{\mu_0} \psi B_1 \\[0.15cm]
c_h \psi  \\[0.15cm]
0 \\[0.15cm]
0 \\[0.15cm]
c_h B_1 \\[0.15cm]
\end{pmatrix}.
\label{eq:advective_fluxes_1D}
\end{equation}

In \eqref{eq:advective_fluxes_1D}, we dropped the sub-index in the conservative fluxes to simplify the notation and improve the readability, $\state{f} \leftarrow \state{f}_1$, as this change does not produce ambiguity between the 1D and 3D notations.
The non-conservative term, $\noncon_1 = \noncon^{\supMHD}_1 + \noncon^{\supGLM}_1$, consists of the following two terms
\begin{equation}
\noncon^{\supMHD}_1 = \bigpartialderiv {B_1}{x} \phiMHD = \bigpartialderiv {B_1}{x}
\begin{pmatrix}
0 \\ \mu_0^{-1} B_1 \\ \mu_0^{-1} B_2 \\ \mu_0^{-1} B_3 \\ \mu_0^{-1} \vec{v} \cdot \vec{B} \\ v_1 \\ v_2 \\ v_3 \\  0
\end{pmatrix},
\ \ \ \ \
\noncon^{\supGLM}_1 = \bigpartialderiv {\psi}{x} \phiGLMs_1 = \bigpartialderiv {\psi}{x}
\begin{pmatrix}
0 \\ 0 \\ 0 \\ 0 \\ \mu_0^{-1} v_1 \psi \\ 0 \\ 0 \\ 0 \\ \mu_0^{-1} v_1
\end{pmatrix}.
\label{eq:NonCons_1D}
\end{equation}

Finally, the entropy flux potential in 1D is defined as
\begin{equation}\label{eq:entPotential_1D}
\Psi := \entVar^T \state{f}^a - f^S +\theta B_1.
\end{equation}

\section{Entropy-Stable LGL-DGSEM Discretizations} \label{sec:ESDGSEMs}

In this section, we present a brief literature review of entropy-stable nodal DG discretizations on LGL nodes for general conservation laws and the GLM-MHD equations.
For compactness and readability, we present the one-dimensional forms of these methods. 
The three-dimensional form of the discretizations for curvilinear meshes can be found in the references listed in the following sections.

%%%%%%%%%%%%%%%%%%%%%%%%%%%%%%%%%%%%%%%%%%%%%%%%%%%%%%

\subsection{Entropy-Stable LGL-DGSEM for Conservation Laws}

To obtain an entropy-stable LGL-DGSEM discretization of a conservation law,
\begin{equation} \label{eq:consLaw}
\bigpartialderiv {\mathbf{u}}{t}
+ \bigpartialderiv {\state{f}^a} {x}
= \state{0},
\end{equation}
the simulation domain is tessellated into elements and all variables are approximated within each element by piece-wise Lagrange interpolating polynomials of degree $N$ on Legendre-Gauss-Lobatto (LGL) nodes.
These polynomials are continuous in each element and may be discontinuous at the element interfaces.
Furthermore, \eqref{eq:consLaw} is multiplied by an arbitrary polynomial (test function) of degree $N$ and numerically integrated by parts inside each element of the mesh with an LGL quadrature rule of $N+1$ points on a reference element, $\xi \in [-1,1]$, to obtain \cite{Fisher2013,Carpenter2014,Gassner2016}
\begin{align} \label{eq:LGL_DGSEM_cons}
J \omega_j \dot{\state{u}}_j 
+ 2 \sum_{k=0}^N Q_{jk} \state{f}^{*}_{(j,k)} 
+ \delta_{j0} \left( \state{f}^a_0 - \numfluxb{f}^a_{(0,L)} \right)
- \delta_{jN} \left( \state{f}^a_N - \numfluxb{f}^a_{(N,R)}\right)
= \state{0}
\end{align}
for each degree of freedom $j$ of each element. 
In \eqref{eq:LGL_DGSEM_cons}, $\omega_j$ is the reference-space quadrature weight, $J$ is the geometry mapping Jacobian from reference space to physical space, which is constant within each element in the 1D discretization, $Q_{jk}=\omega_j D_{jk}=\omega_j \ell'_k(\xi_j)$ is the SBP derivative matrix, defined in terms of the Lagrange interpolating polynomials, $\{ \ell_k \}_{k=0}^N$, $\state{f}^{*}_{(j,k)} = \state{f}^{*}(\state{u}_j,\state{u}_k)$ is the volume numerical two-point flux, and $\numfluxb{f}^a_{(i,j)} = \numfluxb{f}^a(\state{u}_i,\state{u}_j)$ is the surface numerical flux, which accounts for the jumps of the solution across the cell interfaces.
The sub-indices $(0,L)$ and $(N,R)$ indicate that the numerical flux is computed between a boundary node and an outer state (left or right).

For compactness and ease of implementation, \eqref{eq:LGL_DGSEM_cons} can be rewritten as
\begin{align} \label{eq:LGL_DGSEM_cons2}
\boxed{
J \omega_j \dot{\state{u}}_j 
+ \sum_{k=0}^N S_{jk} \state{f}^{*}_{(j,k)} 
- \delta_{0N} \numfluxb{f}^a_{(0,L)}
+ \delta_{jN} \numfluxb{f}^a_{(N,R)}
= \state{0},
}
\end{align}
where $\mat{S}$ is a skew-symmetric matrix defined as
\begin{equation}
\mat{S} = 2 \mat{Q} - \mat{B},
\end{equation}
and 
\begin{equation} \label{eq:Bmat_LGL}
    \mat{B} := diag(-1, 0,  \ldots, 0, 1).
\end{equation}

\begin{lemma} \label{lemma:Entropy_LGL_DGSEM_cons}
The semi-discrete entropy balance of the LGL-DGSEM discretization of conservation laws, \eqref{eq:LGL_DGSEM_cons2}, integrating over an entire element, reads
\begin{equation} \label{eq:Entropy_LGL_DGSEM_cons}
\sum_{j=0}^N \omega_j J \dot{S}_j = 
\numflux{f}^S_{(0,L)} -\numflux{f}^S_{(N,R)} + \frac{1}{2} \left( \hat{r}_{(L,0)} + \hat{r}_{(N,R)} \right) + \sum_{j,k=0}^N Q_{jk} r_{(j,k)},
\end{equation}
where the numerical entropy flux is defined as
\begin{align} 
\numflux{f}^S_{(j,k)} &= 
\avg{\entVar}_{(j,k)}^T \numfluxb{f}^a_{(j,k)} 
- \avg{\Psi}_{(j,k)},
\label{eq:numEntFlux_cons} 
\end{align}
the surface and volumetric entropy production terms are defined respectively as
\begin{align} 
\hat{r}_{(j,k)} &=
\jump{\entVar}_{(j,k)}^T 
\numfluxb{f}^a_{(j,k)}
- \jump{\Psi}_{(j,k)},
\label{eq:EntProd_cons} \\
r_{(j,k)} &=
\jump{\entVar}_{(j,k)}^T 
\state{f}^*_{(j,k)}
- \jump{\Psi}_{(j,k)}.
\label{eq:EntProdVol_cons}
\end{align}
\end{lemma}

\begin{proof}
The proof of entropy conservation/stability for the LGL-DGSEM discretization of conservation laws was originally presented by \citet{Gassner2013}, \citet{Fisher2013a} and \citet{Carpenter2014}. Moreover, a proof in a very similar notation to the one used in this paper can be found in \cite{Hennemann2020}.
\end{proof}

Due to Lemma \ref{lemma:Entropy_LGL_DGSEM_cons}, it is possible to control the entropy behavior of the DGSEM discretization by selecting the volume and surface numerical fluxes.
If an entropy conservative flux is used for both the volume and surface numerical fluxes, the scheme is provably entropy conservative in its advective terms.
Furthermore, if an entropy conservative flux is used for the volume numerical fluxes and an entropy-stable flux is used for the surface numerical fluxes, the resulting scheme is provably entropy stable.

\begin{remark}
The volume numerical flux of \eqref{eq:LGL_DGSEM_cons2} is consistent with the surface numerical flux because, when evaluated on a single point, both fluxes converge to the physical flux,
\begin{equation} \label{eq:consistentProp}
\numfluxb{f}^a (\state{u}_j,\state{u}_j) = 
\state{f}^{*} (\state{u}_j,\state{u}_j) =  \state{f}^a ( \state{u}_j).
\end{equation} 
Moreover, the volume numerical flux of \eqref{eq:LGL_DGSEM_cons2} is interchangeable with the surface numerical flux in the entropy conservative scheme.
Any combination of entropy conservative fluxes can be selected for the volume and surface numerical fluxes, and the resulting scheme is entropy conservative.
As a matter of fact, a common practice to obtain a provably entropy-stable scheme is to use the same entropy conservative flux in the volume and the surface, and to add a dissipation term in the surface flux.
\end{remark}

%%%%%%%%%%%%%%%%%%%%%%%%%%%%%%%%%%%%%%%%%%%%%%%%%%%%%%

\subsection{Entropy-Stable LGL-DGSEM for the Non-Conservative GLM-MHD System}

\citet{Bohm2018} proposed an entropy-stable LGL-DGSEM discretization of the resistive GLM-MHD equations that reads as
\begin{empheq}[box=\fbox]{align} \label{eq:LGL_DGSEM_GLMMHD}
J \omega_j \dot{\state{u}}_j 
+2 &\sum_{k=0}^N Q_{jk} \state{f}^{*}_{(j,k)} 
+ \delta_{j0} \left( \state{f}^a_0 - \numfluxb{f}^a_{(0,L)}  \right)
- \delta_{jN} \left( \state{f}^a_N - \numfluxb{f}^a_{(N,R)} \right)
\nonumber \\
+  &\sum_{k=0}^N Q_{jk} \Jan^{*}_{(j,k)} 
+ \delta_{j0} \left( \Jan_0 - \numnonconsD{\Jan}_{(0,L)} \right)
- \delta_{jN} \left( \Jan_N - \numnonconsD{\Jan}_{(N,R)} \right)
= \state{0},
\end{empheq}
where $\Jan$ is a non-derivative version of the non-conservative terms,
\begin{equation}
\Jan := \phiMHD B_1 + \phiGLMs_1 \psi,
\end{equation}
$\Jan^{*}_{(j,k)}$ is the so-called volume numerical non-conservative term, and $\numnonconsD{\Jan}_{i,j}$ is the so-called surface numerical non-conservative term.
The surface numerical non-conservative term is defined as \cite{Derigs2018,Bohm2018}
\begin{align} \label{eq:DiamondFluxes}
\numnonconsD{\Jan}_{(j,j+1)} &= 
\numnonconsD{ \left( \phiMHD B_1 \right) }_{(j,j+1)} 
+
\numnonconsD{ \left( \phiGLMs_1 \psi \right) }_{(j,j+1)} 
\nonumber \\
&=
\avg{B_1}_{(j,j+1)} \phiMHD_j + \avg{\psi}_{(j,j+1)} \phiGLMs_{1,j}.
\end{align}
and the volume numerical non-conservative term is defined as \cite{Bohm2018}
\begin{align} \label{eq:volNonCons_D1}
\Jan^{*}_{(j,k)} &= \noncon^{*\supMHD}_{(j,k)} + \noncon^{*\supGLM}_{(j,k)} \nonumber \\
&= \phiMHD_j B_{1,k} + \phiGLMs_{1,j} \psi_k.
\end{align}

\begin{remark}
The volume numerical non-conservative term of \eqref{eq:LGL_DGSEM_GLMMHD} is consistent with the surface numerical non-conservative term because, when evaluated on a single point, both are identical,
\begin{equation} \label{eq:consistentProp_noncons}
\numnonconsD{\Jan} (\state{u}_j,\state{u}_j) = 
\Jan^{*} (\state{u}_j,\state{u}_j) =  
\Jan ( \state{u}_j).
\end{equation} 
However, these two terms are not interchangeable as they have a very dissimilar algebraic form.
\end{remark}

\begin{lemma} \label{lemma:Entropy_LGL_DGSEM_MHD}
The semi-discrete entropy balance of the LGL-DGSEM discretization of the non-conservative GLM-MHD system, \eqref{eq:LGL_DGSEM_GLMMHD}, integrating over an entire element, reads
\begin{equation} \label{eq:Entropy_LGL_DGSEM_MHD}
\sum_{j=0}^N \omega_j J \dot{S}_j = 
\numflux{f}^S_{(0,L)} -\numflux{f}^S_{(N,R)} + \frac{1}{2} \left( \hat{r}_{(L,0)} + \hat{r}_{(N,R)} \right) + \sum_{j,k=0}^N Q_{jk} r_{(j,k)}.
\end{equation}
where the numerical entropy flux and the entropy production on the surface and volume are defined respectively as
\begin{align}
\numflux{f}^S_{(j,k)} &= 
\avg{\entVar}_{(j,k)}^T \numfluxb{f}^a_{(j,k)} 
+ \frac{1}{2} \entVar^T_j \numnonconsD{\stateG{\Phi}}_{(j,k)}
+ \frac{1}{2} \entVar^T_{k} \numnonconsD{\stateG{\Phi}}_{(k,j)}
- \avg{\Psi}_{(j,k)},
\label{eq:numEntFlux_MHD} 
\\
\hat{r}_{(j,k)} &=
\jump{\entVar}_{(j,k)}^T 
\numfluxb{f}^a_{(j,k)}
+ \entVar^T_{k} \numnonconsD{\stateG{\Phi}}_{(k,j)}
- \entVar^T_{j}   \numnonconsD{\stateG{\Phi}}_{(j,k)}
- \jump{\Psi}_{(j,k)}.
\label{eq:EntProd_MHD} 
\\
r_{(j,k)} &=
\jump{\entVar}_{(j,k)}^T 
\state{f}^*_{(j,k)}
+ \entVar^T_{k} \numnonconsD{\stateG{\Phi}}_{(k,j)}
- \entVar^T_{j}   \numnonconsD{\stateG{\Phi}}_{(j,k)}
- \jump{\Psi}_{(j,k)}.
\label{eq:EntProdVol_MHD} 
\end{align}
\end{lemma}

\begin{proof}
The proof of entropy conservation/stability for the LGL-DGSEM was originally presented by \citet{Bohm2018}, and can also be found in \cite{Rueda-Ramirez2020}, where the exact same notation of this paper is used.
\end{proof}

%%%%%%%%%%%%%%%%%%%%%%%%%%%%%%%%%%%%%%%%%%%%%%%%%%%%%%
% NOVEL METHODS: Gauss ES-DGSEM
%%%%%%%%%%%%%%%%%%%%%%%%%%%%%%%%%%%%%%%%%%%%%%%%%%%%%%
\section{Novel Entropy-Stable Gauss-DGSEM Discretizations} \label{sec:Gauss_DGSEM}

In this section, we present entropy-stable Gauss-DGSEM discretizations.
In the first part,  we rewrite the scheme for conservation laws proposed by \citet{chan2019efficient} in an element-local form in one spatial dimension. 
In the second part, we present the novel Gauss-DGSEM scheme for the non-conservative GLM-MHD system in one dimension and show its extension to 3D and curvilinear meshes.

%%%%%%%%%%%%%%%%%%%%%%%%%%%%%%%%%%%%%%%%%%%%%%%%%%%%%%

\subsection{Entropy-Stable Gauss-DGSEM for Conservation Laws}

We focus on the entropy-stable Gauss-DGSEM discretization of conservation laws proposed by \citet{chan2019efficient}, which is based on the generalized summation by parts property,
\begin{equation} \label{eq:genSBP}
\mat{Q} + \mat{Q}^T = \hat{\mat{B}},
\end{equation}
where the generalized boundary matrix is defined as
\begin{equation}
\hat{\mat{B}} = 
%\underline{\mathbf{V}_f^T} \mat{B}^0 \underline{\mathbf{V}_f}
\mat{V}_f^T \mat{B}^0 \mat{V}_f,
\end{equation}
$\mat{V}_f$ is a matrix that interpolates polynomials at Gauss nodes to values at the element boundaries, and $\mat{B}^0$ is a $2 \times 2$ boundary matrix:
\begin{equation}
\mat{V}_f = 
\begin{bmatrix}
\ell_0(-1) & \ell_1(-1) & \cdots & \ell_N(-1) \\
\ell_0(+1) & \ell_1(+1) & \cdots & \ell_N(+1)
\end{bmatrix}
, ~~~~~~
\mat{B}^0 = 
\begin{bmatrix}
-1 & 0 \\
 0 & 1
\end{bmatrix}.
\end{equation}
As a result, the components of matrix $\hat{\mat{B}}$ can be written explicitly as
\begin{equation}
    \hat{B}_{jk} = \ell_j(+1) \ell_k(+1) - \ell_j(-1) \ell_k(-1).
\end{equation}
Matrices $\mat{Q}$ and $\hat{\mat{B}}$ have the following generalized SBP properties:
\begin{equation} \label{eq:genSBPprop}
\sum_{k=0}^N Q_{jk} = 0,
~~~~~~~~~~
\sum_{k=0}^N \hat{B}_{jk} = \ell_j (+1) - \ell_j (-1),
~~~~~~~~~~
\sum_{k=0}^N \hat{B}_{jk} u_k = \ell_j (+1) u_R - \ell_j (-1) u_L,
~~~~~~~~~~
\forall j,
\end{equation}
for any polynomial $u \in \mathbb{P}^N$ with nodal values $u_k$ and boundary values $u_L$ and $u_R$.

The entropy-stable Gauss-DGSEM discretization of \citet{chan2019efficient} can be written in an element-local fashion as
\begin{empheq}[box=\fbox]{align} \label{eq:Gauss_DGSEM_cons}
J \omega_j \dot{\state{u}}_j 
+ \sum_{k=0}^N \hat{S}_{jk} \state{f}^{*}_{(j,k)} 
-& \ell_{j}(-1) \left[
   \state{f}^* \left( \state{u}_j , \tilde{\state{u}}_L \right)
 - \sum_{k=0}^N \ell_{k}(-1) \, \state{f}^* \left( \tilde{\state{u}}_L , \state{u}_k \right)
 + \numfluxb{f}^a \left( \tilde{\state{u}}_L, \tilde{\state{u}}^{+}_L \right)
 \right]
\nonumber\\
+& \ell_{j}(+1) \left[
 \state{f}^* \left( \state{u}_j , \tilde{\state{u}}_R \right)
 - \sum_{k=0}^N \ell_{k}(+1) \, \state{f}^* \left( \tilde{\state{u}}_R , \state{u}_k \right)
 + \numfluxb{f}^a \left( \tilde{\state{u}}_R, \tilde{\state{u}}^{+}_R \right)
\right]
= \state{0},
\end{empheq}
where $\hat{\mat{S}}$ is a skew-symmetric matrix defined as
\begin{equation} \label{eq:Shat_mat}
\hat{\mat{S}} = 2 \mat{Q} - \hat{\mat{B}},
\end{equation}
$\tilde{\state{u}}_L$ and $\tilde{\state{u}}_R$ denote the so-called \textit{entropy-projected} solution at the element boundaries,
\begin{equation}
\tilde{\state{u}}_L := \state{u} \left( \sum_{i=0}^N \ell_i(-1) \entVar (\state{u}_i) \right), 
~~~~~~
\tilde{\state{u}}_R := \state{u} \left( \sum_{i=0}^N \ell_i(+1) \entVar (\state{u}_i) \right),
\end{equation}
where the operator $\state{u}(\cdot)$ computes the state variables from a set of entropy variables, and the operator $\entVar(\cdot)$ computes the entropy variables from a set of state variables, and $\tilde{\state{u}}_L^+$ and $\tilde{\state{u}}_R^+$ denote the external states at the element boundaries, which can be defined by a boundary condition or a neighbor element's entropy-projected solution.

Note that \eqref{eq:genSBP} is valid for both Gauss and LGL discretizations.
Hence, \eqref{eq:Gauss_DGSEM_cons} can be easily shown to be equivalent to \eqref{eq:LGL_DGSEM_cons2} when LGL nodes are used.

\begin{lemma} \label{lemma:Entropy_Gauss_DGSEM_cons}
The semi-discrete entropy balance of the Gauss-DGSEM discretization of conservation laws, \eqref{eq:Gauss_DGSEM_cons}, integrating over an entire element, reads
\begin{align} \label{eq:Entropy_Gauss_DGSEM_cons}
\sum_{j=0}^N \omega_j J \dot{S}_j = 
\numflux{f}^S \left( \tilde{\state{u}}_L, \tilde{\state{u}}^{+}_L \right)
-\numflux{f}^S \left( \tilde{\state{u}}_R, \tilde{\state{u}}^{+}_R \right) 
&+ \frac{1}{2} \left[
	\hat{r} \left( \tilde{\state{u}}_L, \tilde{\state{u}}^{+}_L \right)
  + \hat{r} \left( \tilde{\state{u}}_R, \tilde{\state{u}}^{+}_R \right)
\right]
\nonumber \\
&+ \frac{1}{2} \sum_{j,k=0}^N \hat{S}_{jk} r_{(j,k)}
+ \sum_{j=0}^N \left( \ell_{j}(+1) \tilde{r}_{(j,R)} - \ell_{j}(-1) \tilde{r}_{(j,L)} \right),
\end{align}
where the last two sums denote the contribution of volumetric entropy production terms \eqref{eq:EntProdVol_cons}, and we use the short-hand notation $\tilde{r}_{(j,L)}$ and $\tilde{r}_{(j,R)}$ for the volumetric entropy production terms between the solution at node $j$ and the entropy-projected solution at the left and right boundaries, respectively, i.e., 
\begin{equation} \label{eq:EntProdTilde_cons}
    \tilde{r}_{(j,L)} := r(\state{u}_j,\tilde{\state{u}}_L),
    \,\,\,\,
    \tilde{r}_{(j,R)} :=
    r(\state{u}_j,\tilde{\state{u}}_R).
\end{equation}
The surface entropy production term is defined in \eqref{eq:EntProd_cons} and the numerical entropy flux defined in \eqref{eq:numEntFlux_cons}.
Recall that \eqref{eq:numEntFlux_cons} is a symmetric term, which implies entropy conservation when the volume and surface entropy production vanishes.
\end{lemma}

\begin{proof}
The proof of entropy conservation/stability of \eqref{eq:Gauss_DGSEM_cons} was originally presented in \cite{chan2019efficient} using so-called \textit{block hybridized} matrix operators.
However, we recreate the one-dimensional proof in our own notation for completeness, where we explicitly derive the different surface and volume entropy production terms.

The semi-discrete entropy balance is obtained by contracting \eqref{eq:Gauss_DGSEM_cons} with the entropy variables and integrating over an element,
\begin{align}
\sum_{j=0}^N \omega_j J \dot{S}_j 
  =& - 
\underbrace{
	\sum_{j=0}^{N} \entVar^T_j 
	\left[
		\sum_{k=0}^N \hat{S}_{jk} \state{f}^{*}_{(j,k)}
	\right]
}_{(a):= \text{ volume term}} 
\nonumber\\
&+
\underbrace{
	\sum_{j=0}^N \entVar_j^T \ell_{j}(-1) 
	\left[
		\state{f}^* \left( \state{u}_j , \tilde{\state{u}}_L \right)
		- \sum_{k=0}^N \ell_{k}(-1) \, \state{f}^* \left( \tilde{\state{u}}_L , \state{u}_k \right)
	\right]
}_{(b):= \text{ new terms (left)}} 
+ \underbrace{ 
	\sum_{j=0}^N \entVar_j^T \ell_{j}(-1) \numfluxb{f}^a \left( \tilde{\state{u}}_L, \tilde{\state{u}}^{+}_L \right)
}_{(c):= \text{ boundary term (left)}} 
\nonumber\\
&-
\underbrace{
	\sum_{j=0}^N \entVar_j^T \ell_{j}(+1) 
	\left[
		\state{f}^* \left( \state{u}_j , \tilde{\state{u}}_R \right)
		- \sum_{k=0}^N \ell_{k}(+1) \, \state{f}^* \left( \tilde{\state{u}}_R, \state{u}_k  \right)
	\right]
}_{(d):= \text{ new terms (right)}} 
- \underbrace{ 
	\sum_{j=0}^N \entVar_j^T \ell_{j}(+1)  \numfluxb{f}^a \left( \tilde{\state{u}}_R, \tilde{\state{u}}^{+}_R \right)
}_{(e):= \text{ boundary term (right)}}.
\label{eq:EntBalance_Gauss_cons}
\end{align}

Let us first manipulate the volume terms:
\begin{align*} %\label{eq:EntBalance_Gauss_cons_vol}
(a)
	=& \sum_{j=0}^{N} \entVar^T_j \sum_{k=0}^N \hat{S}_{jk} \state{f}^{*}_{(j,k)} \\
\text{(skew-symmetry of $\hat{\mat{S}}$)} \qquad 
	=& \frac{1}{2} \sum_{j,k=0}^{N} \entVar^T_j (\hat{S}_{jk} - \hat{S}_{kj}) \state{f}^{*}_{(j,k)} 
	\\ 
\text{(re-index and symmetry of $\state{f}^{*}$)} \qquad 
	=& \frac{1}{2} \sum_{j,k=0}^{N} \hat{S}_{jk} (\entVar^T_j - \entVar^T_k)^T \state{f}^{*}_{(j,k)} 
	\\ 
\text{(definition of $r$, \eqref{eq:EntProdVol_cons})} \qquad 
	=& \frac{1}{2}  
	\sum_{j,k=0}^N \hat{S}_{jk} (\Psi_j - \Psi_k - r_{(j,k)}) 
	\\
\text{(re-index and skew-symmetry of $\hat{\mat{S}}$)} \qquad 
	=&  
	\sum_{j,k=0}^N \hat{S}_{jk} \left( \Psi_j - \frac{1}{2}  r_{(j,k)}) \right) 
	\\
\text{(def. of $\hat{\mat{S}}$ and SBP properties, \eqref{eq:genSBPprop})} \qquad
	=&  
	 \sum_{j=0}^N \Psi_j \underbrace{\sum_{k=0}^N 2 Q_{jk}}_{=0}
	- \sum_{j=0}^N \Psi_j \underbrace{\sum_{k=0}^N \hat{B}_{jk}}_{= \ell_j (+1) - \ell_j (-1)}
	- \frac{1}{2} \sum_{j,k=0}^N \hat{S}_{jk} r_{(j,k)}   
	\\ 
=&  
	 {\Psi}_L - {\Psi}_R - \frac{1}{2} \sum_{j,k=0}^N \hat{S}_{jk} r_{(j,k)},
\end{align*}
where ${\Psi}_L$ and ${\Psi}_R$ denote the interpolation of the entropy potential to the left and right boundaries of the element, respectively.

Next, we analyze the \textit{new} terms that connect all degrees of freedom with the left boundary:
\begin{align*} %\label{eq:EntBalance_Gauss_cons_new}
(b)
	=& \sum_{j=0}^N \entVar_j^T \ell_{j}(-1) 
	\left[
		\state{f}^* \left( \state{u}_j , \tilde{\state{u}}_L \right)
		- \sum_{k=0}^N \ell_{k}(-1) \, \state{f}^* \left( \tilde{\state{u}}_L , \state{u}_k \right)
	\right]
\\
\text{(eval. $\entVar$ on the boundary)} \qquad 
	=& \sum_{j=0}^N \entVar_j^T \ell_{j}(-1) \state{f}^* \left( \state{u}_j , \tilde{\state{u}}_L \right)
	- \sum_{k=0}^N \tilde{\entVar}_L^T \ell_{k}(-1) \, \state{f}^* \left( \tilde{\state{u}}_L , \state{u}_k \right)
\\
\text{(re-index and symmetry of $\state{f}^{*}$)} \qquad 
	=& \sum_{j=0}^N \ell_{j}(-1) (\entVar_j - \tilde{\entVar}_L)^T \state{f}^* \left( \state{u}_j , \tilde{\state{u}}_L \right)
\\
\text{(definition of $\tilde{r}$, \eqref{eq:EntProdTilde_cons},\eqref{eq:EntProdVol_cons})} \qquad 
	=& \sum_{j=0}^N \ell_{j}(-1) (\Psi_j - \tilde{\Psi}_L - \tilde{r}_{(j,L)})
\\
	=& {\Psi}_L - \tilde{\Psi}_L - \sum_{j=0}^N \ell_{j}(-1) \tilde{r}_{(j,L)}
\end{align*}
where $\tilde{\entVar}_L$ is the simple interpolation of the entropy variables to the left boundary, but we write it with a tilde since it also corresponds to the \textit{entropy-projected} quantity, and we introduce the entropy-projected entropy potential,
\begin{equation} \label{eq:entropyProj_Potential}
\tilde{\Psi}_L := \Psi \left( \sum_{i=0}^N \ell_i(-1) \entVar (\state{u}_i) \right).
\end{equation}

The entropy balance of the left boundary term reads
\begin{equation}
(c) = \sum_{j=0}^N \entVar_j^T \ell_{j}(-1) \numfluxb{f}^a \left( \tilde{\state{u}}_L, \tilde{\state{u}}^{+}_L \right)
= \tilde{\entVar}_L^T \numfluxb{f}^a \left( \tilde{\state{u}}_L, \tilde{\state{u}}^{+}_L \right).
\end{equation}
Terms $(d)$ and $(e)$ at the right boundary are analyzed in the same form as terms $(b)$ and $(c)$.
Gathering all contributions we obtain
\begin{align}
\sum_{j=0}^N \omega_j J \dot{S}_j  &= -(a)+(b)+(c)-(d)-(e) 
\nonumber \\
&= 
\tilde{\entVar}_L^T \numfluxb{f}^a \left( \tilde{\state{u}}_L, \tilde{\state{u}}^{+}_L \right)
- \tilde{\Psi}_L
- \tilde{\entVar}_R^T \numfluxb{f}^a \left( \tilde{\state{u}}_R, \tilde{\state{u}}^{+}_R \right)
+ \tilde{\Psi}_R
+ \frac{1}{2} \sum_{j,k=0}^N \hat{S}_{jk} r_{(j,k)}
+ \sum_{j=0}^N \left( \ell_{j}(+1) \tilde{r}_{(j,R)} - \ell_{j}(-1) \tilde{r}_{(j,L)} \right).
\end{align}
Finally, we sum and subtract the following outer terms,
\begin{align} \label{eq:outerTerms}
\sum_{j=0}^N \omega_j J \dot{S}_j =& 
\tilde{\entVar}_L^T \numfluxb{f}^a \left( \tilde{\state{u}}_L, \tilde{\state{u}}^{+}_L \right)
- \tilde{\Psi}_L
- \tilde{\entVar}_R^T \numfluxb{f}^a \left( \tilde{\state{u}}_R, \tilde{\state{u}}^{+}_R \right)
+ \tilde{\Psi}_R
+ \frac{1}{2} \sum_{j,k=0}^N \hat{S}_{jk} r_{(j,k)}
+ \sum_{j=0}^N \left( \ell_{j}(+1) \tilde{r}_{(j,R)} - \ell_{j}(-1) \tilde{r}_{(j,L)} \right)
\nonumber \\
&
+ \frac{1}{2} \left(
(\tilde{\entVar}^+_L)^T  \numfluxb{f}^a \left( \tilde{\state{u}}_L, \tilde{\state{u}}^{+}_L \right) 
- \tilde{\Psi}^+_L 
- (\tilde{\entVar}^+_R)^T  \numfluxb{f}^a \left( \tilde{\state{u}}_R, \tilde{\state{u}}^{+}_R \right)
+ \tilde{\Psi}^+_R
\right)
\nonumber\\
&
- \frac{1}{2} \left(
(\tilde{\entVar}^+_L)^T  \numfluxb{f}^a \left( \tilde{\state{u}}_L, \tilde{\state{u}}^{+}_L \right)
- \tilde{\Psi}^+_L 
- (\tilde{\entVar}^+_R)^T  \numfluxb{f}^a \left( \tilde{\state{u}}_R, \tilde{\state{u}}^{+}_R \right)
+ \tilde{\Psi}^+_R
\right),
\end{align}
where $\tilde{\entVar}^+_L$, $\tilde{\entVar}^+_R$, $\tilde{\Psi}^+_L$ and $\tilde{\Psi}^+_R$ are the outer entropy variables and potentials.
Simplifying using the definitions of the numerical entropy flux, \eqref{eq:numEntFlux_cons}, and the surface entropy production \eqref{eq:EntProd_cons}, we obtain 
\begin{align}
\sum_{j=0}^N \omega_j J \dot{S}_j = 
\numflux{f}^S \left( \tilde{\state{u}}_L, \tilde{\state{u}}^{+}_L \right)
-\numflux{f}^S \left( \tilde{\state{u}}_R, \tilde{\state{u}}^{+}_R \right) 
&+ \frac{1}{2} \left[
	\hat{r} \left( \tilde{\state{u}}_L, \tilde{\state{u}}^{+}_L \right)
  + \hat{r} \left( \tilde{\state{u}}_R, \tilde{\state{u}}^{+}_R \right)
\right]
\nonumber \\
&+ \frac{1}{2} \sum_{j,k=0}^N \hat{S}_{jk} r_{(j,k)}
+ \sum_{j=0}^N \left( \ell_{j}(+1) \tilde{r}_{(j,R)} - \ell_{j}(-1) \tilde{r}_{(j,L)} \right),
\end{align}

\end{proof}
The proof of EC/ES for the three-dimensional curvilinear discretization can also be found in \cite{chan2019efficient}.

\subsection{Entropy-Stable Gauss-DGSEM for the Non-Conservative GLM-MHD System} \label{sec:Gauss_DGSEM_MHD}

Unfortunately, it is not trivial to carry over the LGL-DGSEM discretization of the non-conservative GLM-MHD system \eqref{eq:LGL_DGSEM_GLMMHD} to Gauss points.
The direct way to carry over \eqref{eq:LGL_DGSEM_GLMMHD} to Gauss is to simply add the surface and volume numerical non-conservative terms to \eqref{eq:Gauss_DGSEM_cons}.
However, the resulting discretization is not entropy consistent, as we show in Appendix \ref{sec:FirstAttemptGaussMHD}.
Instead, in this paper we propose the following novel Gauss-DGSEM discretization of the GLM-MHD equations:
\small
\begin{empheq}[box=\fbox]{align} \label{eq:Gauss_DGSEM_GLMMHD}
% Volume terms
%%%%%%%%%%%%%%
J \omega_j & \dot{\state{u}}_j 
+ \sum_{k=0}^N \hat{S}_{jk} \left( \state{f}^{*}_{(j,k)} + \numnonconsD{\Jan}_{(j,k)} \right)
\nonumber \\
% Left boundary terms
%%%%%%%%%%%%%%%%%%%%%
-& \ell_{j}(-1) \left[
   \state{f}^* \left( \state{u}_j , \tilde{\state{u}}_L \right)
 + \numnonconsD{\Jan} \left( \state{u}_j , \tilde{\state{u}}_L \right)
 - \sum_{k=0}^N \ell_{k}(-1) \left( 
 	\state{f}^* \left( \tilde{\state{u}}_L , \state{u}_k \right)
   +\numnonconsD{\Jan} \left( \tilde{\state{u}}_L , \state{u}_k \right) 
 \right)
 + \numfluxb{f}^a \left( \tilde{\state{u}}_L, \tilde{\state{u}}^{+}_L \right)
 + \numnonconsD{\Jan} \left( \tilde{\state{u}}_L, \tilde{\state{u}}^{+}_L \right)
 \right]
\nonumber\\
% Right boundary terms
%%%%%%%%%%%%%%%%%%%%%%
+& \ell_{j}(+1) \left[
 \state{f}^* \left( \state{u}_j , \tilde{\state{u}}_R \right)
 + \numnonconsD{\Jan} \left( \state{u}_j , \tilde{\state{u}}_R \right)
 - \sum_{k=0}^N \ell_{k}(+1) \left(
 	 \state{f}^* \left( \tilde{\state{u}}_R  , \state{u}_k \right)
 	 +\numnonconsD{\Jan} \left( \tilde{\state{u}}_R , \state{u}_k \right) 
 \right)
 + \numfluxb{f}^a \left( \tilde{\state{u}}_R, \tilde{\state{u}}^{+}_R \right)
 + \numnonconsD{\Jan} \left( \tilde{\state{u}}_R, \tilde{\state{u}}^{+}_R \right)
\right]
= \state{0}.
\end{empheq}
\normalsize

\begin{remark} \label{rem:newLGL=Marvins}
Using \eqref{eq:Gauss_DGSEM_GLMMHD}, it is possible to reformulate the LGL discretization of the GLM-MHD system, \eqref{eq:LGL_DGSEM_GLMMHD}, as
\begin{empheq}{align} \label{eq:LGL_DGSEM_GLMMHD_cleanedUp} %[box=\fbox]
J \omega_j \dot{\state{u}}_j 
+ &\sum_{k=0}^N S_{jk} \left( \state{f}^{*}_{(j,k)} + \numnonconsD{\Jan}_{(j,k)} \right)
- \delta_{j0} \left( \numfluxb{f}^a_{(0,L)} + \numnonconsD{\Jan}_{(0,L)} \right)
+ \delta_{jN} \left( \numfluxb{f}^a_{(N,R)} + \numnonconsD{\Jan}_{(N,R)} \right)
= \state{0},
\end{empheq}
where the volume and surface numerical non-conservative terms are now consistent \textbf{and} interchangeable, because they are the same term, and there is no need for Bohm et al.'s \cite{Bohm2018} volume numerical non-conservative term, $\Jan^*_{(j,k)}$.

It can be shown that \eqref{eq:LGL_DGSEM_GLMMHD_cleanedUp} is algebraically equivalent to \eqref{eq:LGL_DGSEM_GLMMHD} using the identity
\begin{equation} \label{eq:nonConsIdentity}
    \Jan^*_{(j,k)} = 2\numnonconsD{\Jan}_{(j,k)} - \Jan_j,
\end{equation}
which is obtained by combining \eqref{eq:DiamondFluxes} and \eqref{eq:volNonCons_D1}.
The new volume non-conservative term then reads
\begin{align*}
    \sum_{k=0}^N S_{jk} \numnonconsD{\Jan}_{(j,k)}
    &= \sum_{k=0}^N \left( 2Q_{jk} - B_{jk} \right) \numnonconsD{\Jan}_{(j,k)} \\
    &= \sum_{k=0}^N \left( Q_{jk} (\Jan^*_{(j,k)} + \Jan_j) - B_{jk} \numnonconsD{\Jan}_{(j,k)} \right) \\
    &= \sum_{k=0}^N \left( Q_{jk} \Jan^*_{(j,k)} - B_{jk} \numnonconsD{\Jan}_{(j,k)} \right) \\
\text{(LGL boundary matrix, \eqref{eq:Bmat_LGL})} \qquad 
    &= \sum_{k=0}^N Q_{jk} \Jan^*_{(j,k)} - \delta_{jN} \Jan_N + \delta_{j0} \Jan_0,
\end{align*}
which is equivalent to the non-conservative terms in \eqref{eq:LGL_DGSEM_GLMMHD}.
\end{remark}

\begin{lemma} \label{lemma:Entropy_Gauss_DGSEM_GLMMHD}
The semi-discrete entropy balance of the Gauss-DGSEM discretization of the non-conservative GLM-MHD system, \eqref{eq:Gauss_DGSEM_GLMMHD}, integrating over an entire element, reads
\begin{align} \label{eq:Entropy_Gauss_DGSEM_GLMMHD}
\sum_{j=0}^N \omega_j J \dot{S}_j = 
\numflux{f}^S \left( \tilde{\state{u}}_L, \tilde{\state{u}}^{+}_L \right)
-\numflux{f}^S \left( \tilde{\state{u}}_R, \tilde{\state{u}}^{+}_R \right) 
&+ \frac{1}{2} \left[
	\hat{r} \left( \tilde{\state{u}}_L, \tilde{\state{u}}^{+}_L \right)
  + \hat{r} \left( \tilde{\state{u}}_R, \tilde{\state{u}}^{+}_R \right)
\right]
\nonumber \\
&+ \frac{1}{2} \sum_{j,k=0}^N \hat{S}_{jk} r_{(j,k)}
+ \sum_{j=0}^N \left( \ell_{j}(+1) \tilde{r}_{(j,R)} - \ell_{j}(-1) \tilde{r}_{(j,L)} \right),
\end{align}
where the last two sums denote the contribution of the volumetric entropy production terms \eqref{eq:EntProdVol_MHD}, $\tilde{r}_{(j,L)}$ and $\tilde{r}_{(j,R)}$ are again the volumetric entropy production terms between the solution at node $j$ and the entropy-projected solution at the left boundary and right boundary, respectively \eqref{eq:EntProdTilde_cons}, and the surface numerical entropy flux and entropy production are defined as in \eqref{eq:numEntFlux_MHD}  and \eqref{eq:EntProd_MHD}.
\end{lemma}

\begin{proof}
The proof follows from the proof of Lemma \ref{lemma:Entropy_Gauss_DGSEM_cons}. 
We first manipulate the volume terms to obtain
\begin{align*} %\label{eq:EntBalance_Gauss_cons_vol}
(a)
	=& \sum_{j=0}^{N} \entVar^T_j \sum_{k=0}^N \hat{S}_{jk} \left( \state{f}^{*}_{(j,k)} + \numnonconsD{\Jan}_{(j,k)} \right) \\
\text{(skew-sym. of $\hat{\mat{S}}$, sym. of $\state{f}^*$ \& re-index)} \qquad 
	=& \frac{1}{2} \sum_{j,k=0}^{N} \hat{S}_{jk} 
	\left( 
	(\entVar^T_j - \entVar^T_k)^T \state{f}^{*}_{(j,k)} 
	+ \entVar^T_j \numnonconsD{\Jan}_{(j,k)}
	- \entVar^T_k \numnonconsD{\Jan}_{(k,j)}
	\right)
	\\ 
\text{(definition of $r$, \eqref{eq:EntProdVol_MHD})} \qquad 
	=& \frac{1}{2}  
	\sum_{j,k=0}^N \hat{S}_{jk} (\Psi_j - \Psi_k - r_{(j,k)}) 
	\\
\text{(re-index, def. of $\hat{\mat{S}}$ and SBP properties, \eqref{eq:genSBPprop})} \qquad
	=&  
	 {\Psi}_L - {\Psi}_R - \frac{1}{2} \sum_{j,k=0}^N \hat{S}_{jk} r_{(j,k)},
\end{align*}
where we again use the interpolation of the entropy potential at the boundaries, ${\Psi}_L$ and ${\Psi}_R$.

Next, analyzing the \textit{new} terms that connect all degrees of freedom with the left boundary we find
\begin{align*} %\label{eq:EntBalance_Gauss_cons_new}
(b)
	=& \sum_{j=0}^N \entVar_j^T \ell_{j}(-1) 
	\left[
		\state{f}^* \left( \state{u}_j , \tilde{\state{u}}_L \right)
		+\numnonconsD{\Jan} \left( \state{u}_j , \tilde{\state{u}}_L \right)
		- \sum_{k=0}^N \ell_{k}(-1) 
		\left( \state{f}^* \left( \tilde{\state{u}}_L , \state{u}_k \right)
		+\numnonconsD{\Jan} \left( \tilde{\state{u}}_L , \state{u}_j \right)
		\right)
	\right]
\\
\text{(eval. $\tilde{\entVar}$ \& re-index)} \qquad 
	=& \sum_{j=0}^N \ell_{j}(-1) \left(
	(\entVar_j - \tilde{\entVar}_L)^T \state{f}^* \left( \state{u}_j , \tilde{\state{u}}_L \right)
	+ \entVar_j^T \numnonconsD{\Jan} \left( \state{u}_j , \tilde{\state{u}}_L \right)
	- \tilde{\entVar}_L^T \numnonconsD{\Jan} \left( \tilde{\state{u}}_L , \state{u}_j \right)
	\right)
\\
\text{(def. of $\tilde{r}$, \eqref{eq:EntProdTilde_cons},\eqref{eq:EntProdVol_MHD})} \qquad 
	=& {\Psi}_L - \tilde{\Psi}_L - \sum_{j=0}^N \ell_{j}(-1) \tilde{r}_{(j,L)},
\end{align*}
where $\tilde{\entVar}_L$ is again the interpolation of the entropy variables to the left boundary, and we use the entropy-projected entropy potential, \eqref{eq:entropyProj_Potential}.

The entropy balance of the left boundary term reads
\begin{equation}
(c) = \sum_{j=0}^N \entVar_j^T \ell_{j}(-1)
\left[
 \numfluxb{f}^a \left( \tilde{\state{u}}_L, \tilde{\state{u}}^{+}_L \right)
+\numnonconsD{\Jan} \left( \tilde{\state{u}}_L, \tilde{\state{u}}^{+}_L \right)
\right]
= \tilde{\entVar}_L^T 
\left[
 \numfluxb{f}^a \left( \tilde{\state{u}}_L, \tilde{\state{u}}^{+}_L \right)
+\numnonconsD{\Jan} \left( \tilde{\state{u}}_L, \tilde{\state{u}}^{+}_L \right)
\right].
\end{equation}
As in the previous proof, terms $(d)$ and $(e)$ are analyzed in the same form as terms $(b)$ and $(c)$.
Gathering all contributions we obtain
\begin{align}
\sum_{j=0}^N \omega_j J \dot{S}_j  =& -(a)+(b)+(c)-(d)-(e) 
\nonumber \\
=& 
% Left boundary
%%%%%%%%%%%%%%%
\tilde{\entVar}_L^T 
\left[
 \numfluxb{f}^a \left( \tilde{\state{u}}_L, \tilde{\state{u}}^{+}_L \right)
+\numnonconsD{\Jan} \left( \tilde{\state{u}}_L, \tilde{\state{u}}^{+}_L \right)
\right]
- \tilde{\Psi}_L
% Right boundary
%%%%%%%%%%%%%%%%
- \tilde{\entVar}_R^T 
\left[
 \numfluxb{f}^a \left( \tilde{\state{u}}_R, \tilde{\state{u}}^{+}_R \right)
+\numnonconsD{\Jan} \left( \tilde{\state{u}}_R, \tilde{\state{u}}^{+}_R \right)
\right]
+ \tilde{\Psi}_R
\nonumber \\
% Standard production terms
%%%%%%%%%%%%%%%%%%%%%%%%%%%
&
+ \frac{1}{2} \sum_{j,k=0}^N \hat{S}_{jk} r_{(j,k)}
+ \sum_{j=0}^N \left( \ell_{j}(+1) \tilde{r}_{(j,R)} - \ell_{j}(-1) \tilde{r}_{(j,L)} \right).
\end{align}
Finally, summing and subtracting the outer terms, as in \eqref{eq:outerTerms}, and simplifying we obtain 
\begin{align}
\sum_{j=0}^N \omega_j J \dot{S}_j = 
\numflux{f}^S \left( \tilde{\state{u}}_L, \tilde{\state{u}}^{+}_L \right)
&-\numflux{f}^S \left( \tilde{\state{u}}_R, \tilde{\state{u}}^{+}_R \right) 
+ \frac{1}{2} \left[
	\hat{r} \left( \tilde{\state{u}}_L, \tilde{\state{u}}^{+}_L \right)
  + \hat{r} \left( \tilde{\state{u}}_R, \tilde{\state{u}}^{+}_R \right)
\right]
\nonumber \\
&+ \frac{1}{2} \sum_{j,k=0}^N \hat{S}_{jk} r_{(j,k)}
+ \sum_{j=0}^N \left( \ell_{j}(+1) \tilde{r}_{(j,R)} - \ell_{j}(-1) \tilde{r}_{(j,L)} \right).
\end{align}

\end{proof}

\subsubsection{Extension to 3D Curvilinear Meshes}
In one dimension, it was trivial to use the surface numerical non-conservative term, $\numnonconsD{\Jan}$, in the volume integral and in the terms that connect all nodes with the surface.
However, in three-dimensional curvilinear meshes, we require the volume numerical fluxes (and non-conservative terms) to perform \textit{metric dealiasing} \cite{Gassner2016}.
For the conservative two-point fluxes, the metric dealiasing is achieved with the simple average of the metric terms, but the metric dealiasing is not as trivial with the numerical non-conservative terms.

Fortunately, the generalized "surface" numerical non-conservative term, introduced by \citet[Appendix C.1]{Rueda-Ramirez2020} as an auxiliary variable for the proofs, provides a consistent way to do the metric dealiasing, which also guarantees entropy consistency, as we show below.

We obtain the three-dimensional Gauss-DGSEM discretization of the GLM-MHD system on curvilinear meshes by extending the one-dimensional variant using tensor-product basis expansions.
For brevity, we use the same polynomial degree in all three spatial directions, although different polynomial degrees can be used for different directions, as in \cite{RuedaRamirez2019,RuedaRamirez2019a}. The extended version of \eqref{eq:Gauss_DGSEM_GLMMHD} reads

%%%%
%
\small
\begin{empheq}[box=\fbox]{align} \label{eq:Gauss_DGSEM_GLMMHD_3D}
J_{ijk} \omega_{ijk} \dot{\state{u}}_{ijk} 
%=========
% Xi terms
%=========
&+ 
\omega_{jk} \left\lbrace
% Volume terms
%%%%%%%%%%%%%%
  \sum_{m=0}^N \hat{S}_{im} 
    \tilde{\state{\Gamma}}^1_{(i,m)jk}
% Boundary terms
%%%%%%%%%%%%%%%%
+ \left[\ell_{i}(\xi_b) \Big(
    \tilde{\state{\Gamma}}^1\!\left( \state{u}_{ijk} , \tilde{\state{u}}_b \right)
 - \sum_{m=0}^N \ell_{m}(\xi_b) 
    \tilde{\state{\Gamma}}^1\!\left( \tilde{\state{u}}_b , \state{u}_{mjk} \right)
 + \numfluxb{\Gamma}^1\!\left( \tilde{\state{u}}_b, \tilde{\state{u}}^{+}_b \right)
 \Big)\right]_{\xi_b=-1, \, b\mapsto{Ljk}}^{\xi_b=+1, \, b \mapsto{Rjk}}
\right\rbrace
\nonumber\\
%==========
% Eta terms
%==========
&+ 
\omega_{ik} \left\lbrace
% Volume terms
%%%%%%%%%%%%%%
  \sum_{m=0}^N \hat{S}_{jm} 
    \tilde{\state{\Gamma}}^2_{i(j,m)k}
% Boundary terms
%%%%%%%%%%%%%%%%
+ \left[\ell_{j}(\eta_b) \Big(
    \tilde{\state{\Gamma}}^2\!\left( \state{u}_{ijk} , \tilde{\state{u}}_b \right)
 - \sum_{m=0}^N \ell_{m}(\eta_b) 
    \tilde{\state{\Gamma}}^2\!\left( \tilde{\state{u}}_b , \state{u}_{imk} \right)
 + \numfluxb{\Gamma}^2\!\left( \tilde{\state{u}}_b, \tilde{\state{u}}^{+}_b \right)
 \Big)\right]_{\eta_b=-1, \, b \mapsto{iLk}}^{\eta_b=+1, \, b \mapsto{iRk}}
\right\rbrace
\nonumber\\
%===========
% Zeta terms
%===========
&+ 
\omega_{ij} \left\lbrace
% Volume terms
%%%%%%%%%%%%%%
  \sum_{m=0}^N \hat{S}_{km} 
    \tilde{\state{\Gamma}}^3_{ij(k,m)}
% Boundary terms
%%%%%%%%%%%%%%%%
+ \left[\ell_{k}(\zeta_b) \Big(
    \tilde{\state{\Gamma}}^3\!\left( \state{u}_{ijk} , \tilde{\state{u}}_b \right)
 - \sum_{m=0}^N \ell_{m}(\zeta_b) 
    \tilde{\state{\Gamma}}^3\!\left( \tilde{\state{u}}_b , \state{u}_{ijm} \right)
 + \numfluxb{\Gamma}^3\!\left( \tilde{\state{u}}_b, \tilde{\state{u}}^{+}_b \right)
 \Big)\right]_{\zeta_b=-1, \, b \mapsto{ijL}}^{\zeta_b=+1, \, b\mapsto{ijR}}
\right\rbrace
% RHS
%%%%%
= \state{0}.
\end{empheq}
\normalsize

The mapping Jacobians, $J_{ijk}$, which may now be different at each degree of freedom of the element, and the contravariant basis vectors, $\vec{a}^m_{ijk}=\Nabla \xi^m$, define the mapping from reference space to physical space, $(\xi^1,\xi^2,\xi^3) \in [-1,1]^3 \rightarrow (x,y,z) \in \Omega$ \cite{Kopriva2006,kopriva2009implementing}. 

Moreover, the following new conventions are used:
\begin{itemize}
\item The two- and three-dimensional quadrature weights are defined from the one-dimensional weights as
\begin{equation}
\omega_{ij} := \omega_i \omega_j, ~~~~~ \omega_{ijk} := \omega_i \omega_j \omega_k.
\end{equation}
% New gamma terms
%%%%%%%%%%%%%%%%%
\item The newly introduced transformed two-point terms in the volume are defined as
\begin{equation} \label{eq:GammaTerm}
    \tilde{\state{\Gamma}}^1_{(i,m)jk}
    :=
    \tilde{\state{\Gamma}}^1  \left( \state{u}_{ijk} , \state{u}_{mjk} \right)
    :=
    \tilde{\state{f}}^{1*}_{(i,m)jk}
    +\numnonconsDxi{\tilde{\Jan}}_{(i,m)jk}\,.
\end{equation}

% Volume num fluxes
%%%%%%%%%%%%%%%%%%%

\item The volume numerical two-point fluxes are defined as
\begin{align}
\tilde{\state{f}}^{1*}_{(i,m)jk} &:= \blocktensor{f}^{*}(\state{u}_{ijk}, \state{u}_{mjk}) \cdot \avg{J\vec{a}^1}_{(i,m)jk}, ~~~~~
\tilde{\state{f}}^{2*}_{i(j,m)k} := \blocktensor{f}^{*}(\state{u}_{ijk}, \state{u}_{imk}) \cdot \avg{J\vec{a}^2}_{i(j,m)k}, \nonumber \\
\tilde{\state{f}}^{3*}_{ij(k,m)} &:= \blocktensor{f}^{*}(\state{u}_{ijk}, \state{u}_{ijm}) \cdot \avg{J\vec{a}^3}_{ij(k,m)},
\end{align}
where $\blocktensor{f}^{*}$ is a symmetric and consistent two-point averaging flux function, which can be selected to provide entropy conservation \cite{Gassner2016,Gassner2018,Manzanero2020}.

\item The generalized "surface" numerical non-conservative term between two nodes that are aligned in the $\xi$-direction reads \cite[Appendix C.1]{Rueda-Ramirez2020}
\begin{empheq}[box=\fbox]{align} \label{eq:numNonCons3D}
\numnonconsDxi{\tilde{\Jan}}_{(i,m)jk} := 
\frac{1}{2} \left[ (\vec{B}\cdot J\vec{a}^1)_{ijk} + \vec{B}_{mjk} \cdot \avg{J\vec{a}^1}_{(i,m)jk} \right] \phiMHD_{ijk}  +
\phiGLM_{ijk} \cdot J\vec{a}^1_{ijk} \avg{\psi}_{(i,m)jk},
\end{empheq}
and the other directions are constructed in an analogous way.
We write the term in a box since it is the most important ingredient for the three-dimensional entropy-stable Gauss-DGSEM discretization of the GLM-MHD equations on curvilinear meshes.

\item On an element boundary, again an \textit{entropy-projected} solution is used, such as 
\begin{equation}
\tilde{\state{u}}_{Ljk} := \state{u} \left( \sum_{i=0}^N \ell_i(-1) \entVar (\state{u}_{ijk}) \right), \, 
\tilde{\state{u}}_{Rjk} := \state{u} \left( \sum_{i=0}^N \ell_i(+1) \entVar (\state{u}_{ijk}) \right).
\label{eq:entropy_projection_3d}
\end{equation}

\item The new two-point term at the element interface is defined as
\begin{equation}
    \numfluxb{{\Gamma}}^1  \left(\tilde{\state{u}}_{b} , \tilde{\state{u}}^+_{b} \right)
    :=
    \numfluxb{{f}}^{1a} \left(\tilde{\state{u}}_{b} ,\tilde{\state{u}}^+_{b} \right)
    +\numnonconsDxi{\hat{\Jan}} \left(\tilde{\state{u}}_{b} ,\tilde{\state{u}}^+_{b} \right)\,,
\end{equation}
 and depends on the surface numerical flux and non-conservative surface term, which both include the transformation with the surface metric. The solution from the neighbor element is marked by a $+$.
Note that on a conforming mesh, the normal surface metric is unique at the element interface, and thus the numerical non-conservative term from \eqref{eq:numNonCons3D} evaluated at the right element interface ($\xi_b=+1$) reduces to
\begin{equation}
\numnonconsDxi{\hat{\Jan}} \left( \tilde{\state{u}}_{Rjk} , \tilde{\state{u}}^+_{Rjk} \right) = 
J\vec{a}^1_{Rjk}
\cdot 
\left[
\avg{\vec{B}}_{(R,R^+)jk}  \phiMHD_{Rjk}  +
\phiGLM_{Rjk} \avg{\psi}_{(R,R^+)jk}
\right],
\end{equation}
where $J\vec{a}^1_{Rjk}$ is the non-normalized surface normal and $\avg{\cdot}_{(R,R^+)jk}$ denotes the average operator across the element interface. The non-conservative surface term is then equivalent the one given in \cite{Bohm2018}.
\end{itemize}

\begin{remark}
If  LGL  nodes are used, the novel three-dimensional discretization, \eqref{eq:Gauss_DGSEM_GLMMHD_3D}, is equivalent to the three-dimensional LGL-DGSEM for the GLM-MHD system of \citet{Bohm2018},
\begin{align}  \label{eq:DGSEM_MHD_3D}
J_{ijk} \omega_{ijk} \dot{\state{u}}^{\DG}_{ijk} = 
%%%%%%%%%%%%%%%%%%%%%%%%%%%%%%%%%%%%%%%%%%%%%
% VOLUME TERMS
%%%%%%%%%%%%%%%%%%%%%%%%%%%%%%%%%%%%%%%%%%%%%
&
{%\color{red}
 \omega_{jk} \left( - 2 \sum_{m=0}^{N} Q_{im} \tilde{\state{f}}^{1*}_{(i,m)jk} 
- \sum_{m=0}^{N} Q_{im} \tilde{\Jan}^{1*}_{(i,m)jk} 
-  
  \delta_{i0} \left[ (\blocktensor{f} + \JanVec) \cdot J\vec{a}^1 \right]_{0jk}
+ \delta_{iN} \left[ (\blocktensor{f} + \JanVec)\cdot J\vec{a}^1 \right]_{Njk} \right) 
}
\nonumber\\ 
+& \omega_{ik} \left( - 2\sum_{m=0}^{N} Q_{jm} \tilde{\state{f}}^{2*}_{i(j,m)k}  
- \sum_{m=0}^{N} Q_{jm} \tilde{\Jan}^{2*}_{i(j,m)k} 
- 
  \delta_{j0} \left[ (\blocktensor{f} + \JanVec) \cdot J\vec{a}^2 \right]_{i0k}
+ \delta_{jN} \left[ (\blocktensor{f} + \JanVec) \cdot J\vec{a}^2 \right]_{iNk}  \right) 
\nonumber\\ 
+&\omega_{ij} \left( - 2\sum_{m=0}^{N} Q_{km} \tilde{\state{f}}^{3*}_{ij(k,m)} 
- \sum_{m=0}^{N} Q_{km} \tilde{\Jan}^{3*}_{ij(k,m)} 
- 
  \delta_{k0} \left[ (\blocktensor{f} + \JanVec) \cdot J\vec{a}^3 \right]_{ij0}
+ \delta_{kN} \left[ (\blocktensor{f} + \JanVec) \cdot J\vec{a}^3 \right]_{ijN} \right)
\nonumber\\
%%%%%%%%%%%%%%%%%%%%%%%%%%%%%%%%%%%%%%%%%%%%%
% BOUNDARY TERMS
%%%%%%%%%%%%%%%%%%%%%%%%%%%%%%%%%%%%%%%%%%%%%
+&
{%\color{blue} 
\omega_{jk} \left(
  \delta_{i0} \left[\numfluxb{f}^{1a}_{(0,L)jk} + \numnonconsDxi{\hat{\Jan}}_{(0,L)jk} \right]
- \delta_{iN} \left[\numfluxb{f}^{1a}_{(N,R)jk} + \numnonconsDxi{\hat{\Jan}}_{(N,R)jk} \right] \right)
}
\nonumber\\
+& \omega_{ik} \left(
  \delta_{j0} \left[\numfluxb{f}^{2a}_{i(0,L)k} + \numnonconsDeta{\hat{\Jan}}_{i(0,L)k} \right]
- \delta_{jN} \left[\numfluxb{f}^{2a}_{i(N,R)k} + \numnonconsDeta{\hat{\Jan}}_{i(N,R)k} \right] \right)
\nonumber\\
+& \omega_{ij} \left(
  \delta_{k0} \left[\numfluxb{f}^{3a}_{ij(0,L)} + \numnonconsDzeta{\hat{\Jan}}_{ij(0,L)} \right]
- \delta_{kN} \left[\numfluxb{f}^{3a}_{ij(N,R)} + \numnonconsDzeta{\hat{\Jan}}_{ij(N,R)} \right] \right).
\end{align}
It is straight-forward to reproduce the findings of Remark \ref{rem:newLGL=Marvins}, since the non-conservative terms identity holds in 3D,
\begin{equation}
    \tilde{\Jan}^{1*}_{(i,m)jk} = 2 \numnonconsDxi{\tilde{\Jan}}_{(i,m)jk} - \JanVec_{ijk} \cdot J\vec{a}^1_{ijk},
\end{equation}
where 
\begin{equation}
\JanVec := \phiMHD \vec{B} + \phiGLM \psi.
\end{equation}
\end{remark}

\begin{lemma} \label{lemma:Entropy_Gauss_DGSEM_GLMMHD_3D}
The semi-discrete entropy balance of the Gauss-DGSEM discretization of the three-dimensional non-conservative GLM-MHD system, \eqref{eq:Gauss_DGSEM_GLMMHD_3D}, integrating over an entire curvilinear element, reads
\begin{align} \label{eq:Entropy_Gauss_DGSEM_GLMMHD_3D}
\!\!\sum_{i,j,k=0}^N\!\!\!\omega_{ijk} J_{ijk} \dot{S}_{ijk} 
=&
	+\!\! \sum_{j,k=0}^N\!\! \omega_{jk} \left[
	 \numflux{f}^{1S}\!\left( \tilde{\state{u}}_{Ljk}, \tilde{\state{u}}^{+}_{Ljk} \right)
	-\numflux{f}^{1S}\!\left( \tilde{\state{u}}_{Rjk}, \tilde{\state{u}}^{+}_{Rjk} \right)
	+ \frac{1}{2} \left( 
	\hat{r}^1\!\left( \tilde{\state{u}}_{Ljk}, \tilde{\state{u}}^{+}_{Ljk} \right)
	+ \hat{r}^1\!\left( \tilde{\state{u}}_{Rjk}, \tilde{\state{u}}^{+}_{Rjk} \right) \right) 
	\right]
\nonumber\\&
	+\!\! \sum_{i,k=0}^N\!\! \omega_{ik}  \left[
	 \numflux{f}^{2S}\!\left( \tilde{\state{u}}_{iLk}, \tilde{\state{u}}^{+}_{iLk} \right)
	-\numflux{f}^{2S}\!\left( \tilde{\state{u}}_{iRk}, \tilde{\state{u}}^{+}_{iRk} \right)
	+ \frac{1}{2} \left( 
	  \hat{r}^2\!\left( \tilde{\state{u}}_{iLk}, \tilde{\state{u}}^{+}_{iLk} \right)
	+ \hat{r}^2\! \left( \tilde{\state{u}}_{iRk}, \tilde{\state{u}}^{+}_{iRk} \right) \right) 
	\right]
\nonumber\\&
	+\!\! \sum_{i,j=0}^N\!\! \omega_{ij}  \left[
	 \numflux{f}^{3S}\!\left( \tilde{\state{u}}_{ijL}, \tilde{\state{u}}^{+}_{ijL} \right)
	-\numflux{f}^{3S}\!\left( \tilde{\state{u}}_{ijR}, \tilde{\state{u}}^{+}_{ijR} \right)
	+ \frac{1}{2} \left( 
	  \hat{r}^3\!\left( \tilde{\state{u}}_{ijL}, \tilde{\state{u}}^{+}_{ijL} \right)
	+ \hat{r}^3\!\left( \tilde{\state{u}}_{ijR}, \tilde{\state{u}}^{+}_{ijR} \right) \right) 
	\right]
\nonumber\\&
	+ \frac{1}{2} \sum_{i,j,k,m=0}^N \!\!\!\omega_{ijk}
	\left(
	\hat{S}_{im}  r^1_{(i,m)jk} + \hat{S}_{jm} r^2_{i(j,m)k} + \hat{S}_{km} r^3_{ij(k,m)}
	\right)
\nonumber\\&
    +\!\! \sum_{i,j,k=0}^N \!\!\omega_{jk} \left( \ell_{i}(+1) \tilde{r}^1_{(i,R)jk} - \ell_{i}(-1) \tilde{r}^1_{(i,L)jk} \right)
\nonumber\\&
    +\!\! \sum_{i,j,k=0}^N\!\! \omega_{ik} \left( \ell_{j}(+1) \tilde{r}^2_{i(j,R)k} - \ell_{j}(-1) \tilde{r}^2_{i(j,L)k} \right)
\nonumber\\&
    +\!\! \sum_{i,j,k=0}^N\!\! \omega_{ij} \left( \ell_{k}(+1) \tilde{r}^3_{ij(k,R)} - \ell_{k}(-1) \tilde{r}^3_{ij(k,L)} \right).
\end{align}
where the first three lines denote the entropy flux and production across boundaries of the element, the fourth line is the entropy production from the volume integral, and the last three lines denote the volumetric entropy production between each node and the entropy-projected solution at the boundaries along lines of the element.
Moreover, the numerical entropy flux and production terms on the surface and volume are defined respectively as
\begin{align} \label{eq:numEntFlux_3D} 
    \numflux{f}^{1S}_{(R,R^+)jk} &= 
    \avg{\entVar}_{(R,R^+)jk}^T \numfluxb{f}^{1a}_{(R,R^+)jk} 
    + \frac{1}{2} \entVar^T_{Rjk} \numnonconsDxi{\hat{\Jan}}_{(R,R^+)jk}
    + \frac{1}{2} \entVar^T_{R^+jk} \numnonconsDxi{\hat{\Jan}}_{(R^+,R)jk}
    - J\vec{a}^1_{Rjk} \cdot \avg{\vec{\Psi}}_{(R,R^+)jk},
\\ \label{eq:EntProd_MHD_3D}
\hat{r}^1_{(R,R^+)jk} &=
\jump{\entVar}_{(R,R^+)jk}^T 
\numfluxb{f}^{1a}_{(R,R^+)jk}
+ \entVar^T_{R^+jk} \numnonconsDxi{\hat{\Jan}}_{(R^+,R)jk}
- \entVar^T_{Rjk} \numnonconsDxi{\hat{\Jan}}_{(R,R^+)jk}
- J\vec{a}^1_{Rjk} \cdot \jump{\vec{\Psi}}_{(R,R^+)jk},
\\ \label{eq:EntProdVol_MHD_3D}
r^1_{(i,m)jk} &=
\jump{\entVar}_{(i,m)jk}^T 
\state{f}^{1*}_{(i,m)jk}
+ \entVar^T_{mjk} \numnonconsDxi{\tilde{\Jan}}_{(m,i)jk}
- \entVar^T_{ijk} \numnonconsDxi{\tilde{\Jan}}_{(i,m)jk}
- \avg{J\vec{a}^1}_{(i,m)jk} \cdot \jump{\vec{\Psi}}_{(i,m)jk},
\end{align}
for two nodes aligned in the $\xi$ direction in the volume and the right boundary. The other directions and the left boundary are defined analogously.
\end{lemma}

\begin{proof}
Following the same strategy as in previous proofs, we first contract the volume terms in the $\xi$ direction with the entropy variables, and compute the integral along the $\xi$ direction.
We scale with the quantity $1/{\omega_{jk}}$ for convenience to obtain
\begin{align} %\label{eq:EntBalance_Gauss_cons_vol}
\frac{(a)^{\xi}_{jk}}{\omega_{jk}}
	=& \sum_{i=0}^{N} \entVar^T_{ijk}
  \sum_{m=0}^N \hat{S}_{im} 
    \tilde{\state{\Gamma}}^1_{(i,m)jk}
	\nonumber\\
\text{(definition of $\tilde{\state{\Gamma}}^1$, \eqref{eq:GammaTerm})} \quad 
	=& \sum_{i=0}^{N} \entVar^T_{ijk}
  \sum_{m=0}^N \hat{S}_{im} \left(
  \tilde{\state{f}}^{1*}_{(i,m)jk}
    +\numnonconsDxi{\tilde{\Jan}}_{(i,m)jk} \right)
    \nonumber	\\
\text{(skew-sym. of $\hat{\mat{S}}$, sym. of $\state{f}^{1*}$ \& re-index)} \quad 
	=& \frac{1}{2} \sum_{i,m=0}^{N} \hat{S}_{im} 
	\left( 
	(\entVar^T_{ijk} - \entVar^T_{mjk})^T \tilde{\state{f}}^{1*}_{(i,m)jk} 
	+ \entVar^T_{ijk} \numnonconsDxi{\tilde{\Jan}}_{(i,m)jk}
	- \entVar^T_{mjk} \numnonconsDxi{\tilde{\Jan}}_{(m,i)jk}
	\right)
	\nonumber\\ 
\text{(definition of $r^1$, \eqref{eq:EntProd_MHD_3D})} \quad 
	=& \frac{1}{2}  
	\sum_{i,m=0}^N \hat{S}_{im} \left( \avg{J\vec{a}^1}_{(i,m)jk} \cdot \left( \vec{\Psi}_{ijk} - \vec{\Psi}_{mjk} \right) - r^1_{(i,m)jk} \right) 
	\nonumber\\ 
\text{(re-index, sym. of $\avg{\cdot}$ \& skew-sym of $\hat{\mat{S}}$)} \quad 
	=& 
	\sum_{i,m=0}^N \hat{S}_{im} \left( \avg{J\vec{a}^1}_{(i,m)jk} \cdot \vec{\Psi}_{ijk} - \frac{1}{2}   r^1_{(i,m)jk} \right) 
	\nonumber\\ 
\text{(definition of $\hat{\mat{S}}$)} \quad 
	=& 
	\sum_{i,m=0}^N 2 Q_{im} \avg{J\vec{a}^1}_{(i,m)jk} \cdot \vec{\Psi}_{ijk}
	\nonumber\\
	&
	- \underbrace{\sum_{i,m=0}^N \hat{B}_{im} \avg{J\vec{a}^1}_{(i,m)jk} \cdot \vec{\Psi}_{ijk}}_{:= \bar{\Psi}^1_{Rjk} - \bar{\Psi}^1_{Ljk}}
	- \frac{1}{2} \sum_{i,m=0}^N \hat{S}_{im} r^1_{(i,m)jk}
	\label{eq:psiBar}\\
\text{(def. of $\avg{\cdot}$ \& SBP properties, \eqref{eq:genSBPprop})} \quad 
	=& 
	\sum_{i=0}^N  J\vec{a}^1_{ijk} \cdot \vec{\Psi}_{ijk} \underbrace{\sum_{m=0}^N Q_{im}}_{:=0}
	+
	\sum_{i,m=0}^N Q_{im}  J\vec{a}^1_{mjk} \cdot \vec{\Psi}_{ijk} 
	\nonumber\\
	&+ \bar{\Psi}^1_{Ljk} - \bar{\Psi}^1_{Rjk}
	- \frac{1}{2} \sum_{i,m=0}^N \hat{S}_{im} r^1_{(i,m)jk}
	\nonumber\\
%\text{(re-index, def. of $\hat{\mat{S}}$)} \quad
	=&  
	 \sum_{i,m=0}^N Q_{im} J\vec{a}^1_{mjk} \cdot \vec{\Psi}_{ijk} 
	 + \bar{\Psi}^1_{Ljk} - \bar{\Psi}^1_{Rjk}
	- \frac{1}{2} \sum_{i,m=0}^N \hat{S}_{im} r^1_{(i,m)jk}.
\end{align}

The quantities introduced in \eqref{eq:psiBar} denote the contravariant entropy potentials at the boundaries, $\bar{\Psi}^1_{Ljk}$ and $\bar{\Psi}^1_{Rjk}$, which are derived using the generalized SBP properties \eqref{eq:genSBPprop},
\begin{align*}
    \sum_{i,m=0}^N \hat{B}_{im} \avg{J\vec{a}^1}_{(i,m)jk} \cdot \vec{\Psi}_{ijk}
    &=
    \sum_{i=0}^N  \vec{\Psi}_{ijk} \cdot \sum_{m=0}^N  \hat{B}_{im} \avg{J\vec{a}^1}_{(i,m)jk}
    \\
    &=
    \underbrace{\sum_{i=0}^N  \ell_i(+1)
    \vec{\Psi}_{ijk} \cdot \avg{J\vec{a}^1}_{(i,R)jk} }_{:\bar{\Psi}^1_{Rjk}}
    - \underbrace{\sum_{i=0}^N \ell_i(-1)
    \vec{\Psi}_{ijk} \cdot \avg{J\vec{a}^1}_{(i,L)jk}
    }_{:=\bar{\Psi}^1_{Ljk}}.
\end{align*}

We now obtain the entropy production by the \textit{new} terms in $\xi$ that connect all degrees of freedom with the left boundary.
Again, we scale with the two-dimensional weight, $\omega_{jk}$, to obtain
\begin{align*} %\label{eq:EntBalan, $$ce_Gauss_cons_new}
\frac{(b)^{\xi}_{jk}}{\omega_{jk}}
	=& \sum_{i=0}^N \entVar_{ijk}^T \ell_{i}(-1) 
	\left[
		 \tilde{\state{\Gamma}}^1  \left( \state{u}_{ijk} , \tilde{\state{u}}_{Ljk} \right)
		- \sum_{m=0}^N \ell_{m}(-1) 
		\tilde{\state{\Gamma}}^1  \left( \tilde{\state{u}}_{Ljk} , \state{u}_{mjk} \right)
	\right]
\\
\text{(eval. $\tilde{\entVar}$ \& re-index)} \quad 
	=& \sum_{i=0}^N \ell_{i}(-1) \left[
	 \entVar_{ijk}^T \tilde{\state{\Gamma}}^1  \left( \state{u}_{ijk} , \tilde{\state{u}}_{Ljk} \right)
	- \tilde{\entVar}_{Ljk}^T \tilde{\state{\Gamma}}^1  \left( \tilde{\state{u}}_{Ljk} , \state{u}_{ijk} \right)
	\right]
\\
\text{(definition of $\tilde{\state{\Gamma}}^1$)} \quad 
=& \sum_{i=0}^N \ell_{i}(-1) \left[
    (\entVar_{ijk} - \tilde{\entVar}_{Ljk})^T \tilde{\state{f}}^{1*} \left( \state{u}_{ijk} , \tilde{\state{u}}_{Ljk} \right) \bigg. \right. 
    \\
    & ~~~~~~~~~~~~~~~~~~~~
    \left. + \entVar_{ijk}^T \numnonconsDxi{\tilde{\Jan}}  \left( \state{u}_{ijk} , \tilde{\state{u}}_{Ljk} \right)
	- \tilde{\entVar}_{Ljk}^T \numnonconsDxi{\tilde{\Jan}}  \left( \tilde{\state{u}}_{Ljk} , \state{u}_{ijk} \right)
	\right]
\\
\text{(definition of $\tilde{r}^1$, \eqref{eq:EntProdTilde_cons},\eqref{eq:EntProdVol_MHD_3D})} \quad 
	=& \sum_{i=0}^N \ell_{i}(-1) \left( \avg{J\vec{a}^1}_{(i,L)jk} \cdot \left( \vec{\Psi}_{ijk} - \vec{\tilde{\Psi}}_{Ljk} \right) - \tilde{r}^1_{(i,L)jk} \right)
\\
\text{(def. of $\bar{\Psi}^1_{Ljk}$ \& boundary metrics)} \quad 
	=& \bar{\Psi}^1_{Ljk} -  \left(\vec{\tilde{\Psi}}_{Ljk} \cdot J\vec{a}^1_{Ljk} \right)
	- \sum_{i=0}^N \ell_{i}(-1) \tilde{r}^1_{(i,L)jk} 
\end{align*}
where $\tilde{\entVar}_{Ljk}$ is again the interpolation of the entropy variables to the left boundary, and we use the entropy-projected entropy potential,
\begin{equation} \label{eq:entropyProj_Potential3D}
    \vec{\tilde{\Psi}}_{Ljk} := \vec{\Psi} \left( \sum_{i=0}^N \ell_i(-1) \entVar (\state{u}_{ijk}) \right).
\end{equation}

The integral of the left $\xi$ surface terms along the $\xi$ direction reads
\begin{align*}
(c)_{jk}^{\xi} =& 
\sum_{i=0}^N \entVar_{ijk}^T \omega_{jk} \ell_{i}(-1)\left[
  \numfluxb{f}^{1a}\! \left( \tilde{\state{u}}_{Ljk}, \tilde{\state{u}}^+_{Ljk} \right) + \numnonconsDxi{\hat{\Jan}} \left( \tilde{\state{u}}_{Ljk}, \tilde{\state{u}}^+_{Ljk} \right) 
  \right] 
  \\
	=& \omega_{jk} 
  \tilde{\entVar}_{Ljk}^T \left[ 
  \numfluxb{f}^{1a}\! \left( \tilde{\state{u}}_{Ljk}, \tilde{\state{u}}^+_{Ljk} \right) + \numnonconsDxi{\hat{\Jan}} \left( \tilde{\state{u}}_{Ljk}, \tilde{\state{u}}^+_{Ljk} \right) 
  \right]
\\
\text{(Sum zero)} \quad
	=& \omega_{jk} 
  \tilde{\entVar}_{Ljk}^T \left[ 
  \numfluxb{f}^{1a}\! \left( \tilde{\state{u}}_{Ljk}, \tilde{\state{u}}^+_{Ljk} \right) + \numnonconsDxi{\hat{\Jan}} \left( \tilde{\state{u}}_{Ljk}, \tilde{\state{u}}^+_{Ljk} \right) 
  \right]	
	\\
	&
	+ \frac{\omega_{jk}}{2} 
	\left(
	(\tilde{\entVar}^+_{Ljk})^T  \numfluxb{f}^{1a}\! \left( \tilde{\state{u}}^+_{Ljk} , \tilde{\state{u}}_{Ljk} \right)
	+ (\tilde{\entVar}^+_{Ljk})^T  \numnonconsDxi{\hat{\Jan}} \left( \tilde{\state{u}}^+_{Ljk} , \tilde{\state{u}}_{Ljk} \right)
	- \left(\vec{\tilde{\Psi}}^+_{Ljk} \cdot J\vec{a}^1_{Ljk} \right)
	\right)
	\\
	&
	- \frac{\omega_{jk}}{2} \left(
	(\tilde{\entVar}^+_{Ljk})^T  \numfluxb{f}^{1a}\! \left( \tilde{\state{u}}^+_{Ljk} , \tilde{\state{u}}_{Ljk} \right)
	+ (\tilde{\entVar}^+_{Ljk})^T \numnonconsDxi{\hat{\Jan}} \left( \tilde{\state{u}}^+_{Ljk} , \tilde{\state{u}}_{Ljk} \right)
	- \left(\vec{\tilde{\Psi}}^+_{Ljk} \cdot J\vec{a}^1_{Ljk} \right) \right)
\\
\text{(replace \eqref{eq:numEntFlux_3D},\eqref{eq:EntProd_MHD_3D})} \quad
	=& 
	\omega_{jk} \left(
	\numflux{f}^{1S}\!  \left( \tilde{\state{u}}_{ijL}, \tilde{\state{u}}^{+}_{ijL} \right)
	+ \frac{1}{2}
	\hat{r}^1\! \left( \tilde{\state{u}}_{ijL}, \tilde{\state{u}}^{+}_{ijL} \right)
	+ \left(\vec{\Psi}_{Ljk} \cdot J\vec{a}^1_{Ljk} \right) \right).
\numberthis \label{eq:EntropyDGSEM_Surf3D}
\end{align*}

Terms $(d)$ and $(e)$ are again analyzed in the same form as terms $(b)$ and $(c)$.
Gathering all contributions in the $\xi$ coordinate direction we obtain
\begin{align*}
 -(a)^{\xi}_{jk}+&(b)^{\xi}_{jk}+(c)^{\xi}_{jk}-(d)^{\xi}_{jk}-(e)^{\xi}_{jk} 
\nonumber \\
=&
\omega_{jk} \left[
	\numflux{f}^{1S}\! \left( \tilde{\state{u}}_{Ljk}, \tilde{\state{u}}^{+}_{Ljk} \right)
	-\numflux{f}^{1S}\! \left( \tilde{\state{u}}_{Rjk}, \tilde{\state{u}}^{+}_{Rjk} \right)
	+ \frac{1}{2} \left( 
	\hat{r}^1\! \left( \tilde{\state{u}}_{Ljk}, \tilde{\state{u}}^{+}_{Ljk} \right)
	+ \hat{r}^1\! \left( \tilde{\state{u}}_{Rjk}, \tilde{\state{u}}^{+}_{Rjk} \right) \right) \right.
	\\
	& ~~~~~~~~
	\left.
	+ \frac{1}{2} \sum_{i,m=0}^N \hat{S}_{im} r^1_{(i,m)jk} - \sum_{i,m=0}^N Q_{im} J\vec{a}^1_{mjk} \cdot \vec{\Psi}_{ijk}
	+ \sum_{i=0}^N \left( \ell_{i}(+1) \tilde{r}^1_{(i,R)jk} - \ell_{i}(-1) \tilde{r}^1_{(i,L)jk} \right)
	\right].
\numberthis \label{eq:EntropyDGSEM_Xi_3D}
\end{align*}

Summing \eqref{eq:EntropyDGSEM_Xi_3D} over $j$ and $k$, and adding the contributions of the terms in directions $\eta$ and $\zeta$ we obtain the desired result,
\begin{align*}
\!\!\sum_{i,j,k=0}^N\!\!\!\omega_{ijk} J_{ijk} \dot{S}_{ijk} 
=&
	+\!\!\sum_{j,k=0}^N \!\!\omega_{ik}  \left[
	 \numflux{f}^{1S}\!\left(\tilde{\state{u}}_{Ljk},\tilde{\state{u}}^{+}_{Ljk}\right)
	-\numflux{f}^{1S}\!\left(\tilde{\state{u}}_{Rjk},\tilde{\state{u}}^{+}_{Rjk}\right)
	+ \frac{1}{2} \left( 
	  \hat{r}^1\!\left(\tilde{\state{u}}_{Ljk},\tilde{\state{u}}^{+}_{Ljk}\right)
	+ \hat{r}^1\!\left(\tilde{\state{u}}_{Rjk},\tilde{\state{u}}^{+}_{Rjk}\right)\right) 
	\right]
\nonumber\\&
	+ \!\!\sum_{i,k=0}^N \!\!\omega_{ik}  \left[
	 \numflux{f}^{2S}\!\left(\tilde{\state{u}}_{iLk}, \tilde{\state{u}}^{+}_{iLk} \right)
	-\numflux{f}^{2S}\!\left(\tilde{\state{u}}_{iRk}, \tilde{\state{u}}^{+}_{iRk} \right)
	+ \frac{1}{2} \left( 
	  \hat{r}^2\!\left(\tilde{\state{u}}_{iLk}, \tilde{\state{u}}^{+}_{iLk} \right)
	+ \hat{r}^2\!\left( \tilde{\state{u}}_{iRk}, \tilde{\state{u}}^{+}_{iRk} \right) \right) 
	\right]
\nonumber\\&
	+ \!\!\sum_{i,j=0}^N\!\! \omega_{ij}  \left[
	 \numflux{f}^{3S}\!\left(\tilde{\state{u}}_{ijL}, \tilde{\state{u}}^{+}_{ijL} \right)
	-\numflux{f}^{3S}\!\left(\tilde{\state{u}}_{ijR}, \tilde{\state{u}}^{+}_{ijR} \right)
	+ \frac{1}{2} \left( 
	  \hat{r}^3\! \left( \tilde{\state{u}}_{ijL}, \tilde{\state{u}}^{+}_{ijL} \right)
	+ \hat{r}^3\!  \left( \tilde{\state{u}}_{ijR}, \tilde{\state{u}}^{+}_{ijR} \right) \right) 
	\right]
\nonumber\\&
	+ \frac{1}{2} \sum_{i,j,k,m=0}^N \omega_{ijk}
	\left(
	\hat{S}_{im}  r^1_{(i,m)jk} + \hat{S}_{jm} r^2_{i(j,m)k} + \hat{S}_{km} r^3_{ij(k,m)}
	\right)
\nonumber\\&
    + \!\!\sum_{i,j,k=0}^N\!\!\! \omega_{jk} \left( \ell_{i}(+1) \tilde{r}^1_{(i,R)jk} - \ell_{i}(-1) \tilde{r}^1_{(i,L)jk} \right)
\nonumber\\&
    + \!\!\sum_{i,j,k=0}^N\!\!\! \omega_{ik} \left( \ell_{j}(+1) \tilde{r}^2_{i(j,R)k} - \ell_{j}(-1) \tilde{r}^2_{i(j,L)k} \right)
\nonumber\\&
    + \!\!\sum_{i,j,k=0}^N\!\!\! \omega_{ij} \left( \ell_{k}(+1) \tilde{r}^3_{ij(k,R)} - \ell_{k}(-1) \tilde{r}^3_{ij(k,L)} \right).
\nonumber \\
	&
	- \!\!\sum_{i,j,k=0}^N\!\!\! \omega_{ijk} 
	\vec{\Psi}_{ijk} \cdot
	\underbrace{
	\sum_{m=0}^N
	\left(
	D_{im} J\vec{a}^1_{mjk} + D_{jm} J\vec{a}^2_{imk} + D_{km} J\vec{a}^3_{ijm}
	\right)
	}_{=0 },
\end{align*}
where the last term is equal to zero if the discrete metric identities,
\begin{equation} \label{eq:DiscMetricIdentities}
\left(
\sum_{l=1}^3 \bigpartialderiv{}{\xi^l} (J{a}^l_d)
\right)_{ijk} = 0, ~~~~~~ d \in \{ 1,2,3 \},
\end{equation}
hold for all the nodes of the element, $i,j,k \in \{ 0, \ldots, N \}$.
For example, they can be computed from a discrete curl \cite{Kopriva2006}.

\end{proof}

\begin{remark}
It is possible to replace the volume non-conservative two-point term \eqref{eq:numNonCons3D},
\begin{empheq}{align}
\numnonconsDxi{\tilde{\Jan}}_{(i,m)jk} := 
\frac{1}{2} \left[ (\vec{B}\cdot J\vec{a}^1)_{ijk} + \vec{B}_{mjk} \cdot \avg{J\vec{a}^1}_{(i,m)jk} \right] \phiMHD_{ijk}  +
\phiGLM_{ijk} \cdot J\vec{a}^1_{ijk} \avg{\psi}_{(i,m)jk},
\end{empheq}
with the alternative term
\begin{empheq}[box=\fbox]{align} \label{eq:numNonCons3D_other}
\numnonconsSxi{\tilde{\Jan}}_{(i,m)jk} := 
 \avg{\vec{B}}_{(i,m)jk} \cdot \avg{J\vec{a}^1}_{(i,m)jk}  \phiMHD_{ijk}  +
\phiGLM_{ijk} \cdot J\vec{a}^1_{ijk} \avg{\psi}_{(i,m)jk}.
\end{empheq}
Although both terms differ, i.e.,
\begin{equation}
    \numnonconsSxi{\tilde{\Jan}}_{(i,m)jk} = \numnonconsDxi{\tilde{\Jan}}_{(i,m)jk} 
    + \frac{1}{4} \phiMHD_{ijk} \vec{B}_{ijk} \cdot \jump{J\vec{a}^1}_{(i,m)jk},
\end{equation}
their evaluation in \eqref{eq:Gauss_DGSEM_GLMMHD_3D} is algebraically equivalent for generalized SBP operators if the metric identities \eqref{eq:DiscMetricIdentities} hold.
\end{remark}

\section{Numerical Results} \label{sec:Results}

In this section, we test the numerical accuracy and entropy consistency of our entropy-stable Gauss-DGSEM discretization of the GLM-MHD equations, and compare the results with other methods in the MHD literature.
In all cases, the time integration was performed with the explicit fourth-order five-stages Runge-Kutta scheme of \citet{carpenter1994fourth}.
The time-step size is computed as in \cite{Krais2019}, 
\begin{equation} \label{eq:TimeStepExp}
\Delta t = 
\frac{\text{CFL} \, \beta^a(N) \Delta x}{\lambda^a_{\max} (2N+1)},
\end{equation}
where CFL is the CFL number, $\lambda^a_{\max}$ is the largest eigenvalue, $\Delta x$ is the element size, and $\beta^a$ is a proportionality coefficient derived for the RK method from numerical experiments such that CFL$ \le 1$ must hold to obtain a (linear) CFL-stable time step for all polynomial degrees.
As a "conservative" approach, we use $\text{CFL}=0.5$.
The hyperbolic divergence cleaning speed, $c_h$, is adjusted at every time step as the maximum value that retains CFL-stability, and we use $\mu_0=1$ as the magnetic permeability of the medium.

All simulations presented in this section were computed {in parallel} with the 3D open-source code FLUXO (\url{www.github.com/project-fluxo/fluxo})
on the High Performance Computing (HPC) system ODIN of the Regional Computing Center of the University of Cologne (RRZK).
Each node of ODIN has two sockets, each with an Intel(R) Xeon(R) CPU E5-2670 0 @ 2.60GHz of eight cores.
The number of nodes for each simulation was selected depending on its computational cost.
The 2D simulations were computed with 2D extruded meshes with one element in the $z$-direction.

\subsection{Numerical Verification of the Schemes}

\subsubsection{Convergence Test with the Manufactured Solutions Method} \label{sec:mansol}

To test the accuracy of the entropy-stable Gauss DGSEM discretization of the GLM-MHD equations on three-dimensional curvilinear meshes, we run a convergence test with the method of manufactured solutions.
As in \cite{Bohm2018,Rueda-Ramirez2020}, we assume an exact solution to the ideal GLM-MHD system of the form
\begin{equation}\label{eq:ManSol}
\state{u}^{\mathrm{exact}} = \left[h,h,h,0,2h^2+h,\frac{1}{2} h,- \frac{1}{4} h,- \frac{1}{4} h,0\right]^T \text{ with } h = h(x,y,z,t) = 0.5 \sin(2\pi(x+y+z-t))+2,
\end{equation}
and a heat capacity ratio of $\gamma = 2$.

To obtain the exact solution \eqref{eq:ManSol}, we equip the ideal GLM-MHD system with a source term:
\begin{equation}
\partial_t \mathbf{u}  + \vec{\nabla} \cdot \blocktensor{f}^a(\state{u}) + \noncon(\mathbf{u}, \Nabla \mathbf{u}) =
\begin{pmatrix} 
h_x \\ 	
h_x + 4hh_x \\ 
h_x + 4hh_x \\ 
4hh_x \\ 
h_x + 12hh_x \\ 
\frac{1}{2} h_x \\ 
-\frac{1}{4} h_x \\ 
-\frac{1}{4} h_x \\ 0 \end{pmatrix}.
\label{residual}
\end{equation} 
We carry out the computations until the final time $t=1$ on the unit cube, $\Omega = [-1,1]^3$, with $2^3$, $4^3$, $8^3$, $16^3$, and $32^3$ hexahedral elements and the polynomial degrees $N=2,3,4,5$.
All boundaries are set to periodic.
In addition, we use the entropy-conservative flux of \citet{Derigs2018} for the volume and surface numerical fluxes, and supplement the surface fluxes with the nine-waves entropy-stable dissipation operator of \citet{Derigs2018}.

A  heavily-warped curvilinear mesh is obtained by applying a transformation function (adapted from \cite{chan2019efficient}) to all nodes of the mesh,
\begin{equation}
X(\xi,\eta,\zeta) = (x,y,z): \Omega \rightarrow f(\Omega)
\end{equation}
such that
\begin{align} \label{eq:warp}
y &= \eta + 
\alpha L_x
\cos \left(  3 \pi \left( \frac{\xi}{L_x} - \tilde{L}  \right) \right) 
\cos \left( \pi \left( \frac{\eta}{L_y} - \tilde{L} \right) \right) 
\cos \left( \pi \left( \frac{\zeta}{L_z} - \tilde{L} \right) \right) , \nonumber \\
x &= \xi + 
\alpha L_z
\cos \left( \pi \left( \frac{\xi}{L_x} - \tilde{L} \right) \right) 
\sin \left( 4 \pi \left( \frac{y}{L_y} - \tilde{L} \right) \right) 
\cos \left( \pi \left( \frac{\zeta}{L_z} - \tilde{L} \right) \right) , \nonumber \\
z &= \zeta + 
\alpha L_y
\cos \left( \pi \left( \frac{x}{L_x} - \tilde{L} \right) \right) 
\cos \left( 2 \pi \left( \frac{y}{L_y}  - \tilde{L} \right) \right) 
\cos \left( \pi \left( \frac{\zeta}{L_z} - \tilde{L} \right) \right) ,
\end{align}
with a warping factor $\alpha = 0.075$, the domain lengths $L_x = L_y = L_z = 2$, and the shift parameter $\tilde{L} = 0$.
The mesh was generated with the HOPR package \cite{hindenlang2015mesh} with a geometry mapping degree $N_{\mathrm{geo}}=2$.

Figure \ref{fig:EOCs} summarizes the results of the convergence analysis for the entropy-stable Gauss- and LGL-DGSEM.
We show the $\mathbb{L}_2$ norm of the discretization error for each variable,
\begin{equation} \label{eq:L2error}
\norm{\epsilon_i}_{2} = 
\left( \frac{\int^N_{\Omega} (u_i - u_i^{\mathrm{exact}})^2 \d \vec{x}}{\int^N_{\Omega} \d \vec{x}} \right)^{\frac{1}{2}}, ~~ \forall i \in \{1, \ldots, 9\},
\end{equation}
as a function of the number of degrees of freedom (DOFs).
In \eqref{eq:L2error}, the superscript $N$ denotes the numerical integration with a Gauss/LGL quadrature of $N+1$ points.
The entropy-stable Gauss collocation scheme is always more accurate than the LGL scheme.
In many cases, it is as accurate as the LGL scheme with one polynomial degree higher.
However, we note that the LGL-DGSEM allows for larger time-step sizes \cite{gassner2011comparison}.

\begin{figure}[htb]
\centering

\includegraphics[trim=150 60 520 195,clip,width=0.32\linewidth]{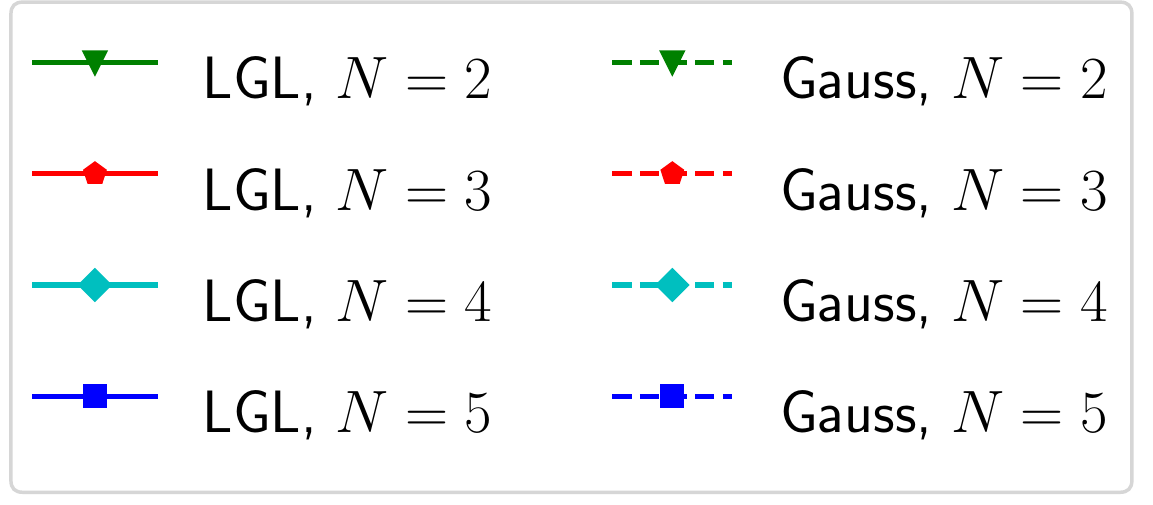}

\includegraphics[trim=0 0 0 0,clip,width=0.32\linewidth]{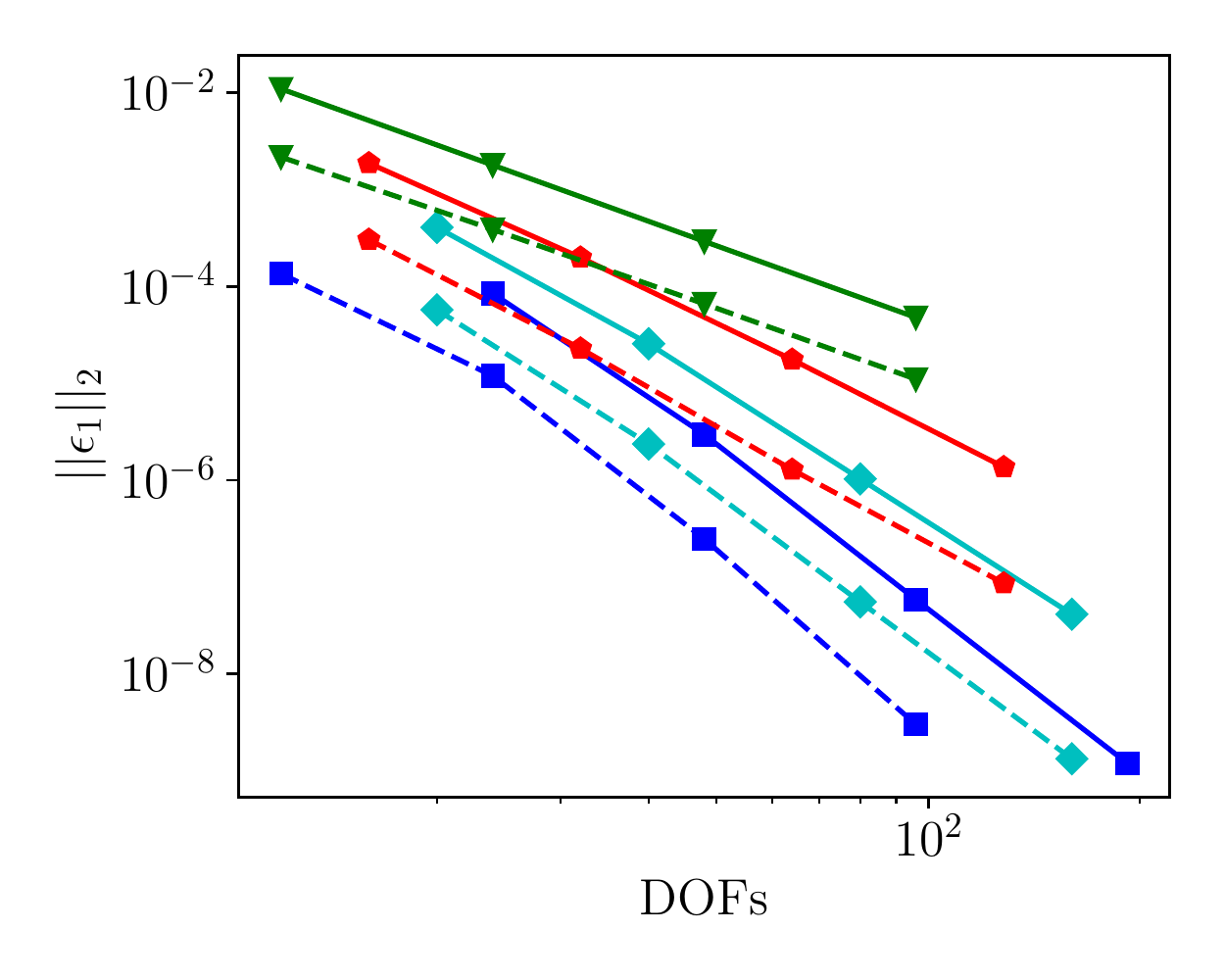}
\includegraphics[trim=0 0 0 0,clip,width=0.32\linewidth]{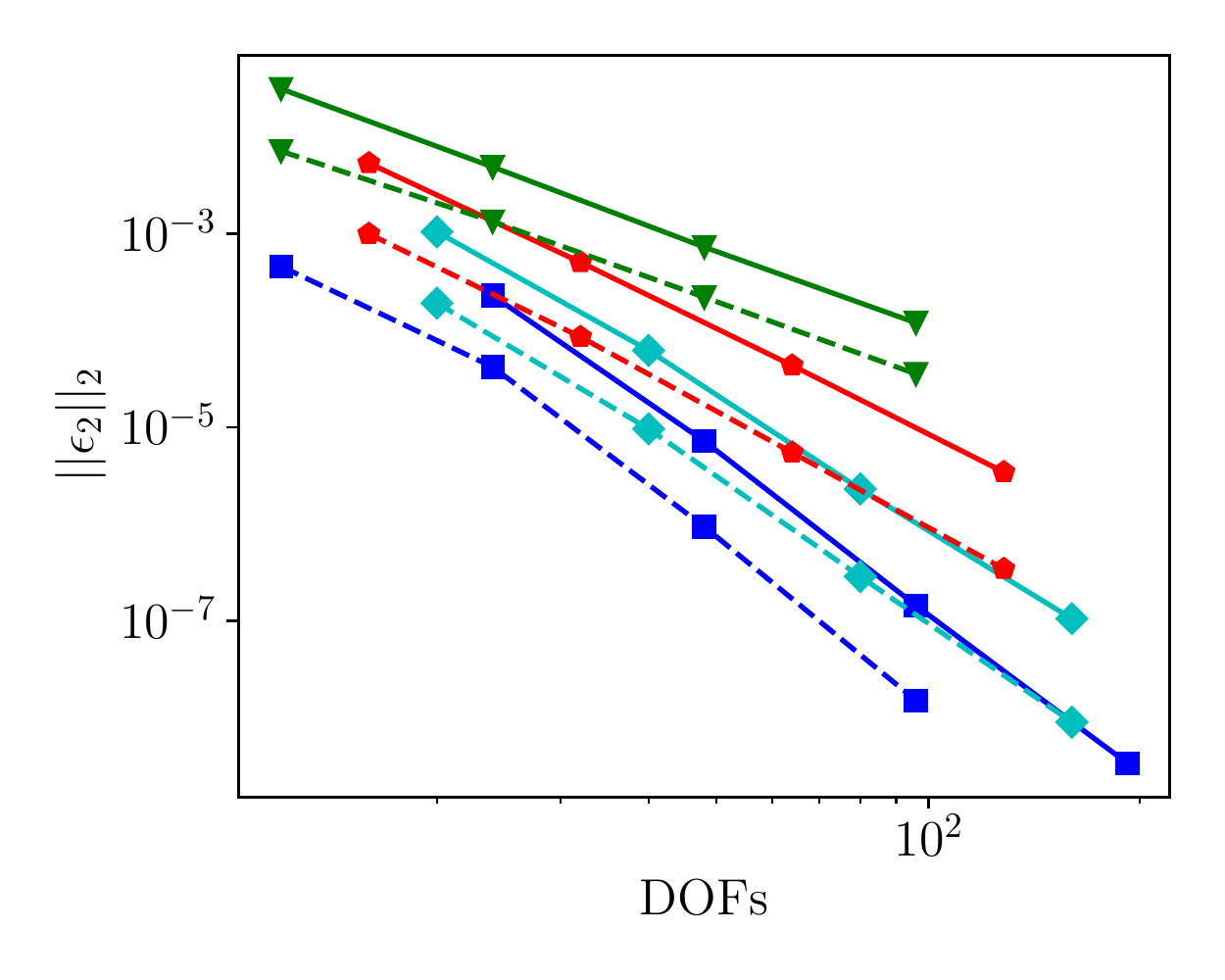}
\includegraphics[trim=0 0 0 0,clip,width=0.32\linewidth]{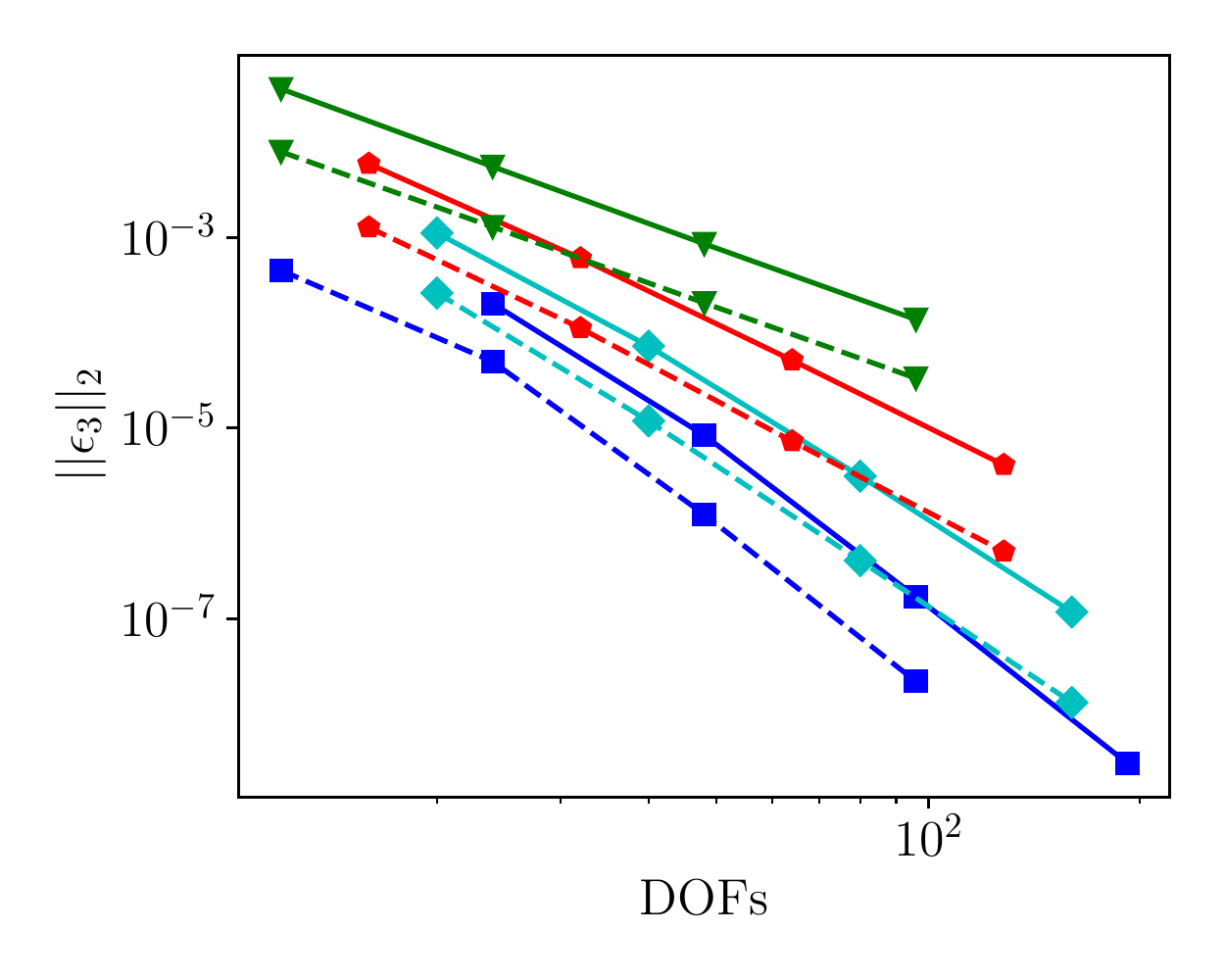}
\includegraphics[trim=0 0 0 0,clip,width=0.32\linewidth]{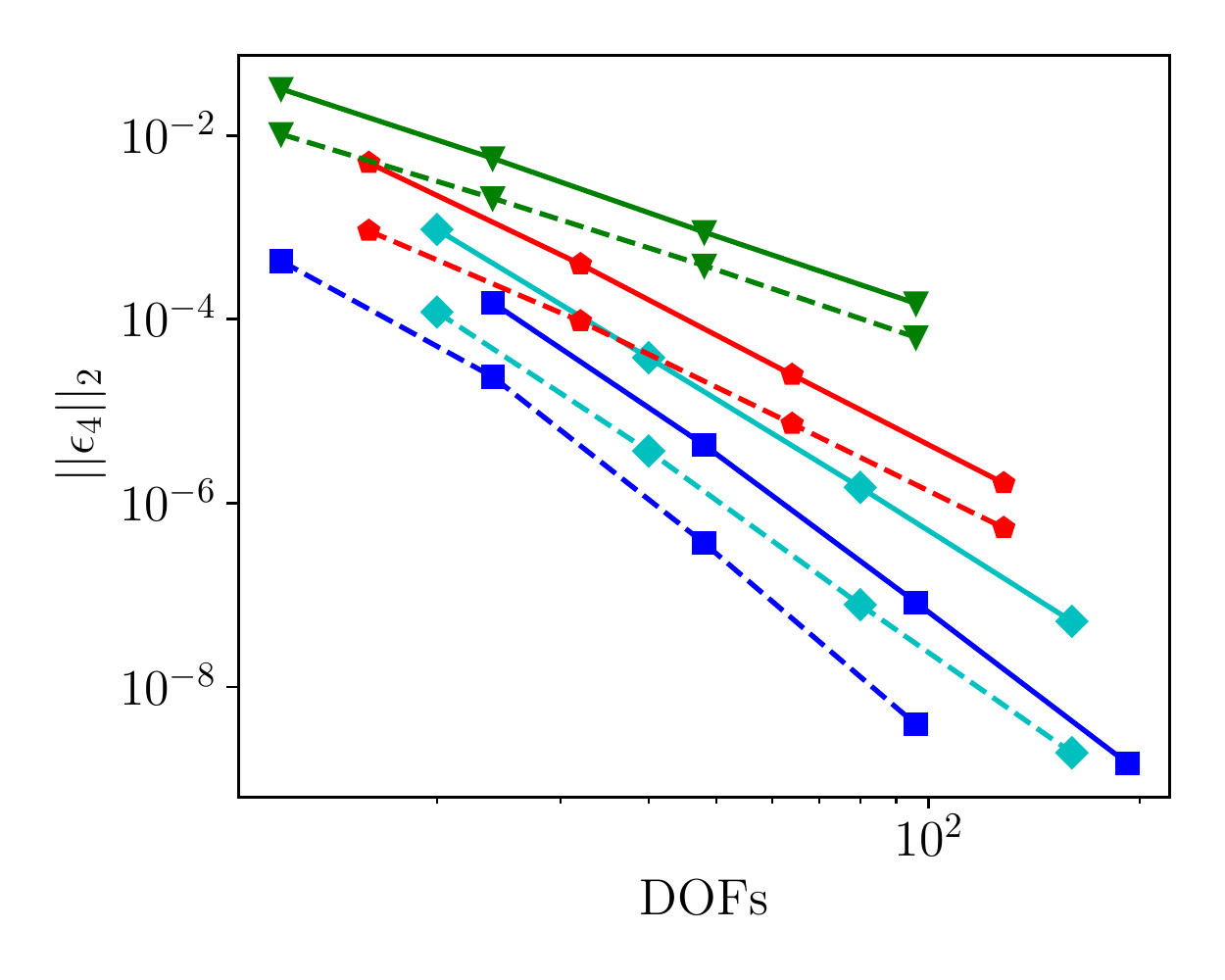}
\includegraphics[trim=0 0 0 0,clip,width=0.32\linewidth]{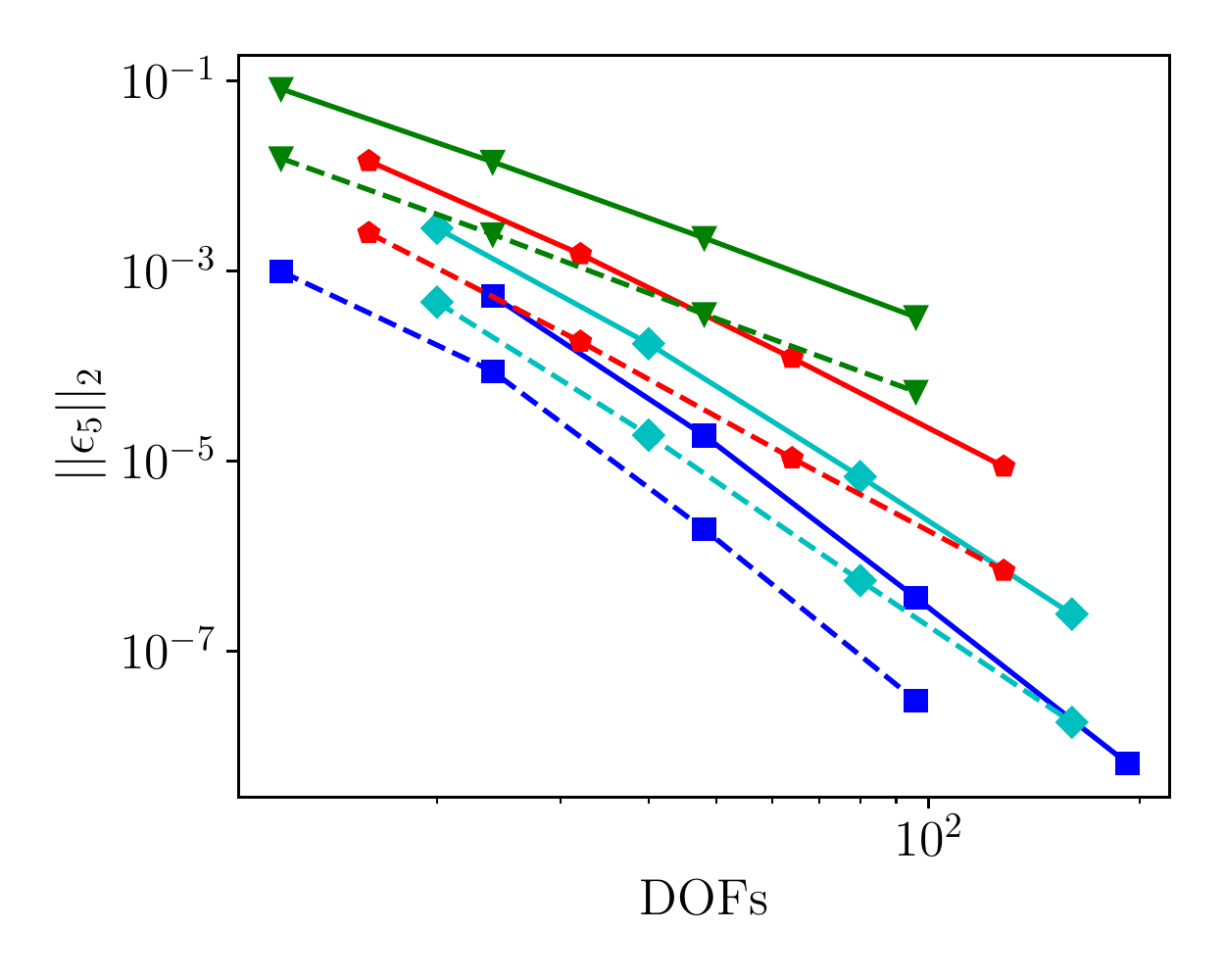}
\includegraphics[trim=0 0 0 0,clip,width=0.32\linewidth]{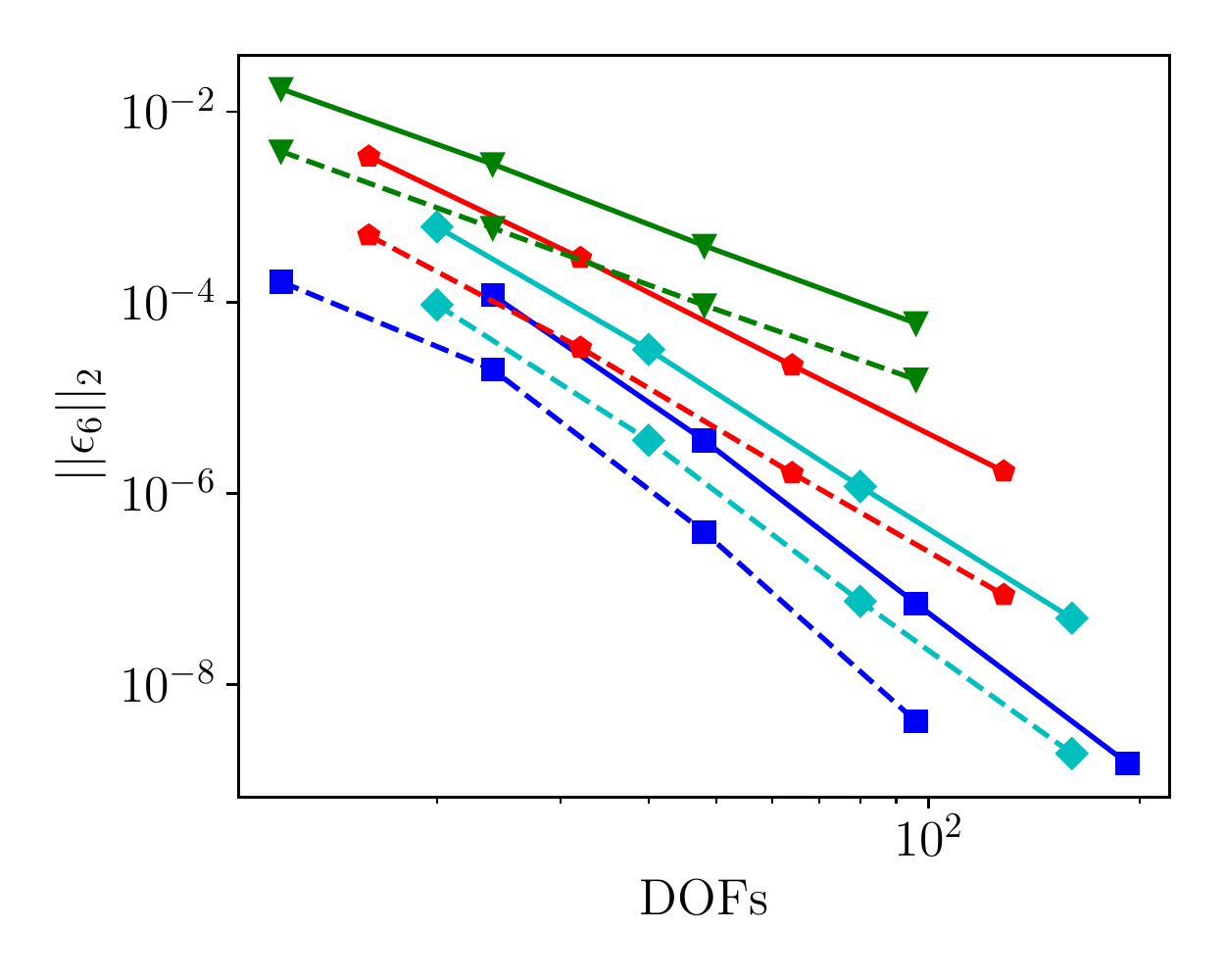}
\includegraphics[trim=0 0 0 0,clip,width=0.32\linewidth]{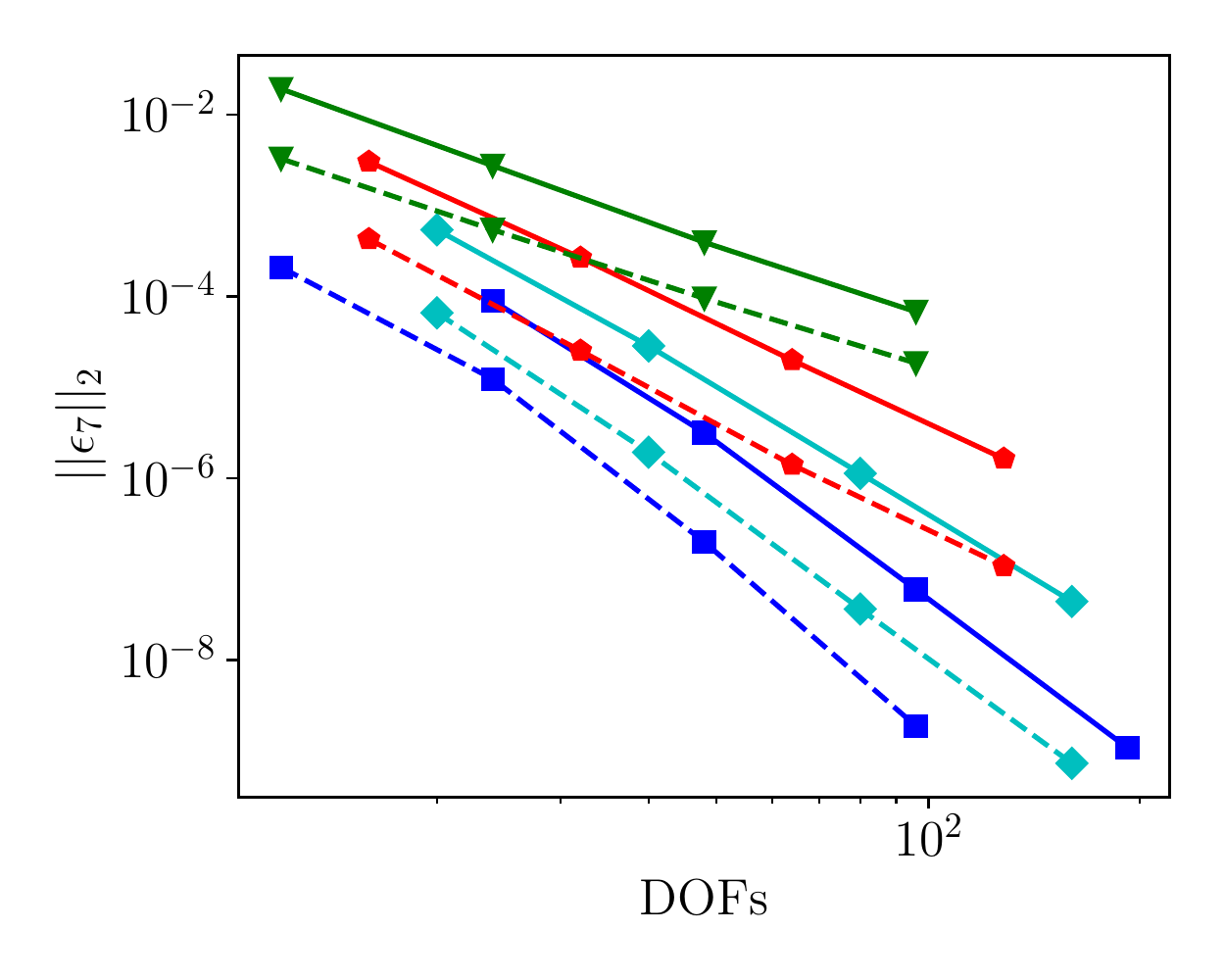}
\includegraphics[trim=0 0 0 0,clip,width=0.32\linewidth]{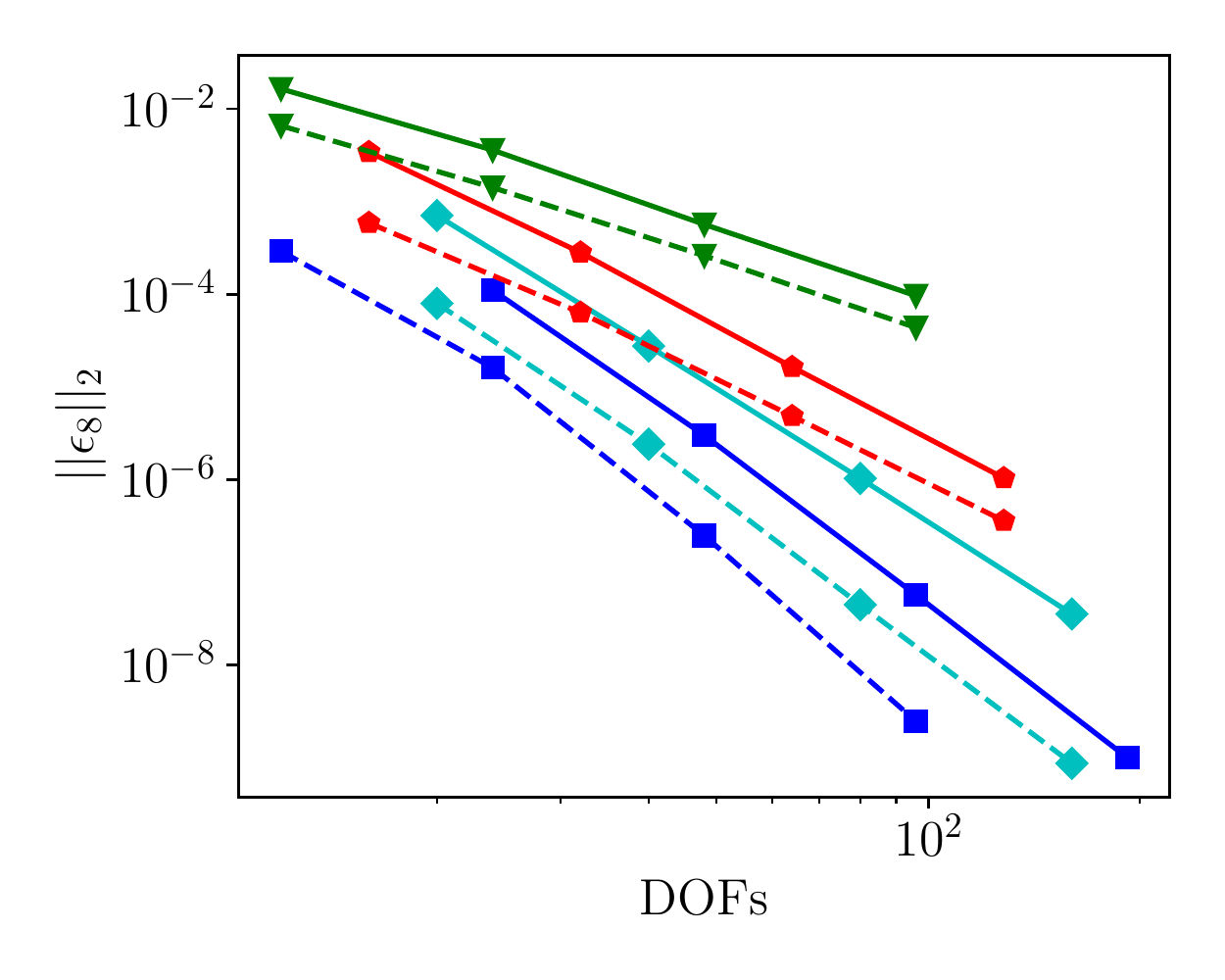}
\includegraphics[trim=0 0 0 0,clip,width=0.32\linewidth]{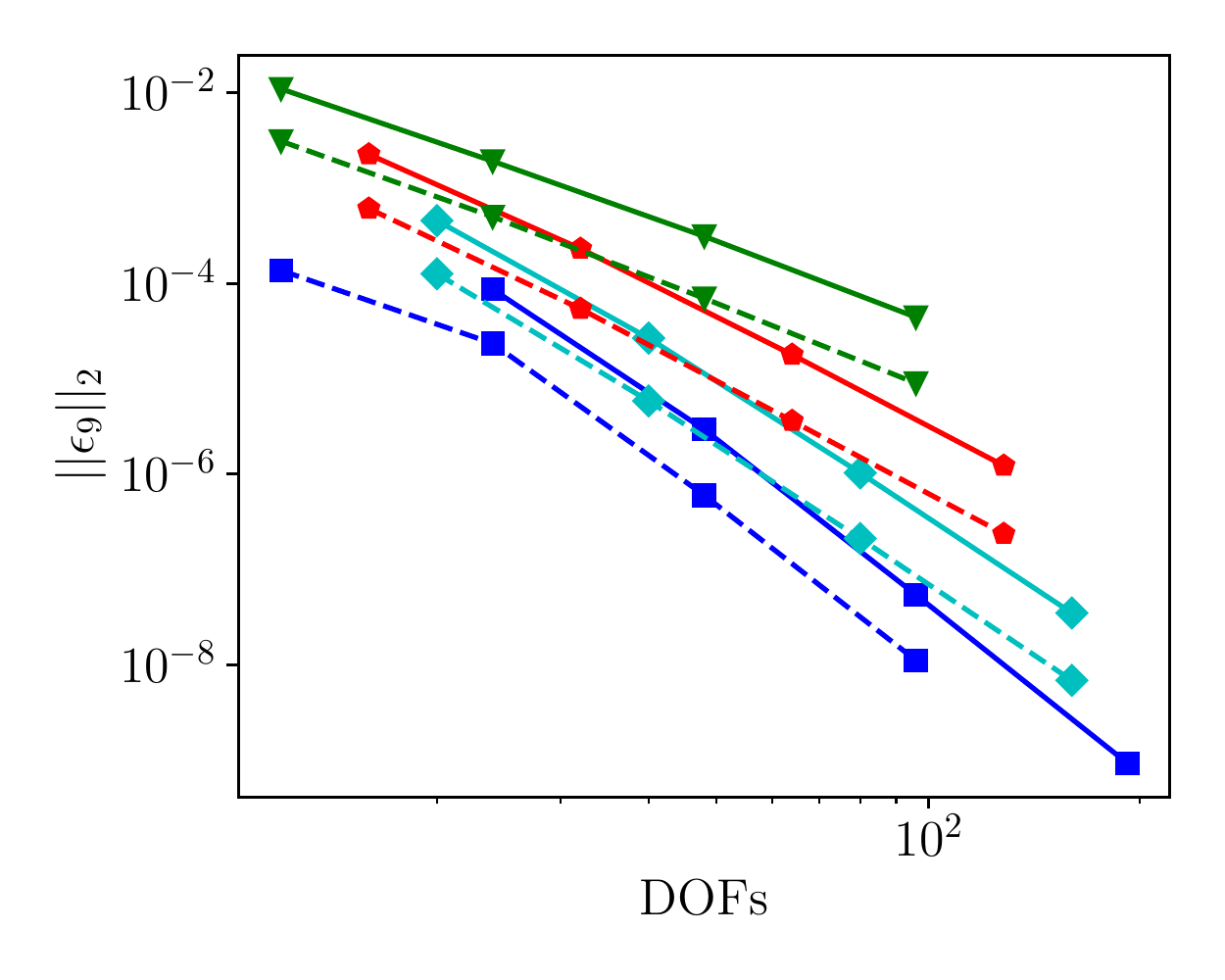}

\caption{Convergence results for each variable using the ES Gauss- and LGL-DGSEM to solve the 3D manufactured solutions test on a heavily-warped mesh.
}
\label{fig:EOCs}
\end{figure}

Tables~\ref{tab:EOC_LGL} and \ref{tab:EOC_Gauss} show the average experimental order of convergence (EOC) for each of the state quantities of the LGL and Gauss scheme, respectively.
The average EOC is computed with the error of the three simulations with the highest number of degrees of freedom from Figure \ref{fig:EOCs}.
Both the LGL- and Gauss-DSEM exhibit a formal EOC of around $N+1$, with small deviations due to the heavily warped mesh. 

\begin{table}[]
    \centering
    \begin{tabular}{c|ccccccccc}
        $N$ & $\rho$ & $\rho v_1$ & $\rho v_2$ & $\rho v_3$ & $\rho E$ & $B_1$  & $B_2$  & $B_3$ & $\psi$ \\
        \hline
         $2$ & $2.62$ & $2.68$ & $2.66$ & $2.62$ & $2.71$ & $2.77$ & $2.67$ & $2.61$ & $2.72$ \\
$3$ & $3.60$ & $3.60$ & $3.60$ & $3.96$ & $3.71$ & $3.71$ & $3.67$ & $4.04$ & $3.78$ \\
$4$ & $4.64$ & $4.61$ & $4.63$ & $4.76$ & $4.73$ & $4.67$ & $4.67$ & $4.80$ & $4.79$ \\
$5$ & $5.64$ & $5.53$ & $5.71$ & $5.75$ & $5.73$ & $5.62$ & $5.76$ & $5.78$ & $5.82$ \\
    \end{tabular}
    \caption{Average experimental order of convergence for the state quantities in the $\mathbb{L}_2$ norm for the manufactured solution test using the LGL-DGSEM scheme on a heavily warped mesh.
    The average is computed with the error of the three simulations with the highest number of degrees of freedom.}
    \label{tab:EOC_LGL}
%\end{table}
% Better to have both tables together!
%\begin{table}[]
    \centering
    \begin{tabular}{c|ccccccccc}
        $N$ & $\rho$ & $\rho v_1$ & $\rho v_2$ & $\rho v_3$ & $\rho E$ & $B_1$  & $B_2$  & $B_3$ & $\psi$ \\
        \hline
         $2$ & $2.57$ & $2.62$ & $2.64$ & $2.51$ & $2.75$ & $2.64$ & $2.44$ & $2.51$ & $2.90$ \\
$3$ & $4.03$ & $3.98$ & $3.89$ & $3.73$ & $4.01$ & $4.30$ & $3.93$ & $3.73$ & $3.92$ \\
$4$ & $5.40$ & $5.04$ & $4.91$ & $5.45$ & $5.02$ & $5.45$ & $5.68$ & $5.72$ & $4.87$ \\
$5$ & $5.98$ & $5.72$ & $5.57$ & $6.28$ & $5.76$ & $6.12$ & $6.34$ & $6.33$ & $5.52$ \\
    \end{tabular}
    \caption{Average experimental order of convergence for the state quantities in the $\mathbb{L}_2$ norm for the manufactured solution test using the Gauss-DGSEM scheme on a heavily warped mesh.
    The average is computed with the error of the three simulations with the highest number of degrees of freedom.}
    \label{tab:EOC_Gauss}
\end{table}

\pagebreak 

\subsubsection{Convergence for Different Magneto-Hydrodynamic Waves}

We now solve the three-dimensional test proposed by \citet{Stone2008}, which measures convergence of errors in the propagation of different linear amplitude magneto-hydrodynamic wave families. 
This is a simple but sensitive test, as each wave is checked separately.

The aim of this problem is to test the accuracy of our entropy-stable Gauss and LGL schemes for each wave family and to compare the results with the ones obtained with the MHD code Athena \cite{Stone2008}.
For this test, we use Athena's constrained-transport second-order finite volume method and the in-built Rusanov flux.
FLUXO was run with the two-point entropy conserving flux of \citet{hindenlang2019new} for the volume numerical fluxes, and an entropy stable version of the Rusanov flux \cite{Winters2016} for the surface numerical fluxes (to be able to compare with Athena results).

The initial condition is given by
\begin{equation}
    \state{u}(x,y,z,t=0) = \state{u}_0 + \tilde{\state{u}},
\end{equation}
where $\state{u}_0$ is a uniform medium with 
\begin{align}
    \rho_0 &= 1, & p_0 &= 3/5, & \vec{v}_0 &= \vec{0}, & \textrm{and} \,\,   \vec{B} &= \left( 1,\sqrt{2},1/2 \right)^T.
\end{align}
The quantity $\tilde{\state{u}}$ is a sinusoidal perturbation with amplitude $A=10^{-6}$ at an oblique angle (to test the 3D capabilities of the code) multiplied by the exact eigenfunctions of the fast, slow, Alfvén, and contact waves.

We carry out the computations in the domain $x \in [0,3]$, $y,z \in [0,1.5]$ with polynomial degrees $N=2,3,4,5$ and $2N_d \times N_d \times N_d$ degrees of freedom. 
Athena computations are performed on a Cartesian mesh, while FLUXO simulations are carried out on Cartesian and moderately-warped meshes.
The moderately-warped meshes are obtained from the unit cube, $\Omega = [-1,1]^3$, using the transformation \eqref{eq:warp} with a warping factor $\alpha = 1/16$, the domain lengths $L_x =3$, $ L_y = L_z = 1.5$, and the shift parameter $\tilde{L} = 0.5$.

Figure~\ref{fig:CartesianAthena} shows convergence results obtained with FLUXO for each wave family using different polynomial degrees on Cartesian meshes, and a comparison with the astrophysics code Athena \cite{Stone2008}.
We show the $\mathbb{L}_1$ norm of the discretization error
\begin{equation}
\norm{\epsilon}_{1} = \max_{i \in [1,\ldots,9]} 
\left( \frac{\int^N_{\Omega} |u_i - u_i^{\mathrm{exact}}| \d \vec{x}}{\int^N_{\Omega} \d \vec{x}} \right), ~~ \forall i \in \{1, \ldots, 9\},
\end{equation}
because that is the quantity reported for the Athena code \cite{Stone2008}.

It is clear that our DG methods exhibit high-order convergence for all wave families and that the entropy-stable Gauss scheme is always more accurate than the LGL scheme (down to machine precision) for the same number of degrees of freedom and polynomial degree.

Interestingly, the fifth-order ($N=4$) Gauss-DGSEM method exhibits superconvergence in this particular case, being more accurate than the sixth-order Gauss and LGL methods.

The convergence results on the moderately-warped meshes are reported in Figure~\ref{fig:WarpedAthena}.
We retain high-order convergence on the warped meshes, but the accuracy of most DG approximations deteriorates slightly due to the mesh distortion.
However, both the Gauss and LGL schemes now show superconvergence for $N=4$.

\begin{figure}[htb]
\centering
\includegraphics[trim=70 180 50 50, clip, width=0.35\linewidth] {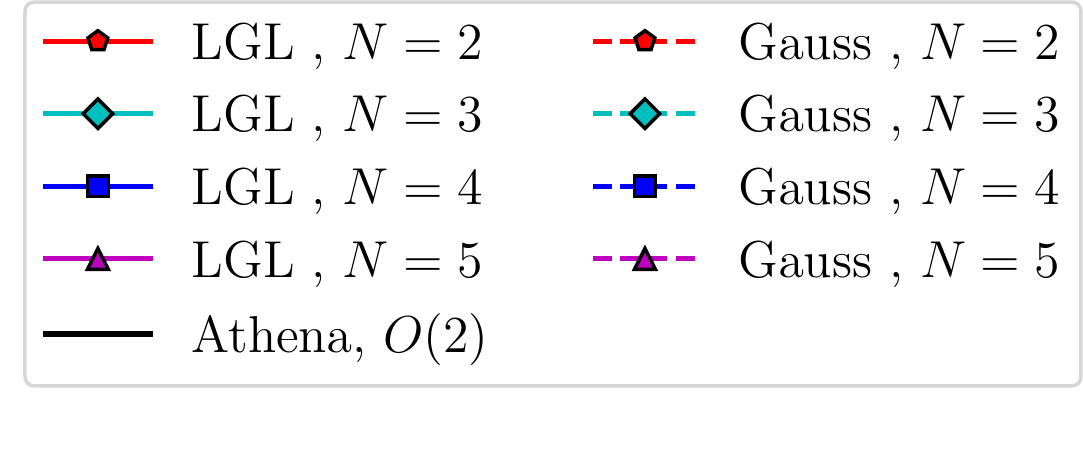}

\begin{subfigure}[b]{0.38\linewidth}
\includegraphics[trim=0 0 0 0, clip, width=\linewidth] {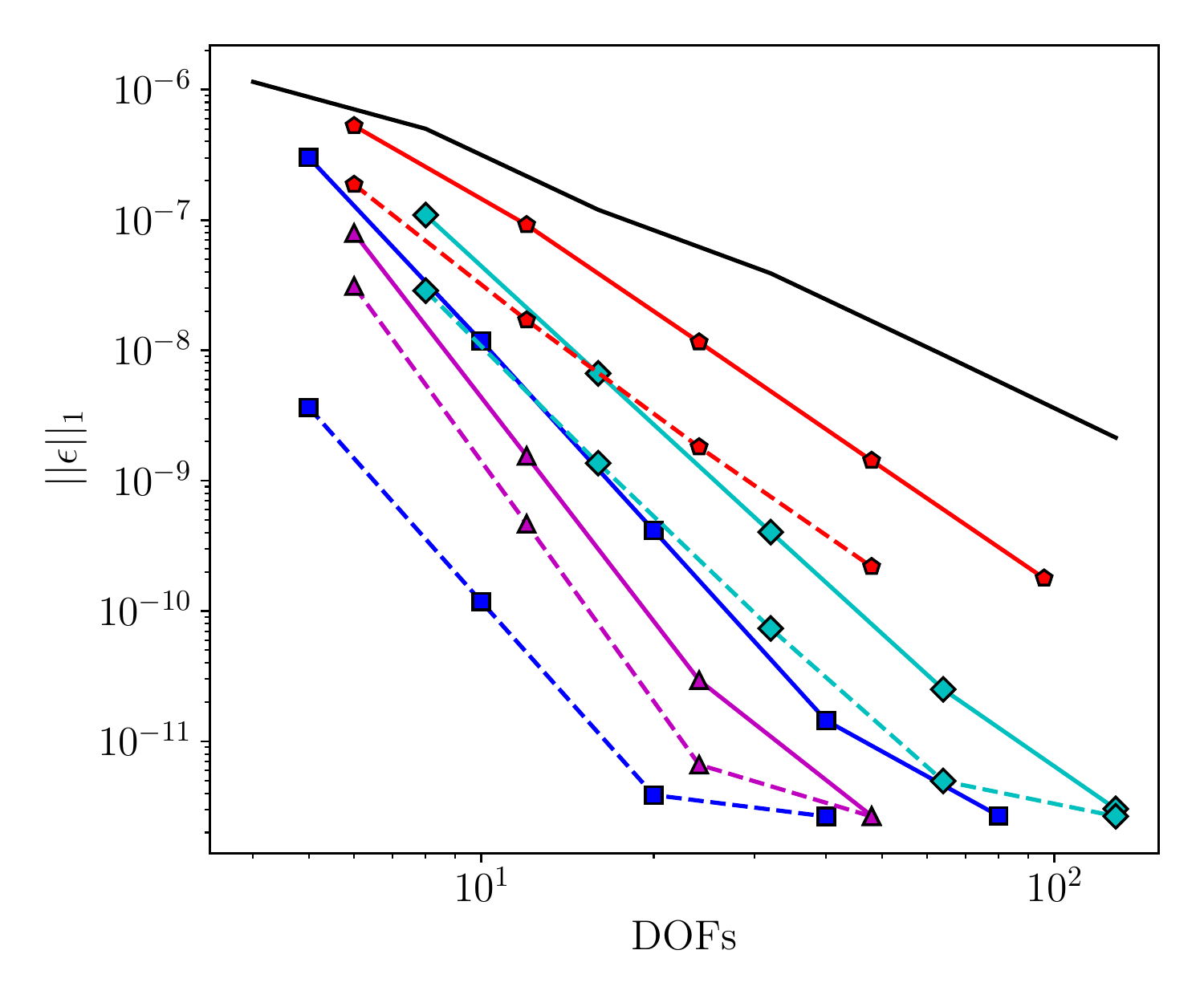}
\caption{Fast magnetosonic wave}
\label{fig:EOC2_cart_fast}
\end{subfigure}
\begin{subfigure}[b]{0.38\linewidth}
\includegraphics[trim=0 0 0 0, clip, width=\linewidth] {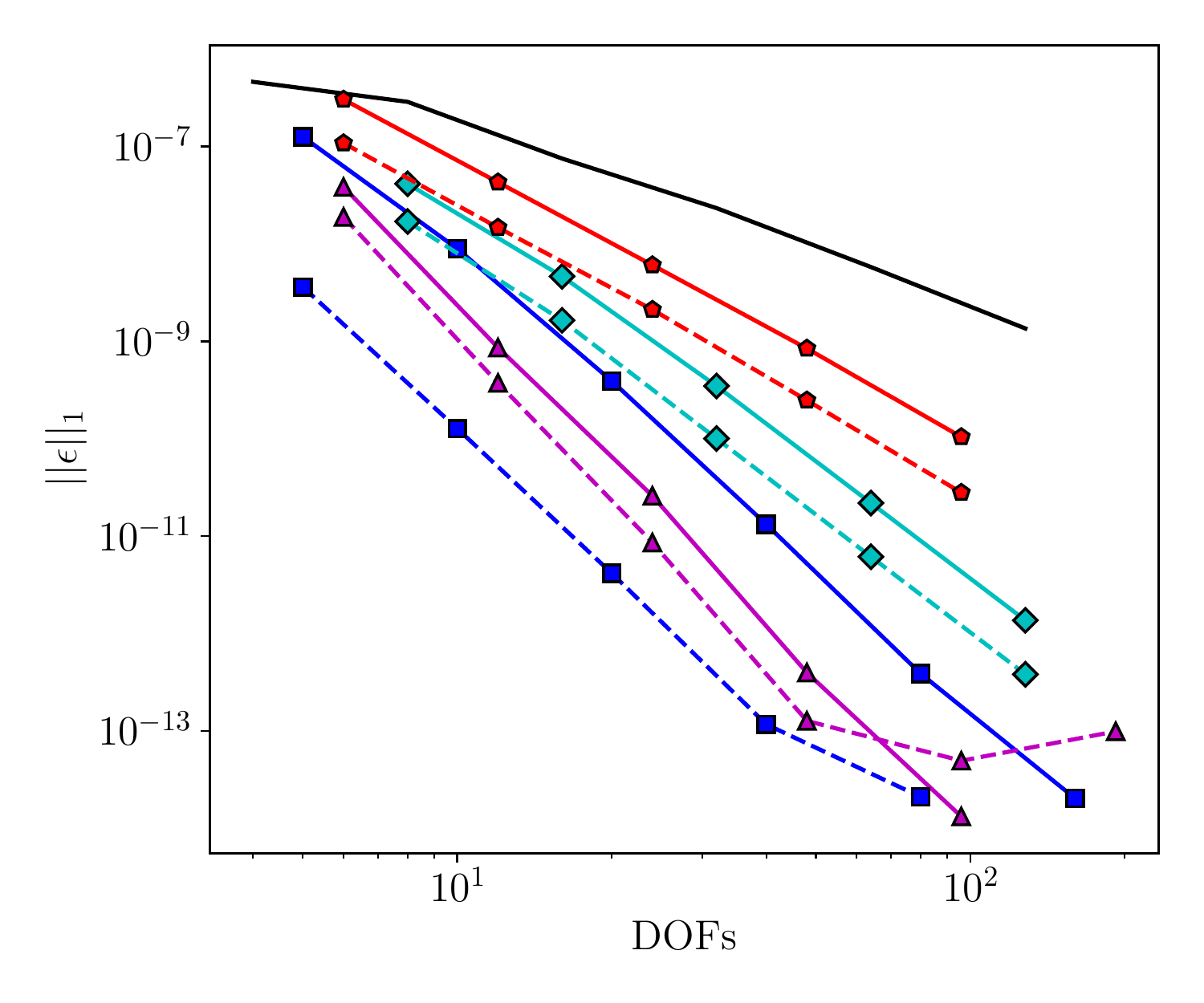}
\caption{Alfvén magnetosonic wave}
\label{fig:EOC2_cart_alfven}
\end{subfigure}
\begin{subfigure}[b]{0.38\linewidth}
\includegraphics[trim=0 0 0 0, clip, width=\linewidth] {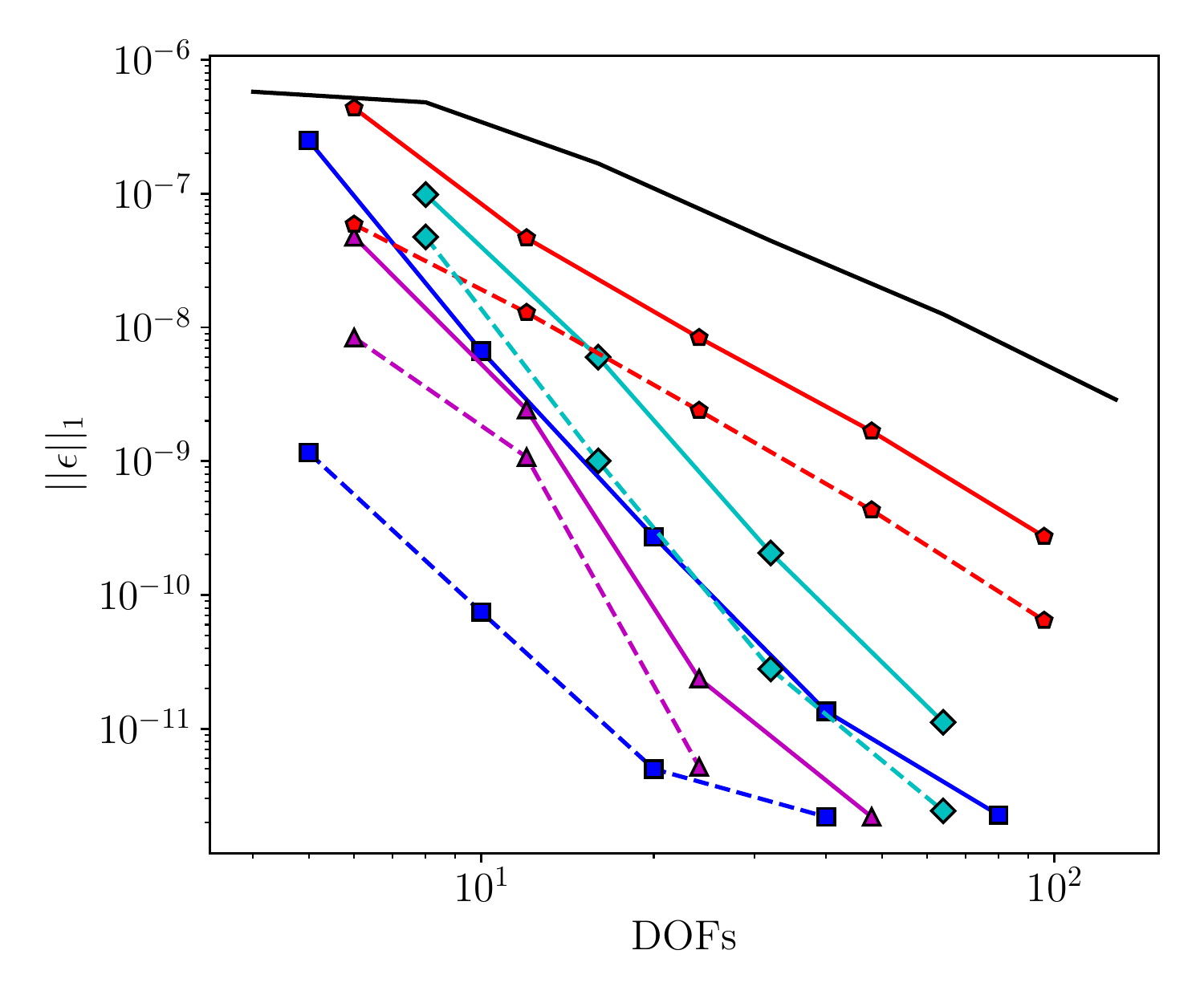}
\caption{Slow magnetosonic wave}
\label{fig:EOC2_cart_slow}
\end{subfigure}
\begin{subfigure}[b]{0.38\linewidth}
\includegraphics[trim=0 0 0 0, clip, width=\linewidth] {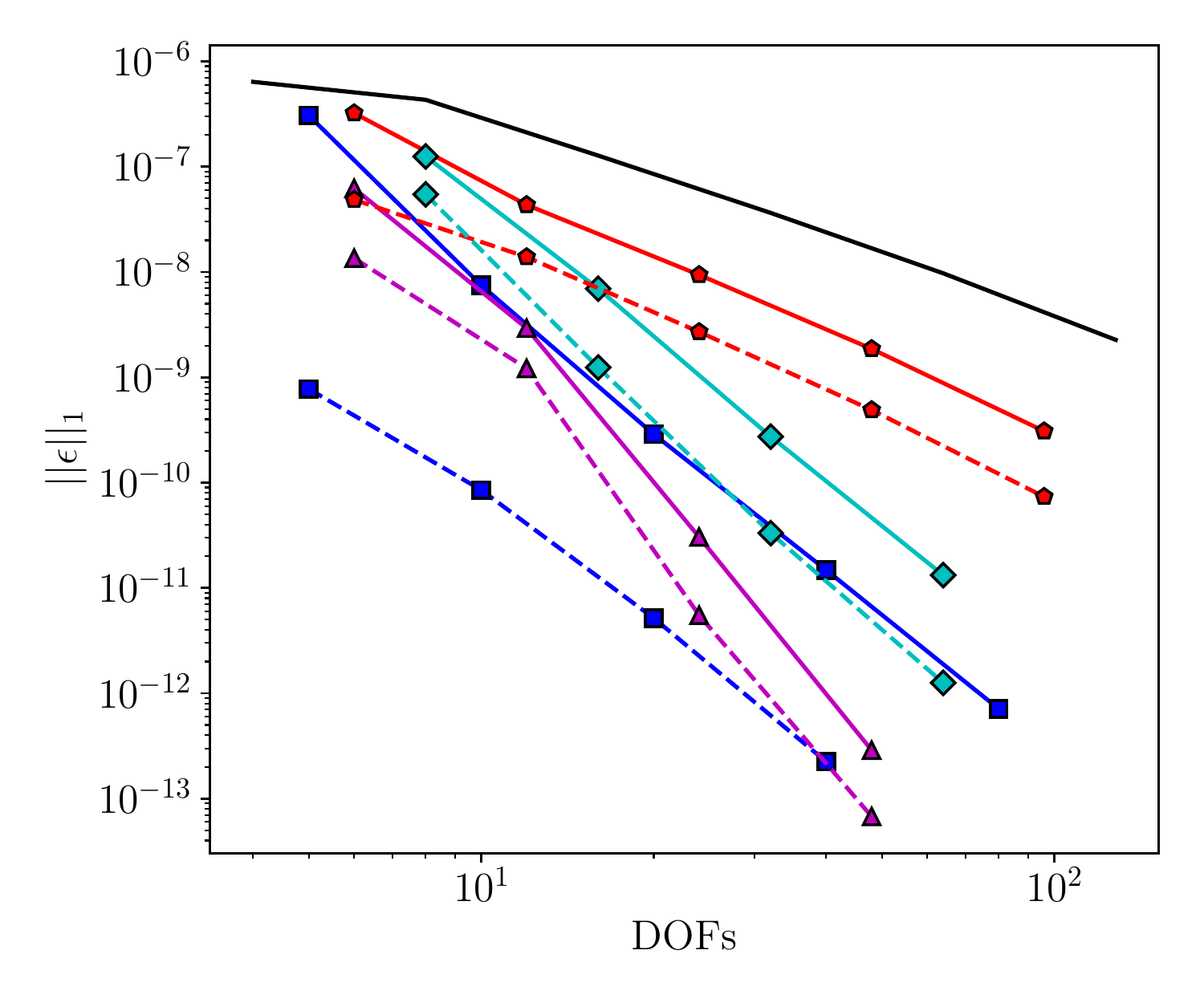}
\caption{Contact wave}
\label{fig:EOC2_cart_contact}
\end{subfigure}
\caption{Convergence results on Cartesian meshes and comparison with Athena (fast, Alfvén, slow, and contact waves). 
Solid lines are LGL and dashed lines are Gauss.}
\label{fig:CartesianAthena}
\end{figure}

\begin{figure}[htb]
\centering
\includegraphics[trim=70 180 50 50, clip, width=0.35\linewidth] {figs/EOC_Athena/Convergence_legend.pdf}

\begin{subfigure}[b]{0.38\linewidth}
\includegraphics[trim=0 0 0 0, clip, width=\linewidth] {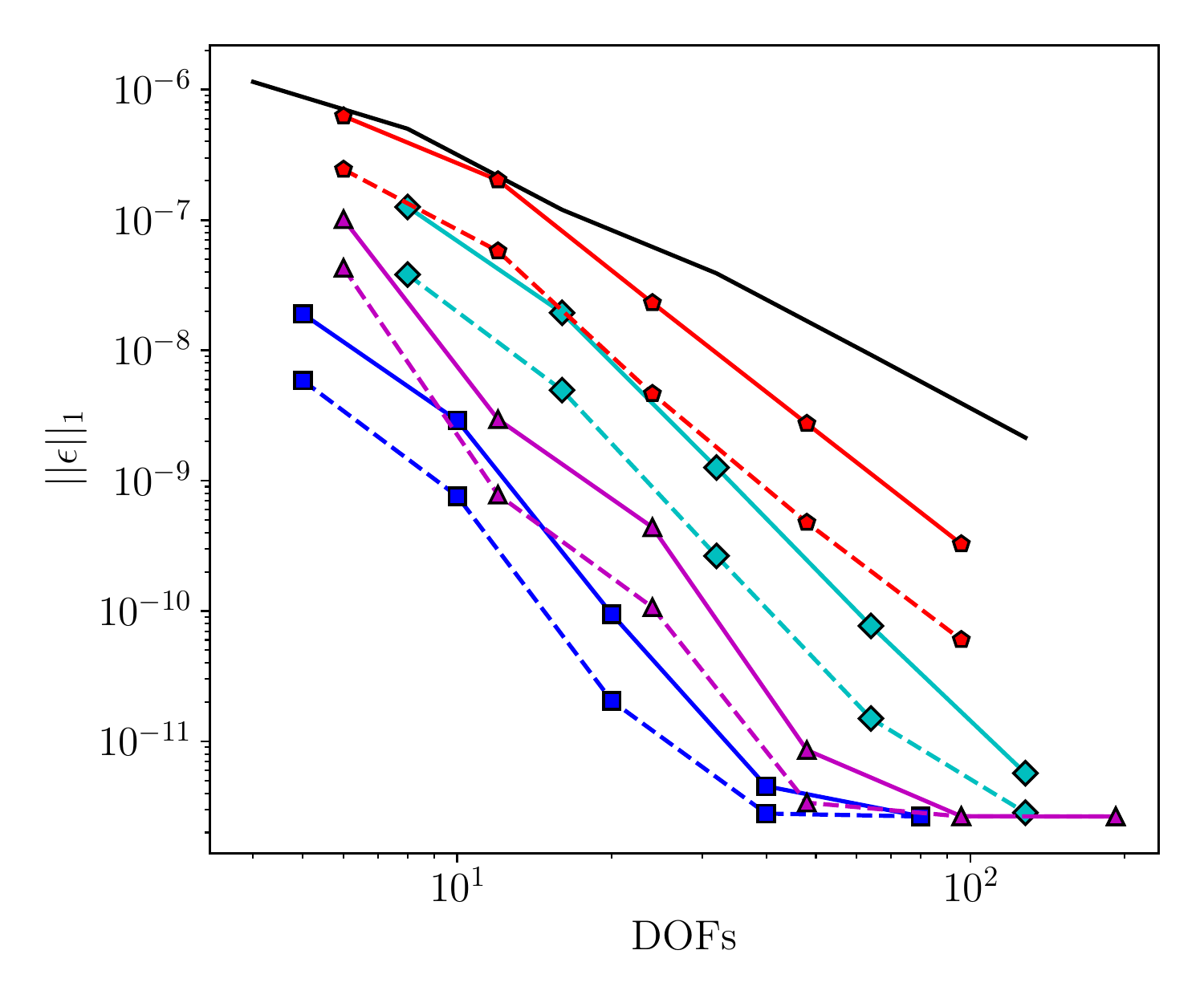}
\caption{Fast magnetosonic wave}
\label{fig:EOC2_warped_fast}
\end{subfigure}
\begin{subfigure}[b]{0.38\linewidth}
\includegraphics[trim=0 0 0 0, clip, width=\linewidth] {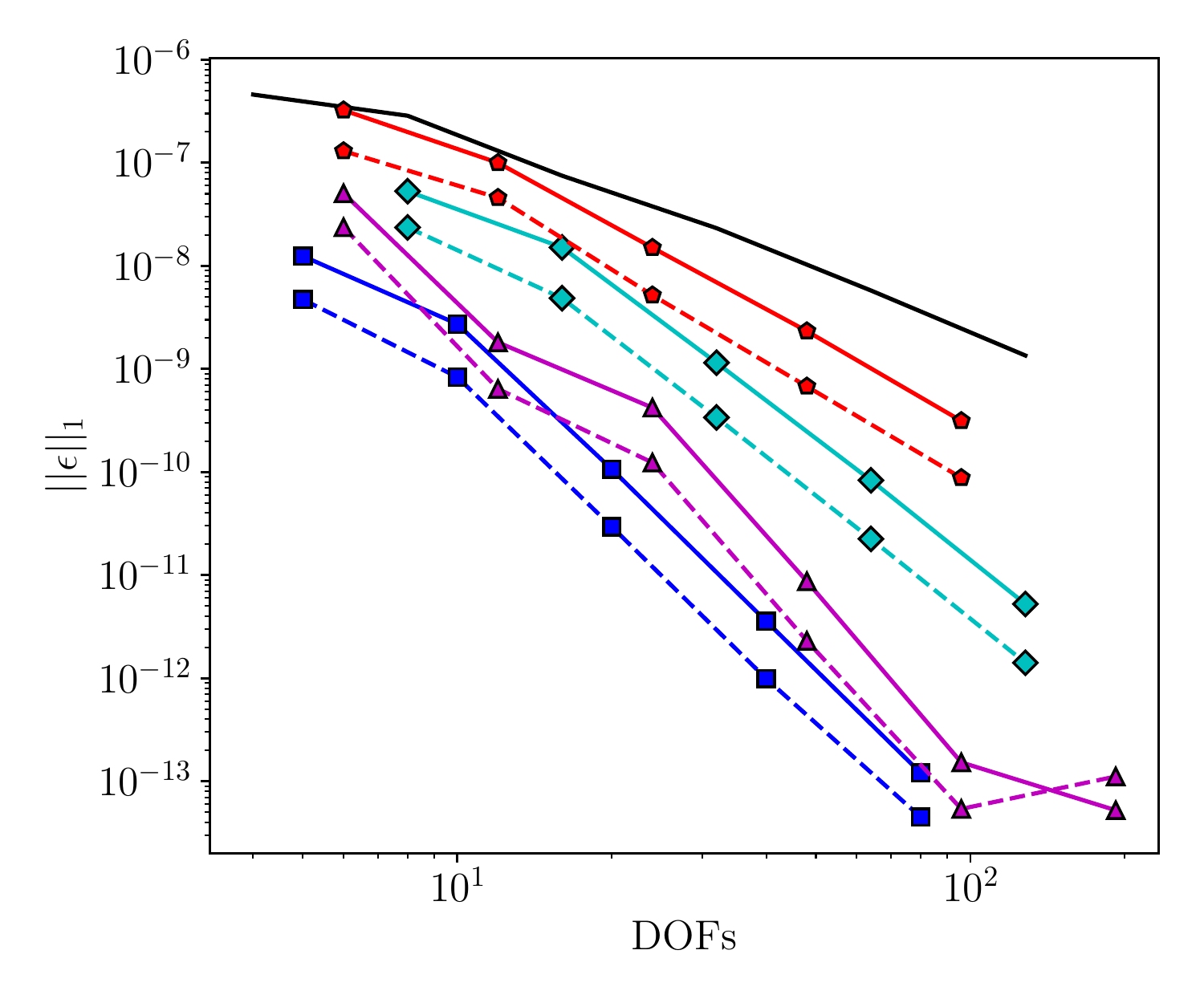}
\caption{Alfvén magnetosonic wave}
\label{fig:EOC2_warped_alfven}
\end{subfigure}
\begin{subfigure}[b]{0.38\linewidth}
\includegraphics[trim=0 0 0 0, clip, width=\linewidth] {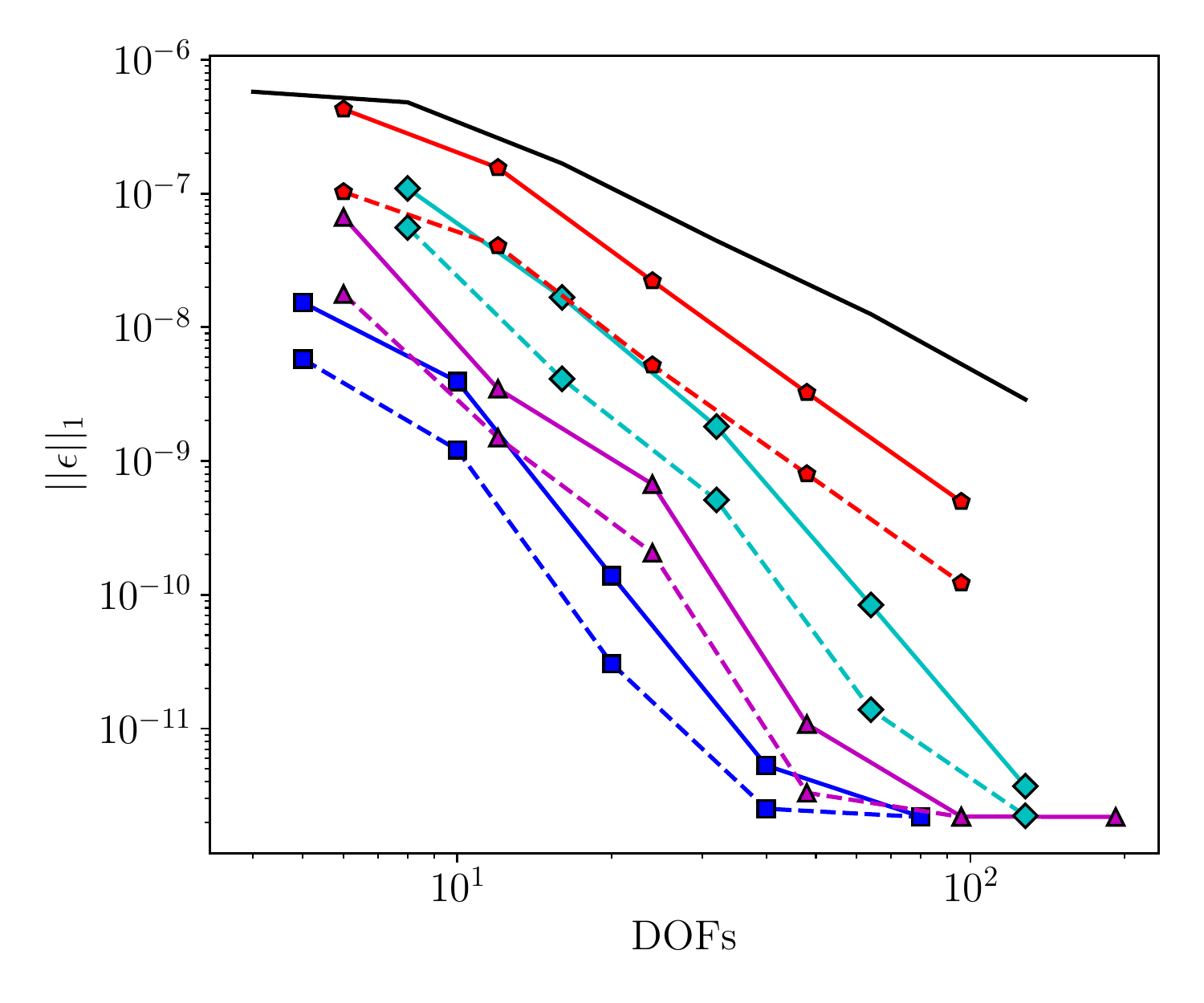}
\caption{Slow magnetosonic wave}
\label{fig:EOC2_warped_slow}
\end{subfigure}
\begin{subfigure}[b]{0.38\linewidth}
\includegraphics[trim=0 0 0 0, clip, width=\linewidth] {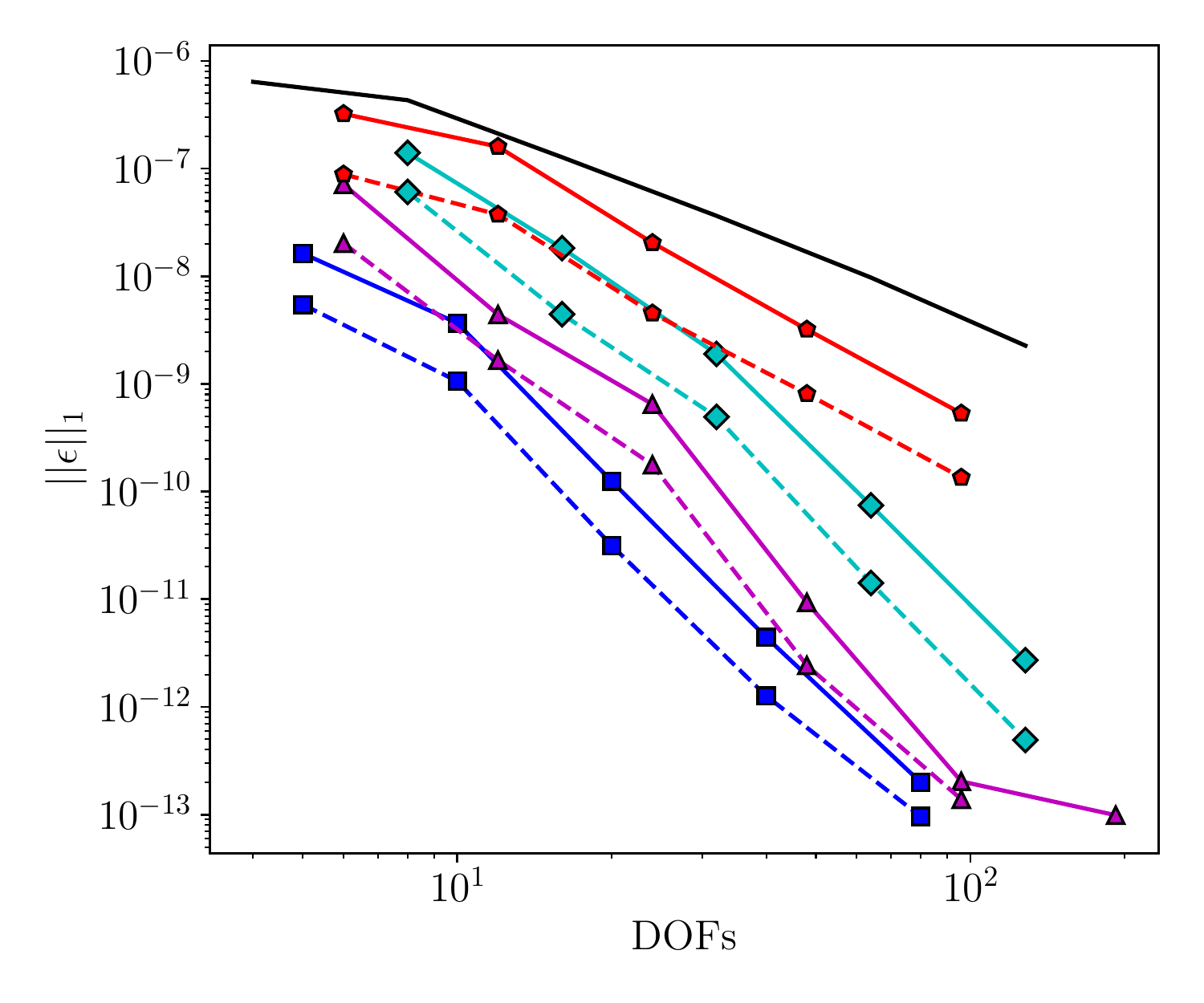}
\caption{Contact wave}
\label{fig:EOC2_warped_contact}
\end{subfigure}
\caption{Convergence results on moderately-warped meshes and comparison with Athena (fast, Alfvén, slow, and contact waves). 
Solid lines are LGL and dashed lines are Gauss.}
\label{fig:WarpedAthena}
\end{figure}

\pagebreak
\subsubsection{Entropy Conservation and Stability}

To test entropy conservation and stability we initialize a weak magnetic blast in the heavily warped domain of Section~\ref{sec:mansol}. 
We use a similar setup as in \cite{Bohm2018,Rueda-Ramirez2020}, where the initial condition is obtained as a blend of two states,
\begin{equation}
\state{u}(t=0) = \frac{\state{u}_{\mathrm{inner}}+ \lambda \state{u}_{\mathrm{outer}}}{1+\lambda}, \ \
\lambda = \exp \left[ \frac{5}{\delta_0} \left( r - r_0 \right) \right], \ \
r = \norm{\vec{x} - \vec{x}_c},
\end{equation}
where $\vec{x}_c = (0,0,0)^T$ is the center of the blast, $r_0 = 0.1$ is the distance to the center of the blast, $\delta_0 = 0.1$ is the approximate distance in which the two states are blended, and the inner and outer states are given in Table \ref{tab:statesBlast}.
We tessellate the mesh using $10^3$ elements and run the simulation until the final time $t=0.4$.

\begin{table}[htbp!]
	\centering
		\begin{tabular}{c|ccccccccc}
			     & $\rho$ 	& $v_1$	& $v_2$	& $v_3$	& $p$ 	& $B_1$	& $B_2$	& $B_3$	& $\psi$ \\\hline
			$\state{u}_{\mathrm{inner}}$ & $1.2$	& $0.1$	& $0.0$	& $0.1$	&$0.9$	& $1.0$	& $1.0$	& $1.0$	& $0.0$ \\
			$\state{u}_{\mathrm{outer}}$ & $1.0$	& $0.2$	& $-0.4$& $0.2$	&$0.3$	& $1.0$	& $1.0$	& $1.0$	& $0.0$ \\\hline 
		\end{tabular}
		\caption{Primitive states for the entropy conservation and entropy stability tests.}
\label{tab:statesBlast}
\end{table}

We report results for entropy-conservative (denoted EC) and entropy-stable (denoted ES) schemes.
The EC scheme uses the entropy-conservative flux of \citet{Derigs2018} for the volume and surface numerical fluxes.
The ES scheme uses the same entropy-conservative flux for the volume fluxes and the nine-waves entropy-stable flux \cite{Derigs2018} for the surface numerical fluxes.

Figure~\ref{fig:Blast} illustrates a slice cut visualization of the heavily warped computational domain with the density and pressure contours at the final time $t=0.4$ for the $N=6$ ES scheme.
The blast explosion gets deformed in the direction of the magnetic field.

\begin{figure}[htb]
\centering
\includegraphics[trim=0 0 0 0, clip, width=0.8\linewidth]{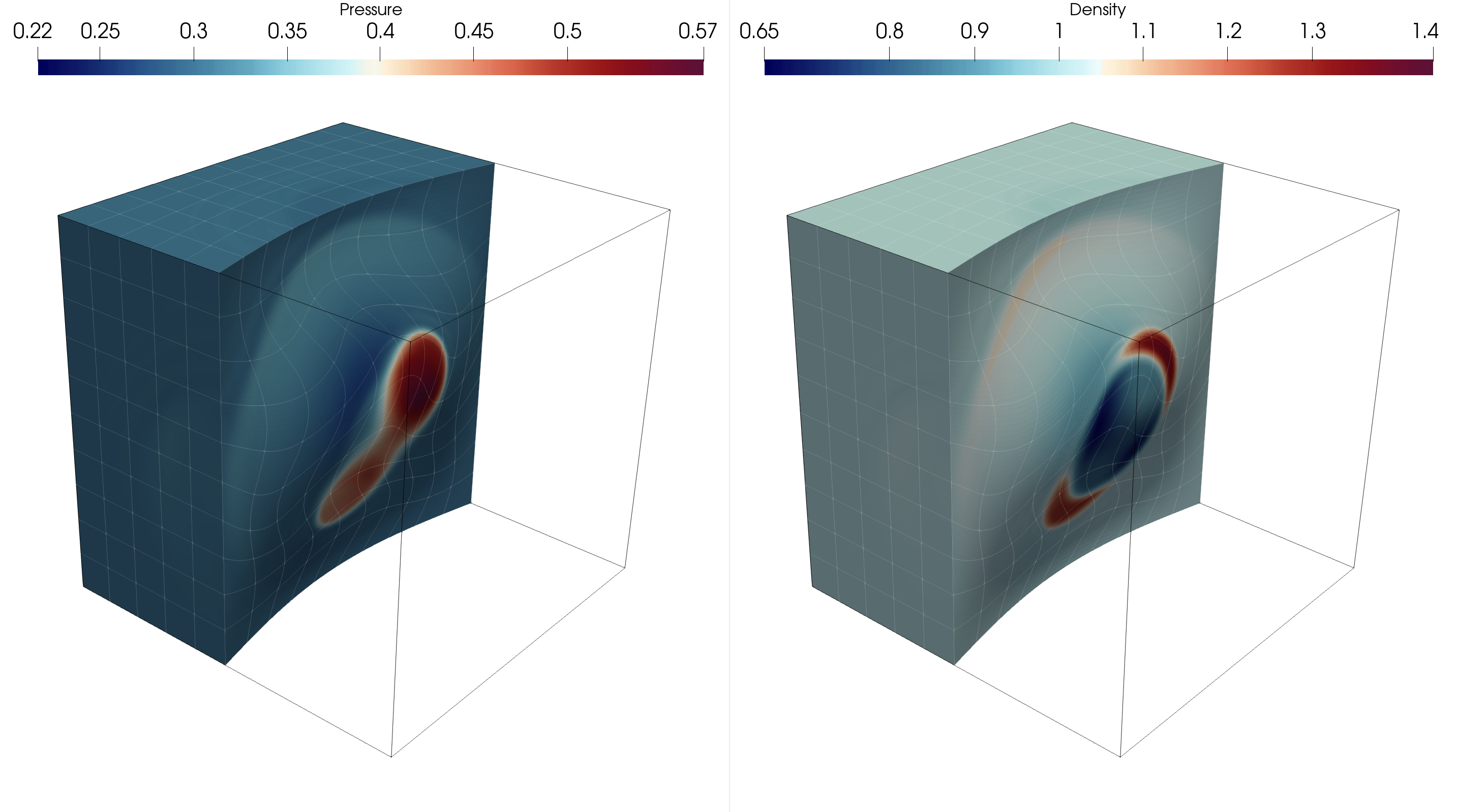}
\caption{Slice cut visualization of the domain showing the heavily warped mesh and the density and pressure distributions at $t=0.4$ for the MHD blast simulation with the $N=6$ ES scheme.}
\label{fig:Blast}
\end{figure}

Figure~\ref{fig:Blast_EC_ES} shows the total entropy change at the end of the simulation for the EC and ES schemes with $N=4$ and $N=6$.
The total entropy in the domain is computed as
\begin{equation}
S_{\Omega} = \int^N_{\Omega} S \d \vec{x}.
\end{equation}

As can be seen in Figure \ref{fig:Blast_EC_ES}, the ES schemes are always entropy-dissipative, and the entropy dissipation of the EC scheme converges to zero with fourth-order accuracy as the time-step size is reduced.
The reason is that, although the EC scheme is semi-discretely entropy conservative, the time-integration scheme adds entropy dissipation depending on the time-step size.

\begin{figure}[htb]
\centering
\includegraphics[trim=0 0 0 0, clip, width=0.55\linewidth]{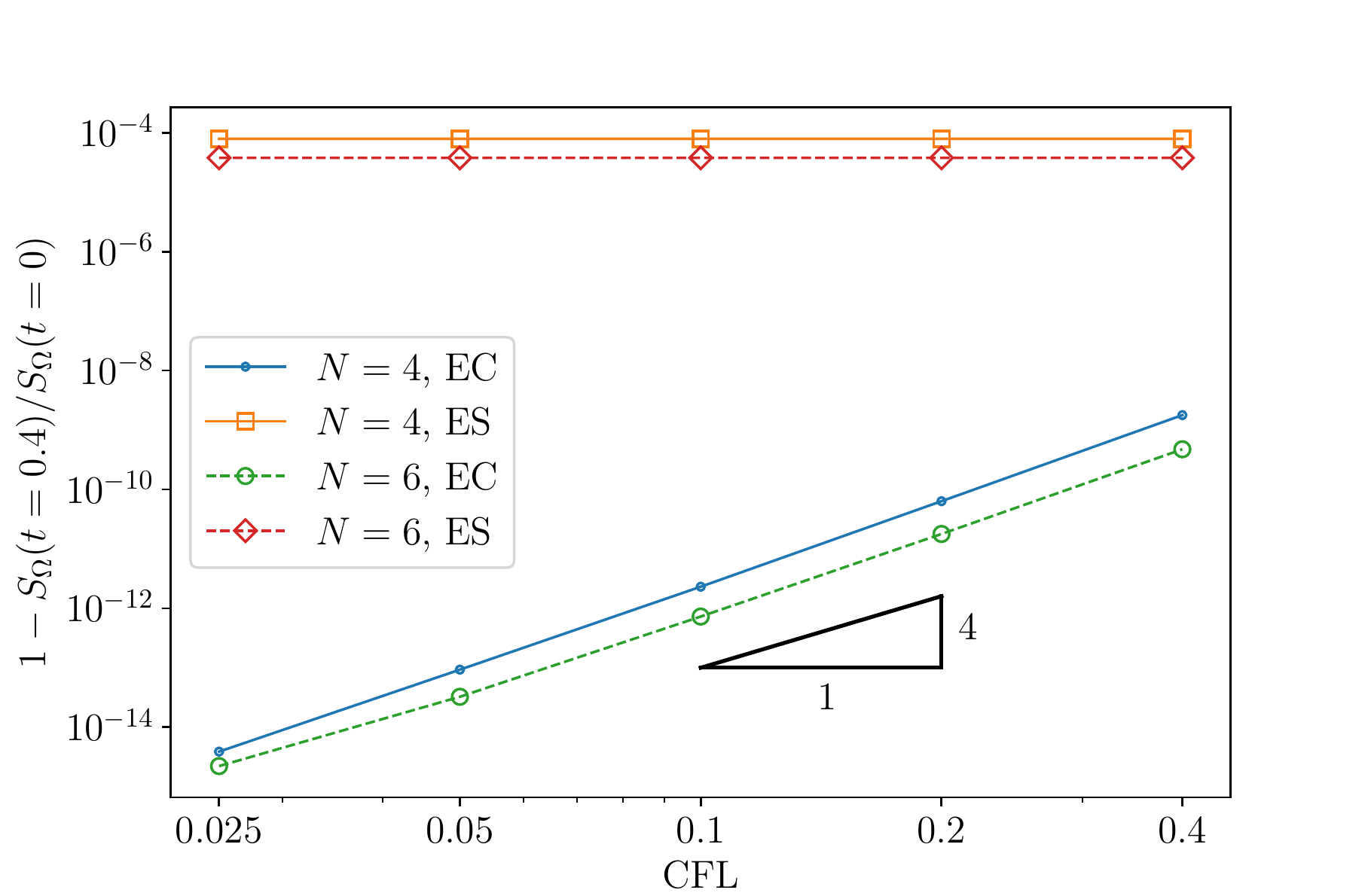}
\caption{Log-log plot of the entropy change from the initial entropy, $S_{\Omega}(t=0)$, to $S_{\Omega}(t=0.4)$ as a function of the CFL number for the EC and ES schemes.}
\label{fig:Blast_EC_ES}
\end{figure}

\subsection{Magnetized Kelvin-Helmholtz Instability}

We consider the magnetized Kelvin-Helmholtz instability problem proposed by \citet{mignone2010high} to study the effect of the grid size and the numerical dissipation of different methods.
The initial condition is given by a shear layer with velocities in the $x$ direction, constant density and pressure, a small single-mode perturbation in the $y$-velocity, and a uniform magnetic field in the $xz$ plane,
\begin{align}
    \rho(x,y,t=0) &= 1, &p(x,y,t=0) &= \frac{1}{\gamma}, &\psi(x,y,t=0) &= 0, \nonumber\\
    v_1 (x,y,t=0) &= \frac{M}{2} \tanh \left( \frac{y}{y_0} \right),
    &v_2(x,y,t=0) &= v_{2,0} \sin (2 \pi x) \exp \left(- \frac{y^2}{\sigma^2} \right), &v_3(x,y,t=0) &= 0, \nonumber \\
    B_1 (x,y,t=0) &= c_a \cos \theta,
    &B_2(x,y,t=0) &= 0, &B_3(x,y,t=0) &= c_a \sin \theta, 
\end{align}
where $\gamma = 5/3$, the Mach number is $M=1$, $y_0=1/20$ is the steepness of the shear, $c_a=0.1$ is the Alfvén speed, and $\theta = \pi/3$ is the angle of the initial magnetic field.
Moreover, the parameters of the perturbation are $v_{2,0}=0.01$ and $\sigma=0.1$.

We carry out the computations with our three-dimensional code using the polynomial degrees $N=3,7$, and the computational domain $x \in [0,1]$, $y \in [-1,1]$ and $z \in [-100,100]$ with $K_x \times 2 K_x \times 1$ Cartesian elements.
We adjust the number of elements in the $x$ direction, $K_x$, to test different resolutions: $64 \times 128$, $128 \times 256$ and $256 \times 512$ degrees of freedom.
The size and number of elements in the $z$ direction is chosen such that the state quantities cannot present significant fluctuations in $z$.
The left, right, front and back boundaries are periodic, whereas the top and bottom boundaries are perfectly conducting slip walls.

We remark that the Lorenz force acts on the $z$ component of the momentum during the simulation due to the non-zero magnetic field in $z$, giving rise to three-dimensional effects in this \textit{pseudo $2D$} example.

For this example, we use the entropy-conservative flux of \citet{Derigs2018} for the volume numerical fluxes and the nine-waves entropy-stable flux \cite{Derigs2018} for the surface numerical fluxes.

Figure~\ref{fig:KHI_contours} shows the evolution of the magnetized Kelvin-Helmholtz instability problem for the fourth- ($N=3$) and eighth-order ($N=7$) DGSEM using Gauss and LGL nodes and the highest resolution that we tested ($256 \times 512$ DOFs).
We show the ratio of the poloidal field, 
\begin{equation}
    B_p := \sqrt{B_1^2 + B_2^2},
\end{equation}
to the toroidal field, $B_t := B_3$, using the same color range as \citet{mignone2010high} for comparison purposes.
At early stages of the simulation ($t \le 5$), the perturbation follows a linear growth phase, in which the cat’s eye vortex structure forms, see e.g. the top row of Figure~\ref{fig:KHI_contours} and \cite{mignone2010high}.
As the simulation continues, magnetic field lines become distorted and energy is transferred to smaller scales in the onset of MHD turbulence.
Energy is then dissipated by the artificial viscosity and resistivity of the methods.

At early stages of the simulation ($t \le 5$), the Gauss- and LGL-DGSEM discretizations show a similar $B_p/B_t$ field distributions, regardless of the polynomial degree.
As the flow field transitions to turbulence ($t \approx 8$), slight differences start to appear between the numerical solutions of the different methods: the eight-order LGL and Gauss, and the fourth-order Gauss methods are remarkably similar, whereas the fourth-order LGL scheme shows fewer small vortical structures.
As turbulence develops in the domain, it is clear that increasing the order of the approximation and switching from LGL to Gauss nodes allows the development of smaller scales.

\begin{figure}[h!]
\newcommand{\tablerow}[1]{
 \includegraphics[trim=335 0 1766 0,clip,height=0.28\linewidth]{figs/KHI/01_N3/LGL_vs_Gauss_t_#1.png}
 &
 \includegraphics[trim=335 0 1766 0,clip,height=0.28\linewidth]{figs/KHI/02_N7/LGL_vs_Gauss_t_#1.png}
 &
 \includegraphics[trim=1766 0 335 0,clip,height=0.28\linewidth]{figs/KHI/01_N3/LGL_vs_Gauss_t_#1.png}
 &
 \includegraphics[trim=1766 0 335 0,clip,height=0.28\linewidth]{figs/KHI/02_N7/LGL_vs_Gauss_t_#1.png}
 &
 \includegraphics[trim=2566 0 0 0,clip,height=0.28\linewidth]{figs/KHI/02_N7/LGL_vs_Gauss_t_#1.png}
 \\
}
\centering

\begin{tabular}{ccccc}
   LGL (N=3) & LGL (N=7) & Gauss (N=3) & Gauss (N=7) & \\
    \tablerow{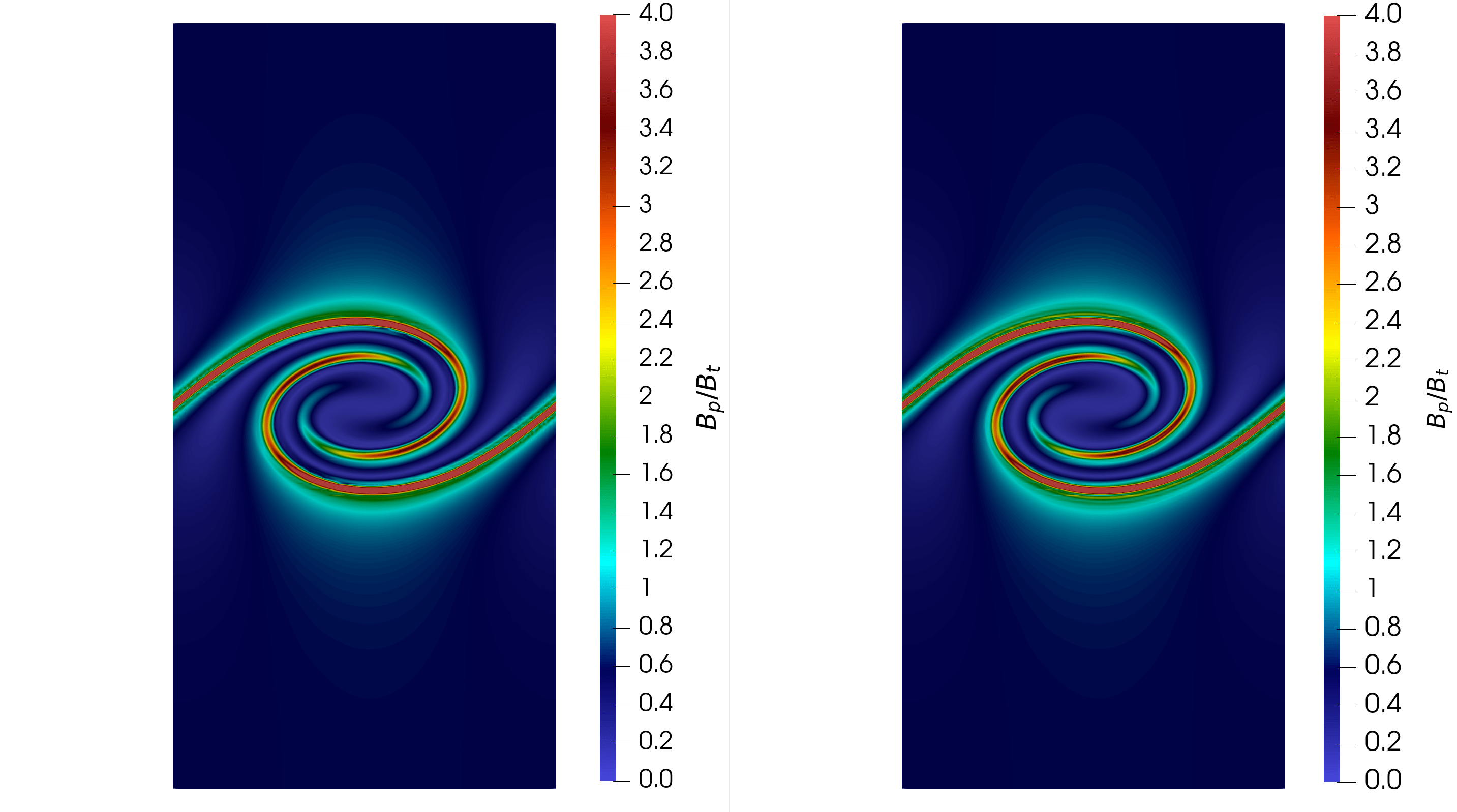}
    \tablerow{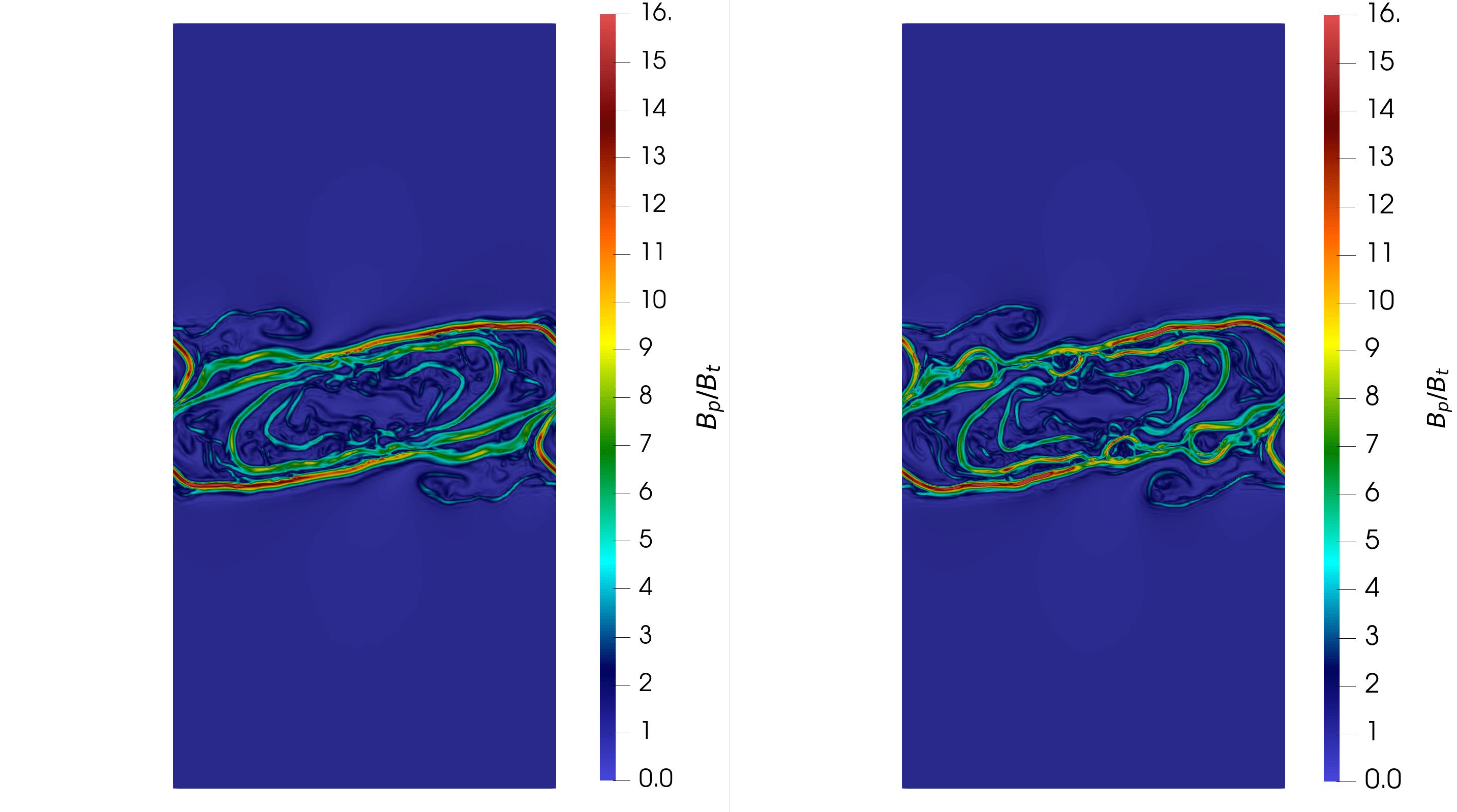}
    \tablerow{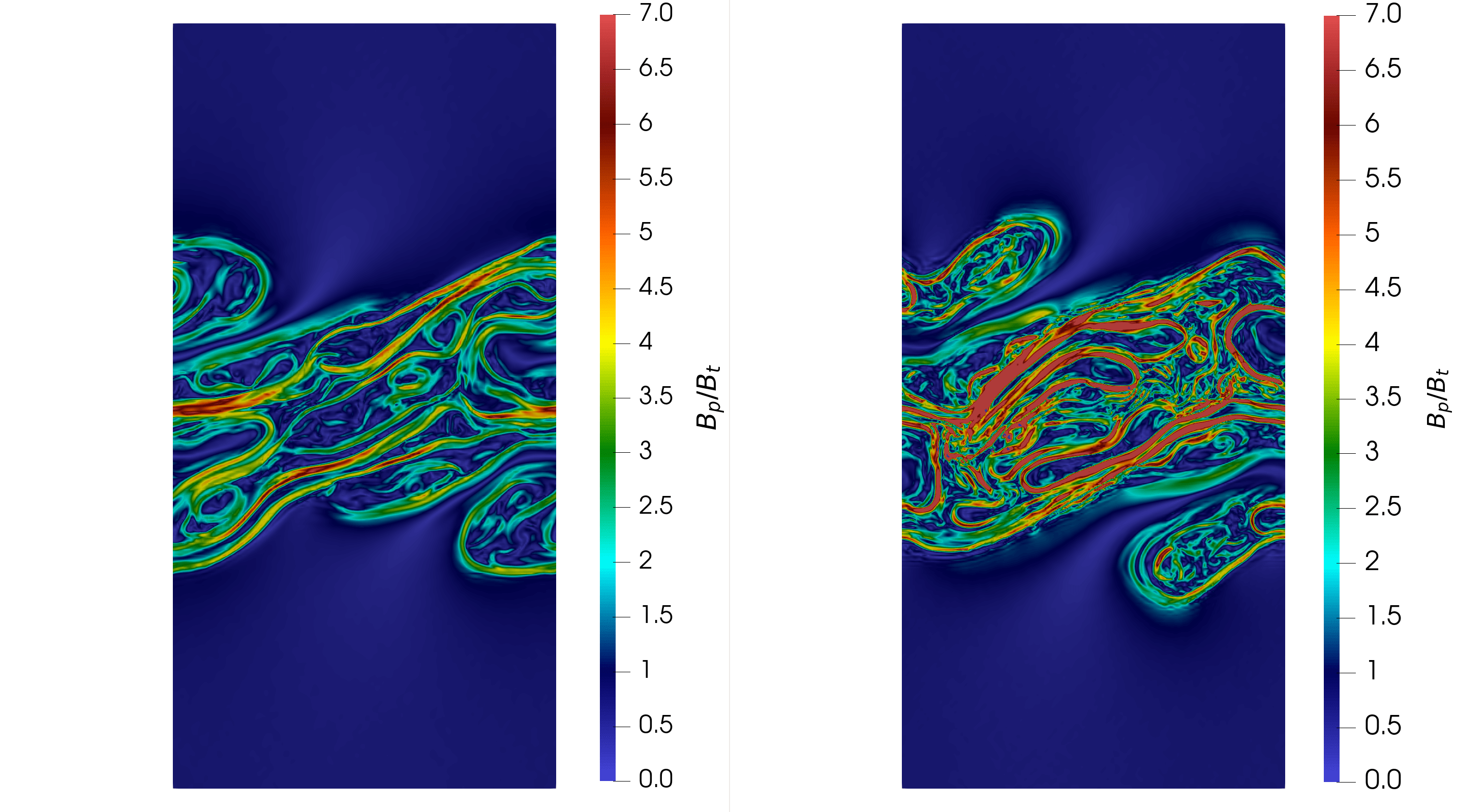}
    \tablerow{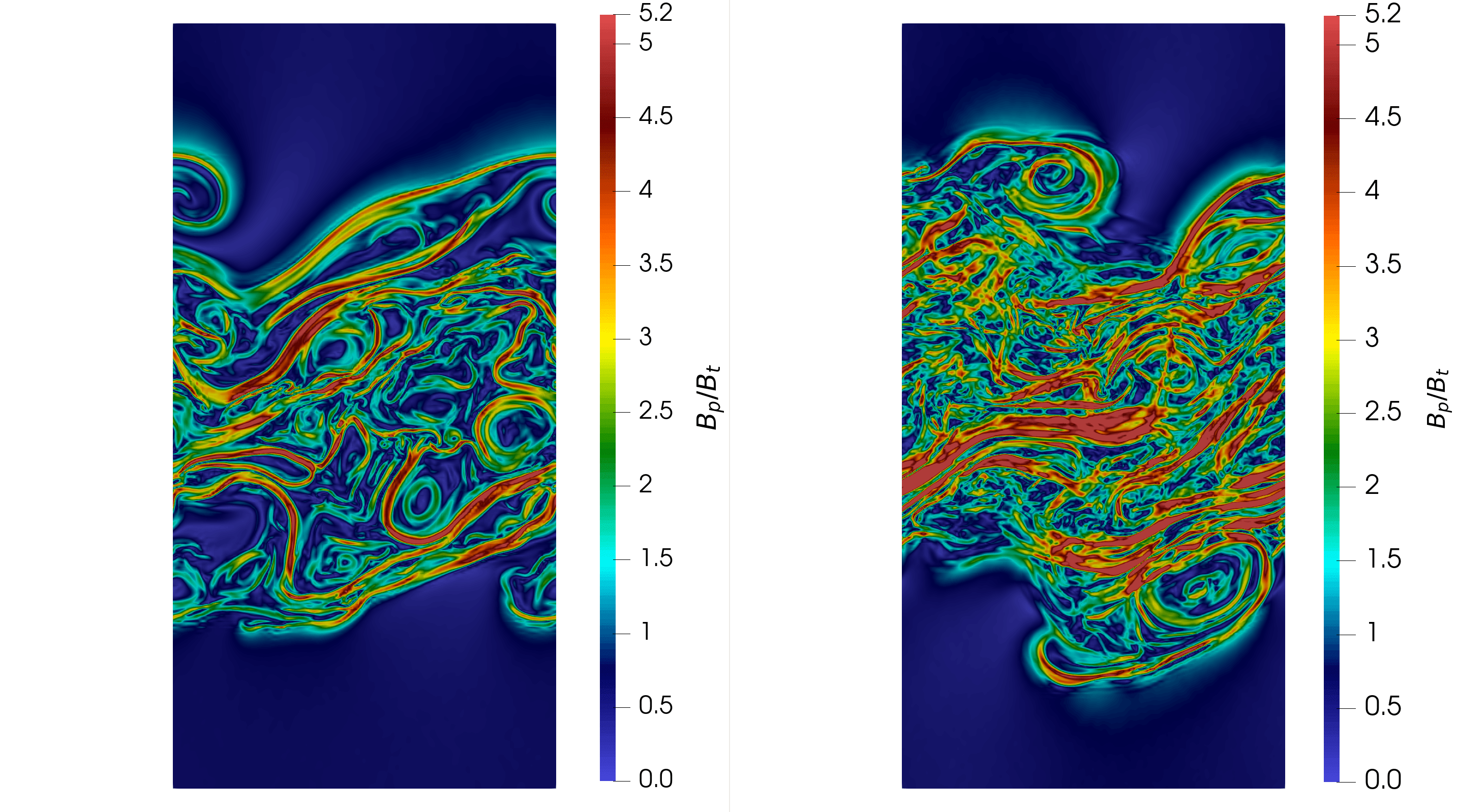}
\end{tabular}
\caption{Evolution of the magnetized Kelvin-Helmholtz instability problem at times $t=5$ (top panel), $t=8$ (second panel), $t=12$ (third panel) and $t=20$ (bottom panel) with the highest resolution ($256 \times 512$ DOFs).
We show the ratio of the poloidal field, $B_p$, to the toroidal field, $B_t := B_3$. 
}
\label{fig:KHI_contours}
\end{figure}

\citet{mignone2010high} proposed the use of the normalized volume-integrated poloidal magnetic energy and the so-called growth rate to compare the numerical dissipation of different schemes and grid resolutions during the transition to MHD turbulence.
The normalized volume-integrated poloidal magnetic energy is defined as
\begin{equation}
    \langle B_p^2 \rangle(t) := \frac{\int_{\Omega} B_p^2(t) \d \vec{x}}{\int_{\Omega} B_p^2(t=0) \d \vec{x}},
\end{equation}
and the growth rate is defined as
\begin{equation}
    \Delta v_y(t) := \frac{v_2^{\max} - v_2^{\min}}{2}.
\end{equation}

Figure~\ref{fig:KHI_plots} shows the evolution of $\langle B_p^2 \rangle$ and $\Delta v_y$ during the transition to turbulence for the magnetized KHI example using the fourth- ($N=3$) and eighth-order ($N=7$) DGSEM with Gauss and LGL nodes on different grid resolutions.

As in \cite{mignone2010high}, we observe that the poloidal magnetic energy grows faster in the transition to turbulence when the resolution is increased and when a scheme with a higher convergence order is selected, which indicates that the overall numerical dissipation is smaller.
Figures~\ref{fig:KHI_plots_N3_Bp}~and~\ref{fig:KHI_plots_N7_Bp} suggest that the Gauss-DGSEM allows for a less dissipative transition to turbulence that the LGL-DGSEM, being the difference greater at low polynomial degrees and coarse resolutions.

The evolution of the growth rate, $\Delta v_y$, follows a similar trend as the poloidal magnetic energy, which was also observed by \citet{mignone2010high}.
Here, however, the difference between the Gauss- and LGL-DGSEM gets smaller as the resolution and convergence order of the scheme are increased.

In Figure~\ref{fig:KHI_plots}, we also plot the \textit{best} result obtained by \citet{mignone2010high} with the highest resolution and convergence order that they tested ($256 \times 512$ degrees of freedom and a fifth-order finite difference scheme) as a reference.
In \cite{mignone2010high}, the same trend for $\langle B_p^2 \rangle$ and $\Delta v_y$ was observed when the resolution and convergence order was increased.

\begin{figure}[h!]
\centering
\begin{subfigure}[b]{0.38\linewidth}
\includegraphics[trim=0 0 0 0, clip,
width=\linewidth]{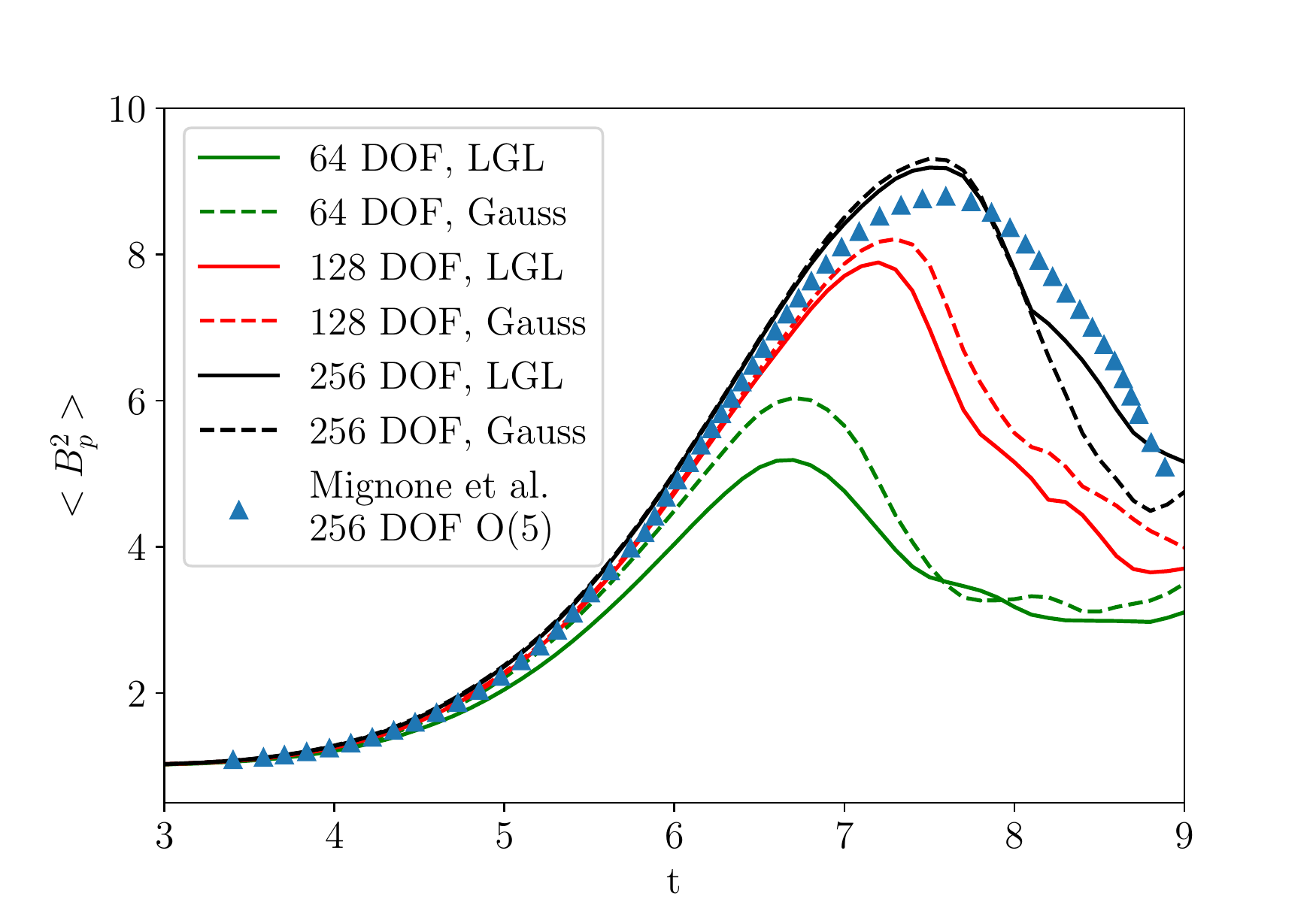}
\caption{Poloidal magnetic energy, $N=3$}
\label{fig:KHI_plots_N3_Bp}
\end{subfigure}
\begin{subfigure}[b]{0.38\linewidth}
\includegraphics[trim=0 0 0 0, clip,
width=\linewidth]{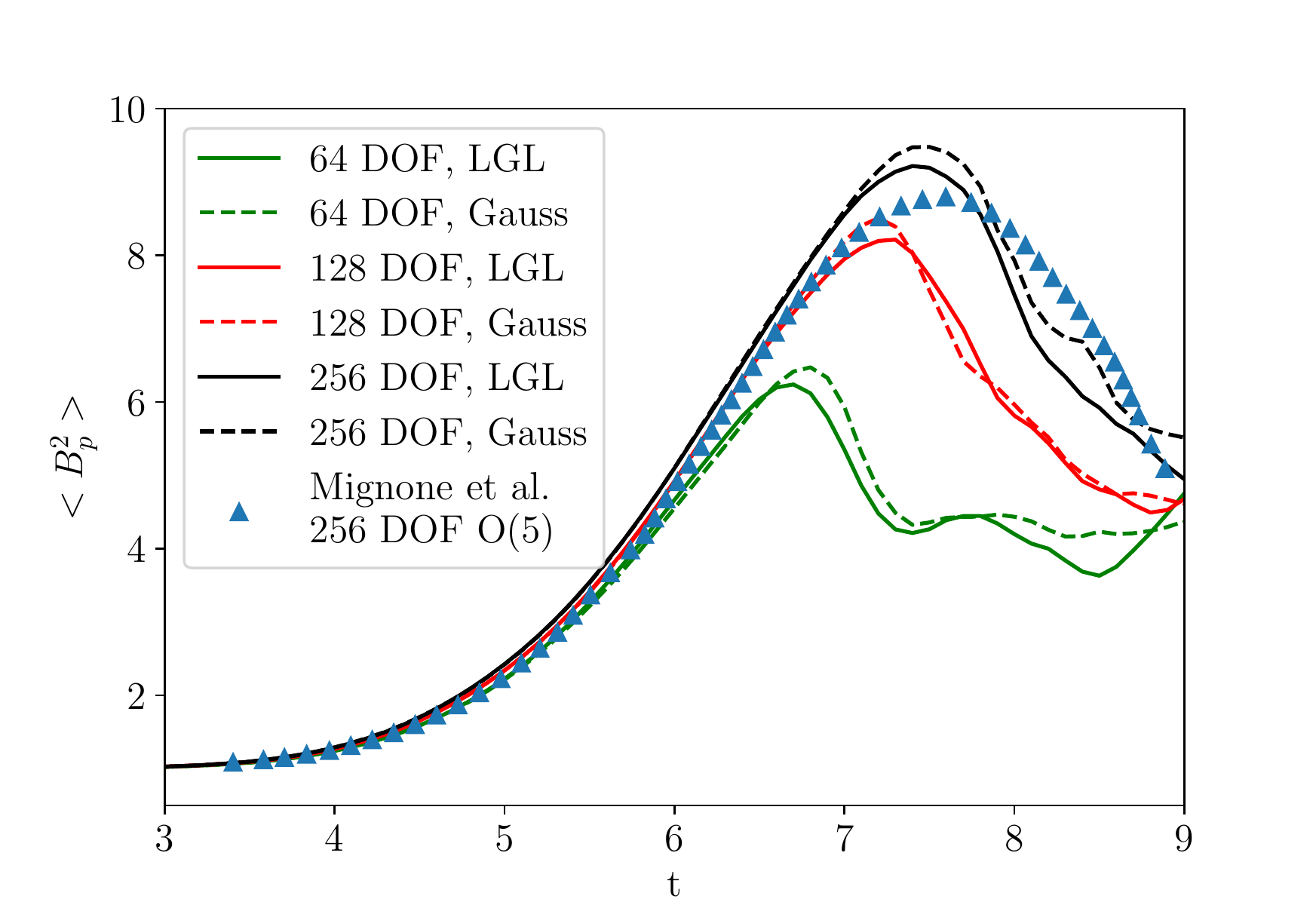}
\caption{Poloidal magnetic energy, $N=7$}
\label{fig:KHI_plots_N7_Bp}
\end{subfigure}
\begin{subfigure}[b]{0.38\linewidth}
\includegraphics[trim=0 0 0 0, clip,
width=\linewidth]{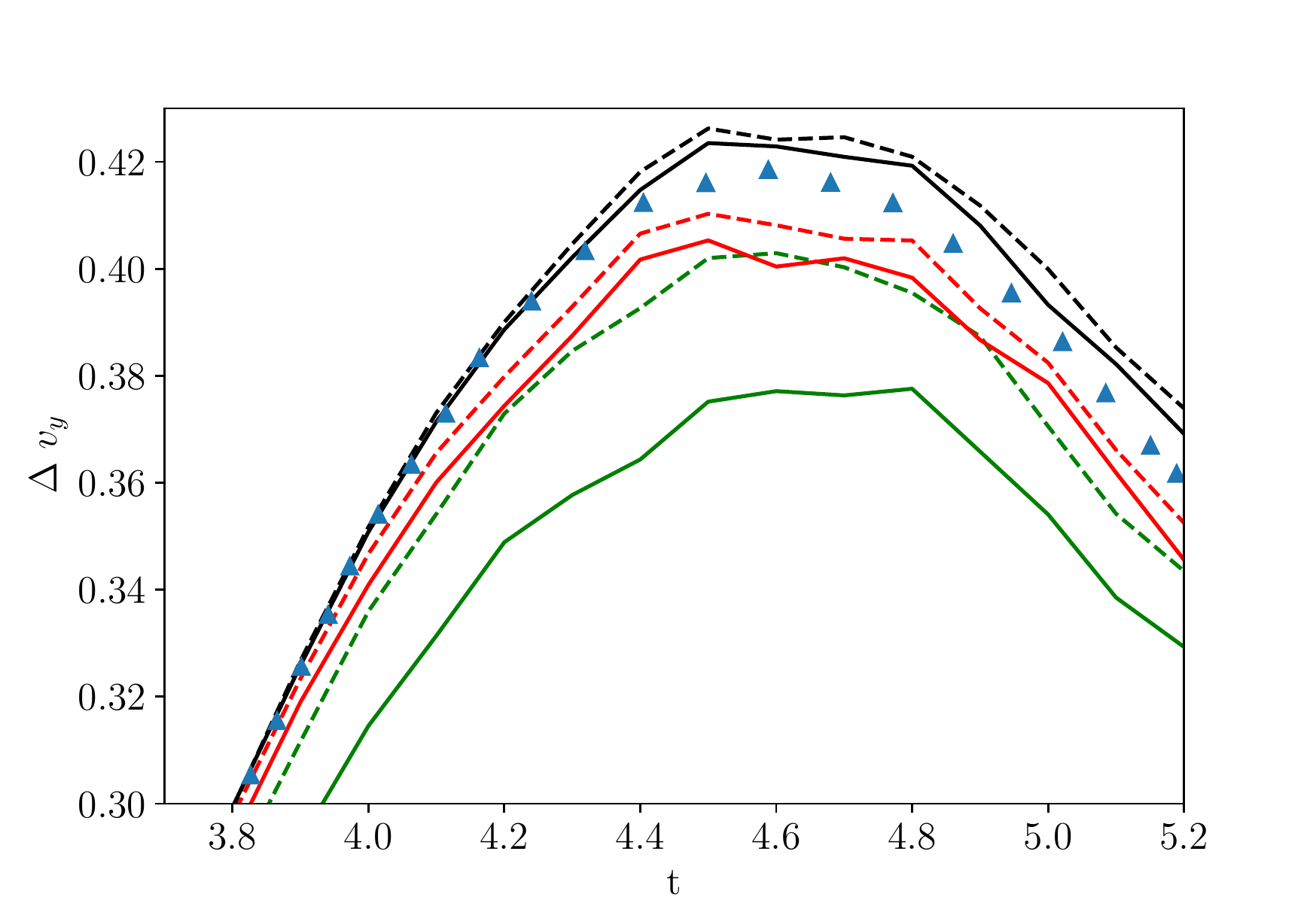}
\caption{Growth rate, $N=3$}
\end{subfigure}
\begin{subfigure}[b]{0.38\linewidth}
\includegraphics[trim=0 0 0 0, clip,
width=\linewidth]{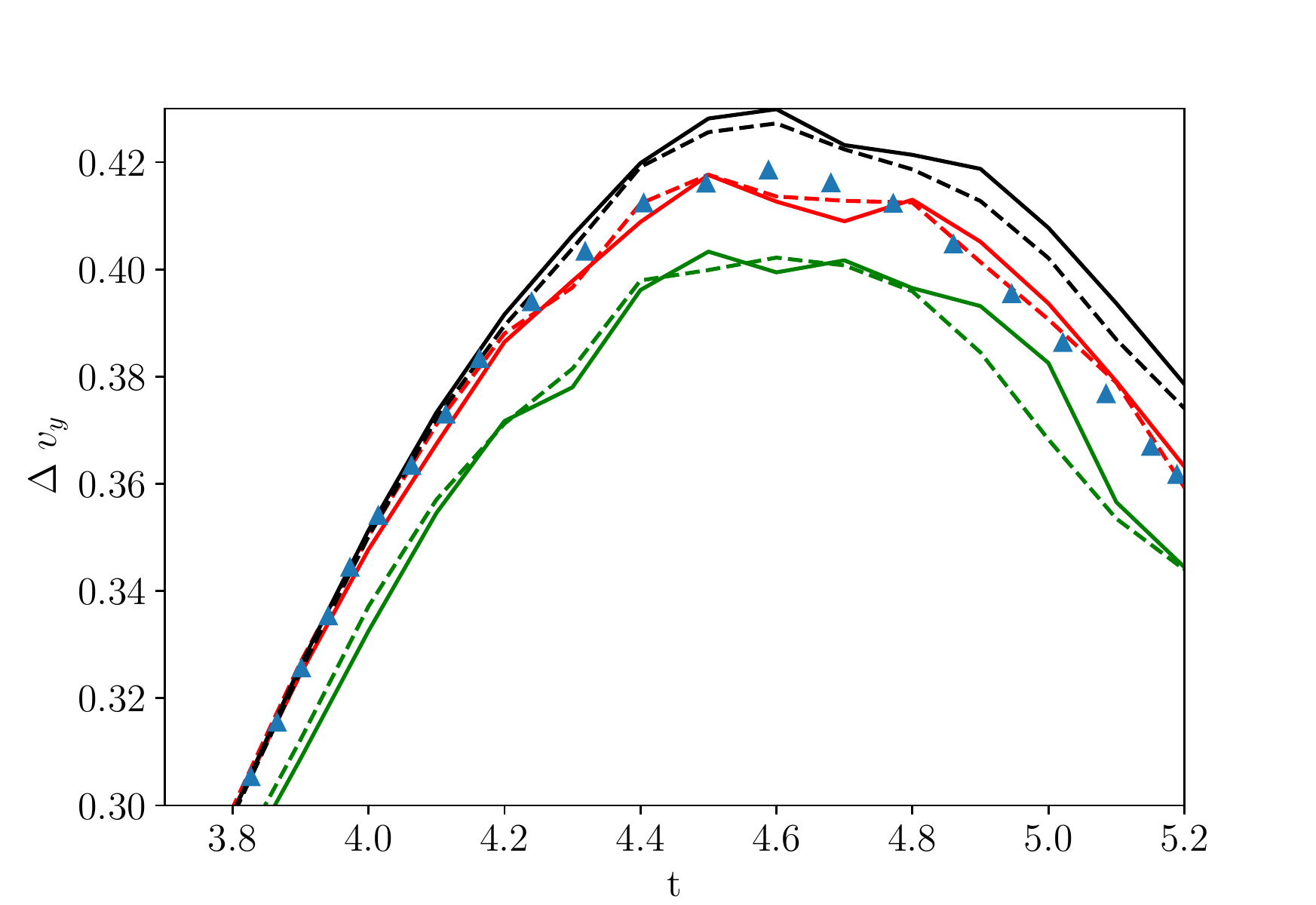}
\caption{Growth rate, $N=7$}
\end{subfigure}
\caption{
Normalized volume-integrated poloidal magnetic energy $\langle B_p^2 \rangle$ (top panels) and growth rate $\Delta v_y :=$ (bottom panels) for the Gauss- and LGL-DGSEM using different resolutions and polynomial degrees.
The number of degrees of freedom (DOF) corresponds to the $x$ direction.
As a reference, we plot in triangles the result obtained by \citet{mignone2010high} with the highest resolution and a fifth-order finite difference method.
}
\label{fig:KHI_plots}
\end{figure}

\section{Conclusions} \label{sec:Conclusions}

In this work, we have presented a provably entropy-stable Discontinuous Galerkin Spectral Element Method based on Gauss points (Gauss-DGSEM) for the GLM-MHD equations.
We have shown numerically that the new method provides formal high-order accuracy and entropy consistency on three-dimensional curvilinear meshes, and better accuracy than the entropy-stable LGL-DGSEM.

Moreover, we reformulated Bohm et al.'s \cite{Bohm2018} entropy-stable LGL-DGSEM discretization of the non-conservative terms of the GLM-MHD equations in a way that fits the Gauss and LGL variants, and proposed two different non-conservative terms that can be used to achieve entropy consistency with generalized SBP operators.

The methods detailed in this paper are implemented in the open-source framework FLUXO (\url{www.github.com/project-fluxo}).
\section*{Acknowledgments}

This work has received funding from the European Research Council through the ERC Starting Grant “An Exascale aware and Un-crashable Space-Time-Adaptive Discontinuous Spectral Element Solver for Non-Linear Conservation Laws” (Extreme), ERC grant agreement no. 714487 (Gregor J. Gassner and Andrés Rueda-Ramírez). 

Jesse Chan gratefully acknowledges support from the National Science Foundation under award DMS-CAREER-1943186.

We furthermore thank the Regional Computing Center of the University of Cologne (RRZK) for providing computing time on the High Performance Computing (HPC) system ODIN as well as support.

\printcredits

\bibliographystyle{model1-num-names}
%\bibliographystyle{cas-model2-names}

\bibliography{Biblio.bib}

\section*{Appendices}
\appendix

\section{Direct Carryover of the GLM-MHD Equations to Gauss-DGSEM } \label{sec:FirstAttemptGaussMHD}

The most straightforward way to carry over the non-conservative discretization of \citet{Bohm2018}, \eqref{eq:LGL_DGSEM_GLMMHD}, to the Gauss-DGSEM of \citet{chan2019efficient}, \eqref{eq:Gauss_DGSEM_cons}, is to generalize the non-conservative part of the volume integral using $\hat{\mat{B}}$ and use the traditional volume non-conservative term for the new terms that connect the inner degrees of freedom with the boundaries:
\small
\begin{empheq}{align} \label{eq:Gauss_DGSEM_GLMMHD_other} % [box=\fbox]
J& \omega_j \dot{\state{u}}_j 
+ \sum_{k=0}^N \left( \hat{S}_{jk} \state{f}^{*}_{(j,k)} + (Q_{jk} - \hat{B}_{jk}) {\Jan}^*_{(j,k)} \right)
\nonumber \\
% Left boundary terms
%%%%%%%%%%%%%%%%%%%%%
-& \ell_{j}(-1) \left[
   \state{f}^* \left( \state{u}_j , \tilde{\state{u}}_L \right)
 + {\Jan}^* \left( \state{u}_j , \tilde{\state{u}}_L \right)
 - \sum_{k=0}^N \ell_{k}(-1) \left( 
 	\state{f}^* \left( \tilde{\state{u}}_L , \state{u}_k \right)
   +{\Jan}^* \left( \tilde{\state{u}}_L , \state{u}_k \right) 
 \right)
 + \numfluxb{f}^a \left( \tilde{\state{u}}_L, \tilde{\state{u}}^{+}_L \right)
 + \numnonconsD{\Jan} \left( \tilde{\state{u}}_L, \tilde{\state{u}}^{+}_L \right)
 \right]
\nonumber\\
% Right boundary terms
%%%%%%%%%%%%%%%%%%%%%%
+& \ell_{j}(+1) \left[
 \state{f}^* \left( \state{u}_j , \tilde{\state{u}}_R \right)
 + {\Jan}^* \left( \state{u}_j , \tilde{\state{u}}_R \right)
 - \sum_{k=0}^N \ell_{k}(+1) \left(
 	 \state{f}^* \left( \state{u}_k , \tilde{\state{u}}_R \right)
 	 +{\Jan}^* \left( \tilde{\state{u}}_R , \state{u}_k \right) 
 \right)
 + \numfluxb{f}^a \left( \tilde{\state{u}}_R, \tilde{\state{u}}^{+}_R \right)
 + \numnonconsD{\Jan} \left( \tilde{\state{u}}_R, \tilde{\state{u}}^{+}_R \right)
\right]
= \state{0}.
\end{empheq}
\normalsize

\begin{remark}
It is plain to see that \eqref{eq:Gauss_DGSEM_GLMMHD_other} is algebraically equivalent to the entropy-stable LGL-DGSEM discretization of the GLM-MHD system of \citet{Bohm2018}, \eqref{eq:LGL_DGSEM_GLMMHD}, when LGL nodes are used.
Hence, it provides entropy conservation/stability for that choice of nodes.
However, when Gauss nodes are used, \eqref{eq:Gauss_DGSEM_GLMMHD_other} differs from \eqref{eq:Gauss_DGSEM_GLMMHD}, which can be easily seen when we subtract the two equations.

When we subtract the volume terms of \eqref{eq:Gauss_DGSEM_GLMMHD_other} from \eqref{eq:Gauss_DGSEM_GLMMHD} we obtain
\begin{align} \label{eq:volTerms}
\state{V} &= \sum_{k=0}^N \left( 
    \hat{S}_{jk} \numnonconsD{\Jan}_{(j,k)}
    - \left( Q_{jk} - \hat{B}_{jk} \right) \Jan^*_{(j,k)}
    \right)
    \nonumber\\
\text{(def. of $\hat{\mat{S}}$, \eqref{eq:Shat_mat} \& identity \eqref{eq:nonConsIdentity})} \qquad
    &= \sum_{k=0}^N \left( 
    \left( 2Q_{jk} - \hat{B}_{jk} \right) \numnonconsD{\Jan}_{(j,k)} 
    - \left( Q_{jk} - \hat{B}_{jk} \right)  \left( 2\numnonconsD{\Jan}_{(j,k)} - \Jan_j \right)
    \right)
    \nonumber\\
\text{(SBP property \eqref{eq:genSBPprop}, and simp.)} \qquad
    &=
    \sum_{k=0}^N 
    \hat{B}_{jk} \left(
    \numnonconsD{\Jan}_{(j,k)} - \Jan_j
    \right).
\end{align}
Clearly, $\state{V}=\state{0}$ when LGL nodes are used due to the consistency property of $\numnonconsD{\Jan}$ and because $\hat{\mat{B}} \eqLGL \mat{B}$ is a diagonal matrix.
However, $\hat{\mat{B}}$ is dense for Gauss nodes, and so $\state{V} \ne \state{0}$ in general.

When we subtract the rest of the terms of \eqref{eq:Gauss_DGSEM_GLMMHD_other} from \eqref{eq:Gauss_DGSEM_GLMMHD} we obtain
\begin{align} \label{eq:otherTerms}
\state{F} =
   &- \ell_{j}(-1) \left[
   \numnonconsD{\Jan} \left( \state{u}_j , \tilde{\state{u}}_L \right)
   - \Jan^* \left( \state{u}_j , \tilde{\state{u}}_L \right)
 - \sum_{k=0}^N \ell_{k}(-1) \left( 
 	\numnonconsD{\Jan} \left( \tilde{\state{u}}_L , \state{u}_k 
 	\right) 
 	-\Jan^* \left( \tilde{\state{u}}_L , \state{u}_k 
 	\right) 
 \right)
 \right]
 \nonumber\\
 &+
 \ell_{j}(+1) \left[
   \numnonconsD{\Jan} \left( \state{u}_j , \tilde{\state{u}}_R \right)
   - \Jan^* \left( \state{u}_j , \tilde{\state{u}}_R \right)
 - \sum_{k=0}^N \ell_{k}(+1) \left( 
 	\numnonconsD{\Jan} \left( \tilde{\state{u}}_R , \state{u}_k 
 	\right) 
 	-\Jan^* \left( \tilde{\state{u}}_R , \state{u}_k 
 	\right) 
 \right)
 \right]
  \nonumber\\
\text{(identity \eqref{eq:nonConsIdentity})} \qquad
    =&- \ell_{j}(-1) \left[
   \Jan_j 
   - \numnonconsD{\Jan} \left( \state{u}_j , \tilde{\state{u}}_L \right)
 - \sum_{k=0}^N \ell_{k}(-1) \left( 
 	\Jan \left( \tilde{\state{u}}_L \right) 
 	- \numnonconsD{\Jan} \left( \tilde{\state{u}}_L , \state{u}_k 
 	\right) 
 \right)
 \right]
 \nonumber\\
 &+
 \ell_{j}(+1) \left[
   \Jan_j
   -\numnonconsD{\Jan} \left( \state{u}_j , \tilde{\state{u}}_R \right)
 - \sum_{k=0}^N \ell_{k}(+1) \left( 
 	\Jan \left( \tilde{\state{u}}_R \right) 
 	-\numnonconsD{\Jan} \left( \tilde{\state{u}}_R , \state{u}_k 
 	\right) 
 \right)
 \right],
\end{align}
which is again zero for LGL, but not necessarily for Gauss.

When we sum the contributions of \eqref{eq:volTerms} and \eqref{eq:otherTerms}, some terms cancel out due to the generalized SBP properties, \eqref{eq:genSBPprop}.
We obtain
\begin{align*}
    \state{V} + \state{F} =&+ \ell_{j}(-1) \left[
    \numnonconsD{\Jan} \left( \state{u}_j , \tilde{\state{u}}_L \right)
 + \sum_{k=0}^N \ell_{k}(-1) \left( 
 	\Jan \left( \tilde{\state{u}}_L \right) 
 	- \numnonconsD{\Jan} \left( \tilde{\state{u}}_L , \state{u}_k 
 	\right) 
 \right)
 \right]
 \nonumber\\
 &-
 \ell_{j}(+1) \left[
   \numnonconsD{\Jan} \left( \state{u}_j , \tilde{\state{u}}_R \right)
 + \sum_{k=0}^N \ell_{k}(+1) \left( 
 	\Jan \left( \tilde{\state{u}}_R \right) 
 	-\numnonconsD{\Jan} \left( \tilde{\state{u}}_R , \state{u}_k 
 	\right) 
 \right)
 \right]
 +
 \sum_{k=0}^N 
    \hat{B}_{jk} \numnonconsD{\Jan}_{(j,k)},
\end{align*}
which is non-zero in the general case for Gauss nodes.
\end{remark}

Due to the difference between the two Gauss-DGSEM discretizations, \eqref{eq:Gauss_DGSEM_GLMMHD_other} fails to provide entropy conservation/stability, as is shown in the next section.

\subsection{Entropy Balance}
To obtain the entropy balance of \eqref{eq:Gauss_DGSEM_GLMMHD_other}, we follow a very similar manipulation of terms as in the proof of Lemma \ref{lemma:Entropy_Gauss_DGSEM_GLMMHD}.
For the volume terms we obtain
\begin{align} \label{eq:EntBalance_Gauss_GLMMHD_other_vol}
(a)
	=
%	&
	 \sum_{j=0}^{N} \entVar^T_j \sum_{k=0}^N \left( \hat{S}_{jk} \state{f}^{*}_{(j,k)} + (Q_{jk} - \hat{B}_{jk}) {\Jan}^*_{(j,k)} \right) 
%\\
%\text{(skew-symmetry of $\hat{\mat{S}}$)} \qquad 
	=
%	&
	 {\Psi}_L - {\Psi}_R - \frac{1}{2} \sum_{j,k=0}^N \hat{S}_{jk} r_{(j,k)}
	- \sum_{j,k=0}^N \entVar_j^T \left[ \hat{B}_{jk} \left( \numnonconsD{\Jan}_{(j,k)} - \Jan_j \right) \right].
\end{align}
It is plain to see that the last term of \eqref{eq:EntBalance_Gauss_GLMMHD_other_vol} vanishes for LGL discretizations due to the diagonal boundary matrix and the consistency property of the surface numerical non-conservative term.
However, for Gauss discretizations, the last term is not guaranteed to vanish.

For the \textit{new} terms that connect all degrees of freedom with the left boundary we obtain
\begin{align} \label{eq:EntBalance_Gauss_GLMMHD_other_new}
(b)
	=& \sum_{j=0}^N \entVar_j^T \ell_{j}(-1) 
	\left[
   \state{f}^* \left( \state{u}_j , \tilde{\state{u}}_L \right)
 + {\Jan}^* \left( \state{u}_j , \tilde{\state{u}}_L \right)
 - \sum_{k=0}^N \ell_{k}(-1) \left( 
 	\state{f}^* \left( \tilde{\state{u}}_L , \state{u}_k \right)
   +{\Jan}^* \left( \tilde{\state{u}}_L , \state{u}_k \right) 
 \right)
 + \numfluxb{f}^a \left( \tilde{\state{u}}_L, \tilde{\state{u}}^{+}_L \right)
 \right]
\nonumber \\
	=& {\Psi}_L - \tilde{\Psi}_L - \sum_{j=0}^N \ell_{j}(-1) \tilde{r}_{(j,L)}
	+ \sum_{j=0}^N \ell_j (-1)
	\left[
		\entVar_j^T \left( \numnonconsD{\Jan} \left( \state{u}_j, \tilde{\state{u}}_L \right) - \Jan_j \right)
	 -  \tilde{\entVar}_L^T \left( \numnonconsD{\Jan} \left( \tilde{\state{u}}_L, \state{u}_j \right) - \Jan \left( \tilde{\state{u}}_L \right) \right)
	\right].
\end{align}
Again, the last sum of \eqref{eq:EntBalance_Gauss_GLMMHD_other_new} vanishes for LGL discretizations due to the existence of nodes at the boundaries, but not for Gauss.

The entropy balance of the left boundary term is exactly the same as in the proof of Lemma \ref{lemma:Entropy_Gauss_DGSEM_GLMMHD},
\begin{equation}
(c) = \sum_{j=0}^N \entVar_j^T \ell_{j}(-1)
\left[
 \numfluxb{f}^a \left( \tilde{\state{u}}_L, \tilde{\state{u}}^{+}_L \right)
+\numnonconsD{\Jan} \left( \tilde{\state{u}}_L, \tilde{\state{u}}^{+}_L \right)
\right]
= \tilde{\entVar}_L^T 
\left[
 \numfluxb{f}^a \left( \tilde{\state{u}}_L, \tilde{\state{u}}^{+}_L \right)
+\numnonconsD{\Jan} \left( \tilde{\state{u}}_L, \tilde{\state{u}}^{+}_L \right)
\right].
\end{equation}

As in the previous proofs, terms $(d)$ and $(e)$ are analyzed in the same form as terms $(b)$ and $(c)$.
Gathering all contributions and manipulating the boundary terms we obtain
\begin{align}
\sum_{j=0}^N \omega_j J \dot{S}_j  =& -(a)+(b)+(c)-(d)-(e) 
\nonumber \\
=& 
\numflux{f}^S \left( \tilde{\state{u}}_L, \tilde{\state{u}}^{+}_L \right)
-\numflux{f}^S \left( \tilde{\state{u}}_R, \tilde{\state{u}}^{+}_R \right) 
+ \frac{1}{2} \left[
	\hat{r} \left( \tilde{\state{u}}_L, \tilde{\state{u}}^{+}_L \right)
  + \hat{r} \left( \tilde{\state{u}}_R, \tilde{\state{u}}^{+}_R \right)
\right]
\nonumber \\
&+ \frac{1}{2} \sum_{j,k=0}^N \hat{S}_{jk} r_{(j,k)}
+ \sum_{j=0}^N \left( \ell_{j}(+1) \tilde{r}_{(j,R)} - \ell_{j}(-1) \tilde{r}_{(j,L)} \right)
+ \dot{S}_{\text{Gauss}}.
\end{align}

The additional production term, $\dot{S}_{\text{Gauss}}$, gathers the entropy production of the discretization of the non-conservative terms in \eqref{eq:Gauss_DGSEM_GLMMHD_other},
\begin{align}
\dot{S}_{\text{Gauss}} = 
% Extra production terms on the left
%%%%%%%%%%%%%%%%%%%%%%%%%%%%%%%%%%%%
&
- \sum_{j=0}^N \ell_j (-1)
	\left[
		\entVar_j^T \numnonconsD{\Jan} \left( \state{u}_j, \tilde{\state{u}}_L \right) 
	 -  \tilde{\entVar}_L^T \left( \numnonconsD{\Jan} \left( \tilde{\state{u}}_L, \state{u}_j \right) - \Jan \left( \tilde{\state{u}}_L \right) \right)
	\right]
\nonumber\\
% Extra production terms on the left
%%%%%%%%%%%%%%%%%%%%%%%%%%%%%%%%%%%%
&
+ \sum_{j=0}^N \ell_j (+1)
	\left[
		\entVar_j^T \numnonconsD{\Jan} \left( \state{u}_j, \tilde{\state{u}}_R \right) 
	 -  \tilde{\entVar}_R^T \left( \numnonconsD{\Jan} \left( \tilde{\state{u}}_R, \state{u}_j \right) - \Jan \left( \tilde{\state{u}}_R \right) \right)
	\right]
% Extra production term from volume integral
%%%%%%%%%%%%%%%%%%%%%%%%%%%%%%%%%%%%%%%%%%%%
+ \sum_{j,k=0}^N \entVar_j^T \hat{B}_{jk} \numnonconsD{\Jan}_{(j,k)},
\end{align}
where some of the additional non-zero terms of $(a)$ cancel out with additional non-zero terms of $(b)$ and $(d)$.
Clearly, $\dot{S}_{\text{Gauss}}$ vanishes in LGL discretizations, but not necessarily in Gauss discretizations.

\section{Hybridized SBP formulation}

In this section, we describe the formulation of the entropy stable Gauss collocation scheme in one dimension using ``hybridized'' summation-by-parts operators \cite{chen2020review}. Formulations using hybridized operators are similar to formulations using traditional SBP operators, but are equivalent to the formulation (\ref{eq:Gauss_DGSEM_cons}) which introduce additional surface terms. 

First, we define the matrix $\mat{V}_h$ 
\[
\mat{V}_h = \begin{bmatrix}
\mat{I}\\
\mat{V}_f
\end{bmatrix} \in \mathbb{R}^{N+3, N+1}
\]
Multiplication by $\mat{V}_h$ maps nodal values at Gauss nodes to the vector containing nodal values at \textit{both} Gauss nodes and face nodes. Next, we introduce the hybridized SBP operator
\[
\mat{Q}_h = \frac{1}{2}\begin{bmatrix}
\mat{S} & \mat{V}_f^T\mat{B}\\
-\mat{B}\,\mat{V}_f & \mat{B}
\end{bmatrix}
\]
where we have used the skew-symmetric matrix $\mat{\hat{S}} = 2\mat{Q}-\mat{V}_f^T\mat{B}\,\mat{V}_f = \mat{Q}-\mat{Q}^T$. The operator $\mat{Q}_h$ satisfies a generalized SBP property and can be used to construct degree $N$ approximations to derivatives \cite{chan2019efficient}. 

We now introduce the skew-symmetric hybridized matrix $\mat{\hat{S}}_h = \mat{Q}_h - \mat{Q}_h^T$. Then, (\ref{eq:Gauss_DGSEM_GLMMHD}) can equivalently be formulated as
\begin{empheq}[box=\fbox]{align}
J \omega_j \dot{\state{u}}_j 
+ \sum_{k=0}^{N+2} 2 {V}_{h, kj} & \state{r}_{h,k}
+ \sum_{k=0}^N \left({V}_f\right)_{kj} \state{s}_k
= \state{0},  \label{eq:Gauss_hybridized_noncons}
\end{empheq}
where $\state{r}_{h,k}, \state{s}_k$ are volume and surface contributions computed in terms of the hybridized SBP operator 
\begin{empheq}[box=\fbox]{align}
\state{r}_{h,j} = \sum_{k=0}^{N+2} {\hat{S}}_{h, jk} \left(\state{f}^{*}_{(j,k)} + \numnonconsD{\Jan}_{(j,k)} \right), \qquad
\state{s}_j = \delta_{jN}
\left( \numfluxb{f}^a_{(N,R)} + \numnonconsD{\Jan}_{(N,R)} \right)
-\delta_{j0}\left( \numfluxb{f}^a_{(0,L)} + \numnonconsD{\Jan}_{(0,L)}  \right).
\end{empheq}
Note that, in contrast to (\ref{eq:Gauss_DGSEM_GLMMHD}), the formulation (\ref{eq:Gauss_hybridized_noncons}) combines the additional local surface terms into the definition of the volume operator. The resulting contributions $\state{r}_h$ and $\state{s}$ more closely resemble the simpler LGL formulation (\ref{eq:LGL_DGSEM_GLMMHD_cleanedUp}).

To show equivalence, we first note that since the entries of $(\mat{V}_f)_{1j}$ and $(\mat{V}_f)_{2j}$ are $\ell_j(-1)$ and $\ell_j(+1)$, the surface term $\sum_{k=0}^N \left({V}_f\right)_{kj} \state{s}_k$ can be rewritten as
\begin{equation}
-\ell_j(-1) \left( \numfluxb{f}^a_{(0,L)} + \numnonconsD{\Jan}_{(0,L)}  \right)
+ \ell_j(+1)\left( \numfluxb{f}^a_{(N,R)} + \numnonconsD{\Jan}_{(N,R)} \right).
\label{eq:hybridized_1}
\end{equation}
Next, we split the since the contribution $\sum_{k=0}^{N+2} 2 {V}_{h, kj} \state{r}_{h,k}$ into two parts. Note this can be written equivalently as a matrix-vector product
\begin{equation*}
\mat{V}_h^T \state{r}_h = \begin{bmatrix}
\mat{I} \\
\mat{V}_f
\end{bmatrix}^T \begin{bmatrix}
\state{r}\\
\state{r}_f
\end{bmatrix} = \state{r} + \mat{V}_f^T\state{r}_f.
\end{equation*}
upper left block of $\mat{\hat{S}}_h$ is the skew symmetric part of the generalized SBP operator $\mat{\hat{S}}$, one can write the entries of the contribution $\state{r}_j$ as
\begin{equation}
\sum_{k=0}^N \hat{S}_{jk} \left(\state{f}^{*}_{(j,k)} + \numnonconsD{\Jan}_{(j,k)} \right) - \ell_j(-1)\left(\state{f}^{*}_{(j,k)} + \numnonconsD{\Jan}_{(j,k)}\right) + \ell_j(+1)\left(\state{f}^{*}_{(j,k)} + \numnonconsD{\Jan}_{(j,k)}\right).
\label{eq:hybridized_2}
\end{equation}
Finally, we consider the contribution $\mat{V}_f^T\state{r}_f$. The term $\state{r}_f$ is equivalent to 
\begin{align*}
\state{r}_{f,j} &= \sum_{k=0}^N -\delta_{j0}\ell_k(-1)\left(\state{f}^{*}_{(L,k)} + \numnonconsD{\Jan}_{(L,k)}\right) + \delta_{jN}\ell_k(+1)\left(\state{f}^{*}_{(R,k)} + \numnonconsD{\Jan}_{(R,k)}\right) \\
&=\sum_{k=0}^N -\delta_{j0}\ell_k(-1)\left(\state{f}^{*}(\tilde{\state{u}}_L,\state{u}_k) + \numnonconsD{\Jan}(\tilde{\state{u}}_L,\state{u}_k)\right) + \delta_{jN}\ell_k(+1)\left(\state{f}^{*}(\tilde{\state{u}}_R,\state{u}_k) + \numnonconsD{\Jan}(\tilde{\state{u}}_R,\state{u}_k)\right) 
\end{align*}
Multiplication by $\mat{V}_f^T$ then yields 
\begin{equation}
\mat{V}_f^T\state{r}_{f} = -\ell_j(-1)\left(\sum_{k=0}^N \ell_k(-1)\left(\state{f}^{*}(\tilde{\state{u}}_L,\state{u}_k) +
\numnonconsD{\Jan}(\tilde{\state{u}}_L,\state{u}_k)\right)\right) + 
\ell_j(+1) \left(\sum_{k=0}^N\ell_k(+1)\left(\state{f}^{*}(\tilde{\state{u}}_R,\state{u}_k) + \numnonconsD{\Jan}(\tilde{\state{u}}_R, \state{u}_k)\right)\right).
\label{eq:hybridized_3}
\end{equation}
Combining (\ref{eq:hybridized_1}), (\ref{eq:hybridized_2}), and (\ref{eq:hybridized_3}) then recovers the Gauss-DGSEM formulation for GLM-MHD (\ref{eq:Gauss_DGSEM_GLMMHD}). 


%Author biography without author photo.
%Author biography. Author biography. Author biography.
%Author biography. Author biography. Author biography.
%Author biography. Author biography. Author biography.
%Author biography. Author biography. Author biography.
%Author biography. Author biography. Author biography.
%Author biography. Author biography. Author biography.
%Author biography. Author biography. Author biography.
%Author biography. Author biography. Author biography.
%Author biography. Author biography. Author biography.
%\bio{figs/pic1}
%Author biography with author photo.
%Author biography. Author biography. Author biography.
%Author biography. Author biography. Author biography.
%Author biography. Author biography. Author biography.
%Author biography. Author biography. Author biography.
%Author biography. Author biography. Author biography.
%Author biography. Author biography. Author biography.
%Author biography. Author biography. Author biography.
%Author biography. Author biography. Author biography.
%Author biography. Author biography. Author biography.
%\bio{figs/pic1}
%Author biography with author photo.
%Author biography. Author biography. Author biography.
%Author biography. Author biography. Author biography.
%Author biography. Author biography. Author biography.
%Author biography. Author biography. Author biography.

\end{document}